\documentclass[11pt,twoside,reqno,centertags]{amsart}
\usepackage{amsmath,amsthm,amsfonts,amssymb,xcolor}
\usepackage[pdftex]{pict2e}
\advance\textheight by 6truemm
\newtheorem{Theorem}{Theorem}

\newtheorem{theorem}[Theorem]{Theorem}

\newtheorem{proposition}[Theorem]{Proposition}

\newtheorem{remark}[Theorem]{Remark}
\newtheorem{definition}[Theorem]{Definition}
\newtheorem{example}[Theorem]{Example}
\newcommand{\R}{{\mathbb R}}
\newcommand{\N}{{\mathbb N}}
\newcommand{\eps}{\varepsilon}
\newcommand{\vp}{\varphi}
\newcommand{\refcol}{\color[rgb]{0,0.35,0}}
\newcommand{\ecol}{\color[rgb]{0.5,0,0.5}}
\newcommand{\blue}{\color[rgb]{0,0,0.7}}
\newcommand{\red}{\color[rgb]{1,0,0}}
\newcommand{\black}{\color[rgb]{0,0,0}}
\def\c#1{{\mathcal #1}}
\def\ID{I^{\c{D}}}
\def\IN{I^{\c{N}}}
\def\IDr{\tilde I^{\c{D}}}
\def\INr{\tilde I^{\c{N}}}
\def\m#1{{\mathfrak #1}}
\font\basic=cmr10

\setbox1=\hbox{\basic\strut Department of Applied Mathematics and Statistics, Comenius University} 
\setbox2=\hbox{\basic\strut Mlynsk\'a dolina, 84248 Bratislava, Slovakia}
\setbox3=\hbox{\basic\strut email: {\tt quittner@fmph.uniba.sk}}

\begin{document}

\title[Necessary and sufficient conditions]{Necessary and sufficient conditions
    \\ for one-dimensional variational problems 
    \\ with applications to elasticity}

\author{Pavol Quittner} 
\address{Department of Applied Mathematics and Statistics, Comenius
University, Mlynsk\'a dolina, 84248 Bratislava, Slovakia} 
\email{quittner@fmph.uniba.sk}
\keywords{Minimizer; natural boundary conditions; conjugate points; 
   field of extremals; elastica}
\subjclass{49K05, 74K10, 74G65, 34B15}{}

\date{}

\begin{abstract}
This paper deals with necessary and sufficient conditions
for weak and strong minimizers of functionals 
$\Phi(u)=\int_a^b f(x,u(x),u'(x))\,dx$, where $u\in C^1([a,b],{\mathbb R}^N)$.
We first derive conditions which are simpler than the known ones,
and then apply them to several particular problems,
including stability problems in the elasticity theory.
In particular, we solve some open problems in [A. Majumdar, A. Raisch:
Stability of twisted rods, helices and buckling solutions in three
dimensions, Nonlinearity 27 (2014), 2841--2867]
by finding optimal conditions for the stability of a naturally straight 
Kirchhoff rod under various types of endpoint constraints.
\end{abstract}

\maketitle


\section{Introduction}
\label{intro}

This paper deals with necessary and sufficient conditions for 
local minimizers of one-di\-men\-sion\-al variational problems
for vector-valued functions. We consider the functional 
\begin{equation} \label{Phi}
 \Phi:C^1([a,b],\R^N)\to\R:u\mapsto\int_a^b f(x,u(x),u'(x))\,dx, 
\end{equation}
where $-\infty<a<b<\infty$, $u=(u_1,u_2,\dots,u_N)$, and the Lagrangian%
\footnote{As in \cite[pp.~11--12]{GH},   
by $u$ we denote both the functions $[a,b]\to\R^N$ and the independent variable in $\R^N$,
and by $p$ we denote the last argument of $f$;
see also similar notation $L(t,x(t),\dot x(t))$ vs.~$L(t,x,v)$ in \cite{Zei}, for example.}
$$f:[a,b]\times\R^N\times\R^N\to\R:(x,u,p)\mapsto f(x,u,p)$$
is sufficiently smooth
($f\in C^3$ or $f\in C^2$).
We also fix a function $u^0\in C^1([a,b],\R^N)$ and
(possibly empty) subsets $\ID_a,\ID_b$ of the index set $I:=\{1,2,\dots,N\}$,
and we look for conditions guaranteeing that $u^0$ is a local minimizer of $\Phi$
in the set
\begin{equation} \label{M}
 \c{M}:=\{u\in C^1([a,b],\R^N): (u_i-u_i^0)(a)=0 \hbox{ for }i\in \ID_a,\
(u_i-u_i^0)(b)=0 \hbox{ for }i\in \ID_b\}.
\end{equation}
This means that at $x=a$ we consider Dirichlet endpoint constraints for the components $u_i$ with $i\in\ID_a$,
while the endpoints of the remaining components $u_j$ with $j\in I\setminus\ID_a$ are free;
similarly for $x=b$.
It is well known (see Proposition~\ref{prop-NBC}) 
that if $u^0$ is a local minimizer of this problem, then $u^0$ has to satisfy
the natural boundary conditions
$$
\frac{\partial f}{\partial p_j}(a,u^0(a),(u^0)'(a))=0\ \hbox{ for }\ j\notin \ID_a\quad\hbox{ and }\quad
\frac{\partial f}{\partial p_j}(b,u^0(b),(u^0)'(b))=0\ \hbox{ for }\ j\notin \ID_b.
$$

We say that $u^0$ is a weak (or strong, resp.) local minimizer
if there exists $\eps>0$ such that $\Phi(u^0)\leq\Phi(u)$ for any
$u\in\c{M}$ satisfying $\|u-u^0\|_{C^1}<\eps$ (or $\|u-u^0\|_{C}<\eps$, resp.),
where $\|\cdot\|_{C^1}$ and $\|\cdot\|_C$ are the usual norms in $C^1$ and $C$,
respectively (see Definition~\ref{def-ws} and the subsequent comments for more
details).
If $u^0$ is a steady state of a mechanical system with potential energy $\Phi$,
and $u^0$ is a weak (or strong) local minimizer of $\Phi$,
then $u^0$ is stable with respect to perturbations which are small
in $C^1$ (or $C$), respectively. On the other hand, if $u^0$ is not a minimizer,
then $u^0$ is unstable.

If $\ID_a=\ID_b=I$, i.e.~if one considers the 
Dirichlet endpoint constraints for all components and both ends, then necessary
and sufficient conditions for $u^0$ to be a minimizer
belong to the classical results in the calculus of variations, see \cite{GF,Ces,GH}, for example.
They are based on the Jacobi theory (conjugate points)
or the Weierstrass theory (field of extremals and excess function).
In the general case such conditions are also known
(see \cite{Zei, Zez93} and the references therein, and cf.~also \cite{Zez97});
however, they use the notion of a \emph{coupled point}
which is more complicated than the classical notion of a \emph{conjugate point}.
This might be the reason why --- as far as the author is aware --- 
that general theory has not yet been
applied in the elasticity theory, for example.
In the scalar case,
another approach to problems with variable endpoints 
(and a special class of Lagrangians) can be found in \cite{Man}
but the conditions there are even more complicated than those in \cite{Zei, Zez93}.
Reference \cite{Man} has been cited by several papers dealing with problems
in the elasticity theory:
Some of those papers use the complicated theory in \cite{Man}
for scalar problems with special Lagrangians  
(see \cite{LG}, for example), some use various ad-hoc estimates
to obtain at least partial results in the vector-valued case
(when the theory in \cite{Man} does not seem to apply,
see \cite{MR}, for example)
and some refrain from considering variable endpoints
because of the complexity of the theory in \cite{Man},
see \cite{BB}, for example, where the authors write:
``\dots the application of the conjugate point test with nonclamped 
ends is a delicate issue \dots''.
Difficulties arising in a scalar problem with variable endpoints
have also been analyzed in \cite{OP}, for example.

The main purpose of this paper is to derive simple conditions 
for $u^0$ to be a minimizer, and to show how they can be applied 
to particular problems. 

In Section~\ref{sec-Jacobi} we derive necessary and sufficient conditions 
for weak minimizers by modifying the Jacobi theory (see Theorem~\ref{thm-Jacobi}
and also Remark~\ref{rem-ZZ} for the comparison of our conditions
with those in \cite{Zei, Zez93}).
In Section~\ref{sec-twist} we use the results from Section~\ref{sec-Jacobi}
to find optimal conditions for the stability of a naturally straight Kirchhoff rod
under various types of endpoint constraints.
The reasons for this particular application are the following:

\vskip2mm
$\bullet$ We show that our general results can easily be applied to 
vector-valued problems in the elasticity theory.

$\bullet$ We solve some open problems (and correct an erroneous result) in \cite{MR}.
 
$\bullet$ We show how the choice of endpoint constraints influences the stability of the rod.

\vskip2mm
In Section~\ref{sec-field} we use the Weierstrass theory
to derive conditions for weak, strong and global
minimizers, see Theorem~\ref{thm-field}.
In this case we restrict our applications in Section~\ref{sec-se} to the scalar case
$N=1$.
The reason for this restriction is the following:
If $N=1$ and the Lagragian $f$ is independent of its first variable $x$,
then the phase plane analysis of the corresponding Du Bois-Reymond equation  
yields a very simple and efficient way
to prove (or disprove) the existence of a suitable field of extremals;
hence it is sufficient to verify the nonnegativity of the excess function
in order to check our conditions.
In particular, this approach does not require the verification of
sufficient conditions based on the Jacobi theory 
and it can be used even if we do not know an explicit formula for $u^0$.
In Section~\ref{sec-se} we first determine the stability
of a planar weightless inextensible and unshearable rod (see Example~\ref{ex-LG}).
This problem has already been analyzed in \cite{LG, B20}, for example, 
but our analysis is simpler than that in \cite{LG} and more complete than that in \cite{B20}.
The notions of weak and strong minimizers are equivalent for functionals $\Phi$ 
in Section~\ref{sec-twist} and Example~\ref{ex-LG}
(see Remark~\ref{rem-twist}(vi) and Proposition~\ref{prop-ws}, respectively).
To illustrate various interesting features of minimizers in a more general case
and demonstrate the applicability of our theory,
in Example~\ref{ex-well} we consider Lagrangians of the form
$f(u,p)=u^2+g(p)$, where $g$ is a double-well function.
In particular, the corresponding functional can possess both
strong (even global) minimizers and minimizers which are weak but not strong.

Some of our results in the scalar case $N=1$ have been obtained 
in the Master thesis \cite{Bed}.

\section{Preliminaries} 
\label{sec-prelim}

Throughout this paper we will use the symbols
$\Phi,f,u^0,a,b,N,I,\ID_a$ and $\ID_b$ introduced in the Introduction.
The partial derivatives of $f$ will be denoted by
$f_x,f_{u_i},f_{p_i},f_{p_ip_j},\dots$.

Given $\m{f}\in\{f,f_x,f_{u_i},f_{p_i},f_{p_ip_j},\dots\}$,
we will use the notation%
\footnote{The superscript $0$ in $\m{f}^0$ denotes evaluation of $\m{f}$ along the reference arc $u^0$;
cf.~similar notation $\hat L(t)=L(t,\hat x(t),\dot{\hat x}(t))$ in \cite{Zei}
or $\overline{\m{f}}(x)=\m{f}(x,u(x),u'(x))$ in \cite[formulas (30), (39) in
Section~2.3, pp.~114--116]{GH}. 
The advantages of our notation will become evident in Section~\ref{sec-se}:
See the notation introduced in Theorem~\ref{thm-field1}.}
$$\m{f}^0(x):=\m{f}(x,u^0(x),(u^0)'(x)).$$

If $x\in\{a,b\}$ and $W$ is a space of functions $[a,b]\to\R^N$, then we set 
$$\begin{aligned}
\IN_x &:=I\setminus \ID_x, \\
\R^N_{\c{D},x} &:=\{\xi\in\R^N:\xi_i=0\hbox{ for }i\in\ID_x\}, \\
\R^N_{\c{N},x} &:=\{\xi\in\R^N:\xi_i=0\hbox{ for }i\in\IN_x\}, \\
W_{\c{D},x} &:=\{v\in W:v(x)\in\R^N_{\c{D},x}\}, \\
W_\c{D} &:=W_{\c{D},a}\cap W_{\c{D},b}.
\end{aligned}$$
In particular, if $W=C^1=C^1([a,b],\R^N)$, then
\begin{equation} \label{C1D}
C^1_\c{D} = \{v\in C^1([a,b],\R^N): v_i(a)=0 \hbox{ for }i\in \ID_a,\ v_i(b)=0 \hbox{ for }i\in \ID_b\}
\end{equation}
is the space of $C^1$-test functions.
(Notice that the set $\c{M}$ in \eqref{M} satisfies $\c{M}=u^0+C^1_\c{D}$.)

The norm in a general Banach space $X$ will be denoted by $\|\cdot\|_X$;
the norm in $W^{1,2}$ will also be denoted by $\|\cdot\|_{1,2}$.
In particular, if $X=C^1=C^1([a,b],\R^N)$ or $X=C=C([a,b],\R^N)$, 
then $\|u\|_{C^1}=\max_{x\in[a,b]}|u(x)|+\max_{x\in[a,b]}|u'(x)|$
or $\|u\|_C=\max_{x\in[a,b]}|u(x)|$,
respectively, where $|u(x)|$ denotes the Euclidean norm of $u(x)\in\R^N$.
We also set $B_\eps:=\{\xi\in\R^N:|\xi|<\eps\}$.

We will assume that $u^0$ is a critical point of $\Phi$
in the set $u^0+C^1_\c{D}$, 
i.e.~$\Phi'(u^0)h=0$ for any test function $h\in C^1_\c{D}$, 
where $\Phi'$ denotes the Fr\'echet derivative of $\Phi$.
The following proposition is well known, but for the reader's convenience we
explain the idea of its proof in the Appendix.

\begin{proposition} \label{prop-NBC}
Let $f\in C^1$ and let $u^0$ be a critical point of $\Phi$ in $u^0+C^1_\c{D}$.
Then $u^0$
is an extremal (i.e.~it satisfies the Euler
equations $\frac{d}{dx}(f^0_{p_i})=f^0_{u_i}$, $i=1,2,\dots,N$), 
and $u^0$ also has to satisfy the natural boundary conditions
\begin{equation} \label{NBC}
f^0_{p_j}(a)=0\ \hbox{ for }\ j\in \IN_a\quad\hbox{ and }\quad
f^0_{p_j}(b)=0\ \hbox{ for }\ j\in \IN_b.
\end{equation}
If $f_{p_i}\in C^1$ for $i=1,2,\dots,N$, and the strengthened Legendre condition 
\begin{equation} \label{f-conv}
(\exists c^0>0)\qquad
 \sum_{i,j=1}^N f^0_{p_ip_j}(x)\xi_i\xi_j\geq c^0|\xi|^2,\quad \xi\in\R^N,\ x\in[a,b],
\end{equation}
is true, then $u^0\in C^2$.
\end{proposition}

It is known that the Legendre condition (i.e.~condition~\eqref{f-conv}
with $c^0=0$) is necessary for $u^0$ to be a minimizer, but even the strengthened Legendre
condition is not sufficient, in general.
Assuming that 
\begin{equation} \label{ass1}
\hbox{$f\in C^3$ satifies \eqref{f-conv}, where
$u^0\in C^1([a,b],\R^N)$ is an extremal satisfying \eqref{NBC},} 
\end{equation}
and denoting $\sum_k=\sum_{k=1}^N$, we set
\begin{equation} \label{psi}
\Psi(h):=\int_a^b\m{F}(x,h(x),h'(x))\,dx, \quad h\in W^{1,2}([a,b],\R^N),
\end{equation}
where
\begin{equation} \label{gLagr}
 \m{F}=\m{F}(x,u,p):=\sum_{i,j}\Bigl(f^0_{p_ip_j}(x)p_ip_j+f^0_{p_iu_j}(x)p_iu_j
+f^0_{u_ip_j}(x)u_ip_j+f^0_{u_iu_j}(x)u_iu_j\Bigr).
\end{equation}
If $h\in C^1$,  
then $\Psi(h)=\Phi''(u^0)(h,h)$, i.e.~$\Psi$ is the second variation of
$\Phi$ at $u^0$.
In addition, if $h\in C^2$, then integration by parts yields
\begin{equation} \label{PsiAB}
 \Psi(h)=\int_a^b\sum_i (\c{A}_ih)h_i\,dx +\sum_i (\c{B}_ih)h_i\Big|_a^b,
\end{equation}
where
\begin{equation} \label{BA}
\c{A}_ih:=-\frac{d}{dx}(\c{B}_ih)+\c{C}_ih,\quad
 \c{B}_ih:=\sum_j\Bigl(f^0_{p_ip_j}h_j'+f^0_{p_iu_j}h_j\Bigr),\quad 
\c{C}_ih:=\sum_j\Bigl(f^0_{u_ip_j}h_j'+f^0_{u_iu_j}h_j\Bigr).
\end{equation}
Set also
$$\c{A}h:=(\c{A}_1h,\dots,\c{A}_Nh),
\ \c{B}h:=(\c{B}_1h,\dots\c{B}_Nh),
\ f_p:=(f_{p_1},\dots,f_{p_N}),\ f_u=(f_{u_1},\dots,f_{u_N}).$$

The (vector-valued) second-order linear differential equation $\c{A}h=0$ 
is called the {\it Jacobi} equation (for $\Phi$ and $u^0$):
It will play a fundamental role 
in the study of positive definiteness of $\Psi$.
Notice also that the Jacobi equation is the Euler equation for functional
$\Psi$.
More precisely,
by using the symmetry relations
$f_{p_ip_j}=f_{p_jp_i}$, $f_{p_iu_j}=f_{u_jp_i}$ and
$f_{u_iu_j}=f_{u_ju_i}$ we obtain
\begin{equation} \label{gBC}
\m{F}_{p_i}(x,h(x),h'(x))=2\c{B}_ih(x), \quad
\m{F}_{u_i}(x,h(x),h'(x))=2\c{C}_ih(x),
\end{equation}
hence
\begin{equation} \label{Jacobi-Euler}
2\c{A}_ih(x)=-\frac{d}{dx}\m{F}_{p_i}(x,h(x),h'(x))+\m{F}_{u_i}(x,h(x),h'(x)).
\end{equation}
Notice also that, given $h,w\in W^{1,2}$,
\eqref{gBC} and the symmetry of the second-order derivatives of $f$
mentioned above imply
\begin{equation} \label{Psiprime}
\begin{aligned}
\Psi'(h)w &=\int_a^b\sum_i\bigl(\m{F}_{p_i}(x,h(x),h'(x))w_i'(x)+\m{F}_{u_i}(x,h(x),h'(x))w_i(x)\bigr)\,dx \\
&=2\int_a^b\sum_i\bigl(\c{B}_ih\cdot w_i'+\c{C}_ih\cdot w_i\bigr)\,dx 
 =2\int_a^b\sum_i\bigl(\c{B}_iw\cdot h_i'+\c{C}_iw\cdot h_i\bigr)\,dx=\Psi'(w)h.
\end{aligned}
\end{equation}

\begin{definition} \label{def-ws} \rm
Let $w\in \c{M}$, where $\c{M}$ is a subset of $C^1([a,b],\R^N)$.
The function $w$ is called
a {\it weak} 
or {\it strong local} {\it minimizer} in $\c{M}$
if there exists $\eps>0$ such that $\Phi(v)\geq\Phi(w)$
for any $v\in \c{M}$ satisfying
$\|v-w\|_{C^1}<\eps$ or $\|v-w\|_C<\eps$, respectively.

Let $w\in \c{N}$, where $\c{N}$ is a subset of $W^{1,2}([a,b],\R^N)$.
The function $w$ is called
a {\it local minimizer} in $\c{N}$
if there exists $\eps>0$ such that $\Phi(v)\geq\Phi(w)$
for any $v\in \c{N}$ satisfying
$\|v\|_{1,2}<\eps$.

If the inequalities $\Phi(v)\geq\Phi(w)$ in the definitions above
are strict for $v\ne w$, then the minimizer $w$ is called {\it strict}.
\end{definition}

Since the adjectives {\it weak} and {\it strong} are not meaningful
in the case of global minimizers, 
we often omit the word ``local'' in the notions of weak and strong local minimizers.
Each strong minimizer is a weak minimizer but the opposite is not true, in
general. For example, if $N=1$ and $f(x,u,p)=p^2+p^3$, then $u^0\equiv0$
is a weak but not strong minimizer of $\Phi$ in $u^0+C^1_\c{D}$
for any choice of $a,b,\ID_a$ and $\ID_b$
(see also Example~\ref{ex-well} for a less trivial example).
On the other hand, the following Proposition~\ref{prop-ws} 
and Remark~\ref{rem-twist}(vi)
show that in some cases the notions
of weak and strong minimizers are equivalent. 
The choice of the class of Lagrangians 
in Proposition~\ref{prop-ws}
is motivated by Example~\ref{ex-LG}, where we consider the stability
of a planar rod. Proposition~\ref{prop-ws} is true for any choice of
$a,b,\ID_a$ and $\ID_b$; its proof is postponed to the Appendix.

\begin{proposition} \label{prop-ws}
Let $N=1$ 
and $f(x,u,p)=(p-K)^2+g(u)$, where $K\in\R$ and $g\in C^1(\R)$. 
If $u^0\in C^1$ is a weak minimizer, 
then it is a strong minimizer. 
\end{proposition}

The following proposition is a consequence of well known facts (see \cite{GH, Ces}, for example). 
The assumptions in that proposition are much stronger than necessary,
but the proposition will be sufficient for our purposes
(see Remark~\ref{rem-twist}(vi), Section~\ref{sec-se} and the proof of Proposition~\ref{prop-lambda}).

\begin{proposition} \label{thm-Euler}
\strut\hbox{\rm(i)} 
Let $f\in C^k$, $k\geq2$.

If $u^0\in C^1$ is a critical point of $\Phi$ 
in $u^0+C^1_\c{D}$ and \eqref{f-conv} is true,
then $u^0\in C^k$  and $u^0$ satisfies the Du Bois-Reymond equation
\begin{equation} \label{eq-dBR}
\frac d{dx}(f^0-(u^0)'\cdot f^0_p) = f^0_x \quad\hbox{ in }[a,b]. 
\end{equation}

Conversely,
if $u^0\in C^2$ satisfies \eqref{eq-dBR} and $(u^0)'\ne0$ a.e., then $u^0$ is an extremal.

\strut\hbox{\rm(ii)}
Let $f\in C^1$ satisfy the growth condition 
$(1+|p|)|f_p|+|f_u|\leq M(|u|)(1+|p|)^2$,
where $M:[0,\infty)\to[0,\infty)$ is nondecreasing.
Then $\Phi\in C^1(W^{1,2})$. In addition,
if $u^0\in W^{1,2}$ is a local minimizer of $\Phi$ in $u^0+W^{1,2}_\c{D}$,
then there exists $C\in\R^N$ such that 
$$ 
 f^0_p(x) =\int_a^x f^0_u(\xi)\,d\xi+C \quad\hbox{for a.e. }x\in[a,b].
$$ 
\end{proposition}

\section{Jacobi theory}
\label{sec-Jacobi}

In this section we will prove necessary and sufficient conditions
for weak minimizers by modifying the classical Jacobi theory.
Throughout this section we assume \eqref{ass1}. 

The following proposition is well known,
but for the reader's convenience we provide its proof in the Appendix.

\begin{proposition} \label{prop-psi}
Assume \eqref{ass1} and let $\Psi$ be defined by \eqref{psi}.

(i) If $\Psi$ is positive definite in $W^{1,2}_\c{D}$, then
$u^0$ is a strict weak minimizer in $u^0+C^1_\c{D}$.

(ii) If $\Psi(h)<0$ for some $h\in W^{1,2}_\c{D}$, then
$u^0$ is not a weak minimizer in $u^0+C^1_\c{D}$.
\end{proposition}

We will consider the scalar case first.
Assume that
\begin{equation} \label{ass2}
\hbox{$h$ is a nontrivial solution of the Jacobi equation $\c{A}h=0$.}
\end{equation} 
Then the following classical result for problems with Dirichlet endpoint
constraints is well known.

\begin{theorem} \label{thm-Jacobi0}
Assume \eqref{ass1} with $N=1$ and \eqref{ass2}. 
Let $\IN_a=\IN_b=\emptyset$ and $h(a)=0$.

\strut\hbox{\rm(i)} If $h(y)=0$ for some $y\in(a,b)$, then $u^0$
is not a weak minimizer.

\strut\hbox{\rm(ii)} If $h(y)\ne0$ for any $y\in(a,b]$, then $u^0$
is a strict weak minimizer.
\end{theorem}

Our analogue in the case of variable endpoints is the following theorem.

\begin{theorem} \label{thm-Jacobi1}
Assume \eqref{ass1} with $N=1$ and \eqref{ass2}.
Let $\IN_a=\IN_b=\{1\}$ and $\c{B}h(a)=0$.

\strut\hbox{\rm(i)} If $h(y)=0$ for some $y\in(a,b]$ or $\c{B}h(b)h(b)<0$, then $u^0$
is not a weak minimizer.

\strut\hbox{\rm(ii)} If $h(y)\ne0$ for any $y\in(a,b]$ and $\c{B}h(b)h(b)>0$, then $u^0$
is a strict weak minimizer.
\end{theorem}

In fact, a slight generalization of Theorem~\ref{thm-Jacobi1}(ii) has been
proved in \cite{Bed}:
The initial condition $\c{B}h(a)=0$ can be replaced with $\c{B}h(a)h(a)\leq0$.
Unfortunately, the method of the proof in \cite{Bed} does not seem to be easily extendable
to the vector-valued case.

Theorems~\ref{thm-Jacobi0} and~\ref{thm-Jacobi1} are special cases of the following general theorem.

\begin{theorem} \label{thm-Jacobi}
Assume \eqref{ass1}.
Let $h^{(1)},\dots,h^{(N)}$ be linearly independent solutions
of the Jacobi equation $\c{A}h=0$
satisfying the initial conditions
$h(a)\in\R^N_{\c{D},a}$, $\c{B}h(a)\in\R^N_{\c{N},a}$.
Set 
$$D(x):=\det(h^{(1)}(x),\dots,h^{(N)}(x)),\
H:=\hbox{\rm span}(h^{(1)},\dots,h^{(N)}),\ 
H_0:=\{h\in H: h(b)=0\}.$$

\strut\hbox{\rm(i)}
If $D(x)=0$ for some $x\in(a,b)$ or 
$$\hbox{$\IN_b\ne\emptyset$ \ and \ $\c{B}h(b)\cdot h(b)<0$ 
 for some $h\in H_{\c{D},b}$},$$
then $u^0$ is not a weak minimizer.

\strut\hbox{\rm(ii)} If $D\ne0$ in $(a,b]$ and
$$\hbox{either \ $\IN_b=\emptyset$ \ or \ 
$\c{B}h(b)\cdot h(b)>0$ 
for any $h\in H_{\c{D},b}\setminus\{0\}$},$$
then $u^0$ is a strict weak minimizer.

\strut\hbox{\rm(iii)}
Let $D\ne0$ in $(a,b)$, $D(b)=0$ (hence $H_0\ne\{0\}$), and $\IN_b\ne\emptyset$. 
If 
\begin{equation}\label{cond-not}
\hbox{there exists $h\in H_0$ such that $\c{B}_i h(b)\ne0$ for some $i\in\IN_b$,}
\end{equation}
then $u^0$ is not a weak minimizer.
If $\ID_b=\emptyset$, then \eqref{cond-not} is always true.
\end{theorem}

The proof of Theorem~\ref{thm-Jacobi} is based
on a modification of the classical Jacobi theory,
and this is also true in the case of the
corresponding proof in \cite{Zez93}. However, our 
conditions in Theorem~\ref{thm-Jacobi} are simpler
than those in \cite{Zei, Zez93}, see Remark~\ref{rem-ZZ} in the Appendix.

In order to prove Theorem~\ref{thm-Jacobi}, we need some preparation.
Given $y\in(a,b]$, let
$$X_y:=\{h\in W^{1,2}([a,b],\R^N):h(a)\in\R^N_{\c{D},a}, 
\ h(x)=0\hbox{ for }x\geq y\}$$
be endowed with the norm
$\|h\|_{X_y}:=(\int_a^b\sum_{i,j}f^0_{p_ip_j}h_i'h_j'\,dx)^{1/2}$
(which is equivalent to the standard norm in $W^{1,2}$ for $h\in X_y$ due to \eqref{f-conv}
and the boundary condition $h(b)=0$),
and let $S_y$ denote the unit sphere in $X_y$.
If $\tilde y\in(y,b]$,
then $X_y\subset X_{\tilde y}$, hence 
$S_y\subset S_{\tilde y}$.
Set also
\begin{equation} \label{lambda1}
  \lambda_1=\lambda_1(y):=\inf_{h\in S_y}\Psi(h)=1+\inf_{h\in S_y}\hat\Psi(h),
\end{equation}
where
$$
\hat\Psi(h):=\int_a^b\sum_{i,j}\Bigl(f^0_{p_iu_j}h_i'h_j
+f^0_{u_ip_j}h_ih_j'+f^0_{u_iu_j}h_ih_j\Bigr)\,dx.
$$
Since $S_y\subset S_{\tilde y}$ if $y<\tilde y$,
the function $\lambda_1$ is nonincreasing.
In addition, one can easily show that $\lambda_1$ is continuous,
and the estimate
$$ \begin{aligned}
|h(x)|&=\Big|\int_x^yh'(\xi)\,d\xi\Big|
\leq\Bigl(\int_x^y|h'(\xi)|^2\,d\xi\Bigr)^{1/2}\sqrt{y-x} \\
&\leq\frac1{\sqrt{c^0}}\Bigl(\int_a^b\sum_{i,j}f^0_{p_ip_j}h_i'h_j'\,d\xi\Bigr)^{1/2}\sqrt{y-a}
= \frac1{\sqrt{c^0}}\sqrt{y-a} 
\end{aligned}
$$ 
for $h\in S_y$ and $x\in(a,y)$ implies $\lim_{y\to a+}\lambda_1(y)= 1$.

\begin{proposition} \label{prop-lambda}
Let $D$ be as in Theorem~\ref{thm-Jacobi} and $y\in(a,b]$.

{\rm(i)} If $\lambda_1(y)=0$, then $D(y)=0$ and $\lambda_1(z)<0$ for $z\in(y,b]$.
If $D(y)=0$, then $\lambda_1(y)\leq0$.

{\rm(ii)} If $h\in X_b$, then $\Psi(h)\geq\lambda_1(b)\|h\|_{X_b}^2$. 
If $\lambda_1(b)<0$, then there exists $h\in X_b$ such that $\Psi(h)<0$.
\end{proposition}

\begin{proof}
Let $\lambda_1(y)=0$ and let $B_y$ denote the closed unit ball in $X_y$. 
Since $\hat\Psi$ is weakly sequentially 
continuous, there exists $h_y\in B_y$
such that $\hat\Psi(h_y)=\inf_{B_y}\hat\Psi=-1$.
We have $h_y\in S_y$
(otherwise $th_y\in B_y$ for some $t>1$, and
$\hat\Psi(th_y)=t^2\hat\Psi(h_y)<\inf_{B_y}\hat\Psi$,
which yields a contradiction).
Since $\Psi(h_y)=\inf_{S_y}\Psi=0$,
$h_y$ is a global minimizer of $\Psi$ in $X_y$.  
Notice that $\m{F}\in C^1$ satisfies the growth condition
$$(1+|p|)|\m{F}_p(x,u,p)|+|\m{F}_u(x,u,p)|\leq 
C(1+|p|)(|u|+|p|)\leq 2C(1+|u|^2)(1+|p|^2),$$
where $C$ depends only on the sup-norm of
$f^0_{p_ip_j},f^0_{p_iu_j},f^0_{u_ip_j},f^0_{u_iu_j}$, 
hence
Proposition~\ref{thm-Euler}(ii) and \eqref{gBC} imply   
\begin{equation} \label{hyC1}
\begin{aligned}
2\c{B}_ih_y(x)=\m{F}_{p_i}(x,h_y(x),h'_y(x))
=\int_a^x\m{F}_{u_i}(\xi,h_y(\xi),h'_y(\xi))\,d\xi+c_i
=\int_a^x2\c{C}_ih_y\,d\xi+c_i
\end{aligned}
\end{equation}
for a.e.~$x\in[a,y]$. 
Since the right-hand side of \eqref{hyC1} is a continuous
function of $x$, $f\in C^3$ and \eqref{f-conv} is true
(hence the matrix $f^0_{p_ip_j}$ is invertible
and the inverse matrix is a continuous function of $x$),   
we see that the restriction of $h_y$ to $[a,y]$ is $C^1$.
Denote this restriction by $\bar h_y$ and
set $C^1_y:=\{w\in C^1([a,y]): w(a)\in \R^N_{\c{D},a},\ w(y)=0\}$,
$\Psi_y(h)=\int_a^y \m{F}(x,h(x),h'(x))\,dx$.
Then $\bar h_y$ is a critical point
of $\Psi_y$ in $\bar h_y+C^1_y=C^1_y$. 
Now Proposition~\ref{prop-NBC}, \eqref{Jacobi-Euler} and \eqref{gBC} imply
that $\bar h_y$ is $C^2$, it satisfies the Jacobi equation $\c{A}h=0$ in $[a,y]$ 
and the natural boundary conditions $\c{B}h(a)\in\R^N_{\c{N},a}$.
Since we also have $h_y(a)\in\R^N_{\c{D},a}$,
there exists $\alpha\in\R^N\setminus\{0\}$ such that
$h_y=\sum_k\alpha_kh^{(k)}$ on $[a,y]$, where $h^{(k)}$ are as in Theorem~\ref{thm-Jacobi}.
Since $h_y(y)=0$, we have $D(y)=0$.

Next assume on the contrary that $\lambda_1(y)=0=\lambda_1(z)$ for some $z\in(y,b]$.
Then the minimizer $h_y$ is a global minimizer of $\Psi$ in $X_z$.
Similarly as above we deduce that $h_y\in C^2([a,z])$
and $h_y$ solves the Jacobi equation in $[a,z]$.
Consequently, $h_y(y)=h_y'(y)=0$, which yields a contradiction with
the uniqueness of solutions of the initial value problem for the Jacobi equation. 

Next assume that $D(y)=0$. Then there exists 
$\alpha=(\alpha_1,\dots,\alpha_N)\in\R^N\setminus\{0\}$ such that
$h:=\sum_k\alpha_kh^{(k)}$ satisfies $h(y)=0$, hence 
if we set $\tilde h(x):=h(x)$ for $x\leq y$ and $\tilde h(x):=0$ otherwise, then
$\tilde h\in X_y$. In addition,
using $\c{A}_ih=0$, $\c{B}_ih(a)\in\R^N_{\c{N},a}$, $h(a)\in\R^N_{\c{D},a}$
and $h(y)=0$ we obtain
$$ \begin{aligned} \Psi(\tilde h) &=
\int_a^b \m{F}(x,\tilde h(x),\tilde h'(x))\,dx 
=\int_a^y \m{F}(x,h(x),h'(x))\,dx \\
&=\int_a^y\sum_i(\c{A}_ih)h_i\,dx+\sum_i(\c{B}_ih)h_i\Big|_a^y=0,
\end{aligned}$$ 
hence $\lambda_1(y)\leq0$.

If $h\in X_b\setminus\{0\}$, then 
 $\Psi(h)=\|h\|_{X_b}^2\Psi(h/\|h\|_{X_b})\geq\lambda_1(b)\|h\|_{X_b}^2$
by the definition of $\lambda_1$.
If $\lambda_1(b)<0$, then the definition of $\lambda_1$ implies the existence of
$h\in S_b$ such that $\Psi(h)<0$.
\end{proof}

\begin{proof}[Proof of Theorem~\ref{thm-Jacobi}]
We will show that 
\begin{equation} \label{ih0}
\hbox{the assumptions in (i) (or (iii)) imply
$\Psi(h)<0$ for some $h\in W^{1,2}_\c{D}$,}
\end{equation}
while 
\begin{equation} \label{iiPsi0}
\hbox{the assumptions in (ii) guarantee 
that $\Psi$ is positive definite in $W^{1,2}_\c{D}$,}
\end{equation}
hence the assertions in Theorem~\ref{thm-Jacobi} will follow from
Proposition~\ref{prop-psi}. 

(i)
If $D(x)=0$ for some $x\in(a,b)$, then Proposition~\ref{prop-lambda}(i) 
implies $\lambda_1(x)\leq0$ and $\lambda_1(b)<0$,
hence Proposition~\ref{prop-lambda}(ii)
implies the existence of $h\in X_b\subset W^{1,2}_\c{D}$ such that $\Psi(h)<0$.

If $\IN_b\ne\emptyset$ and 
$\c{B}h(b)\cdot h(b)<0$ for some $h\in H_{\c{D},b}\subset W^{1,2}_\c{D}$,
then $\c{A}h=0$, $h_i(a)=0$ for $i\in\ID_a$ and $\c{B}_ih(a)=0$ for $i\in\IN_a$,
hence \eqref{PsiAB} implies
$$\Psi(h)=\c{B}h\cdot h\Big|_a^b=\c{B}h(b)\cdot h(b)<0.$$

(ii)
Assume that $D\ne0$ in $(a,b]$. Then Proposition~\ref{prop-lambda} implies
$\lambda_1(b)>0$ and $\Psi(h)\geq\lambda_1(b)\|h\|_{X_b}^2$ for $h\in X_b$.
If $\IN_b=\emptyset$, then $X_b=W^{1,2}_{\c{D}}$, hence we are done.

Next assume that $\IN_b\ne\emptyset$ and 
$\c{B}\tilde h(b)\cdot \tilde h(b)>0$ for any $\tilde h\in H_{\c{D},b}\setminus\{0\}$
(hence $\c{B}\tilde h(b)\cdot \tilde h(b)\geq c_1\|\tilde h\|_{1,2}^2$ for some $c_1>0$
due to $\dim H_{\c{D},b}<\infty$), 
and let $h\in W^{1,2}_{\c{D}}$ be fixed.
Since $D(b)\ne0$, there exists $\alpha\in \R^N$ such that 
$\tilde h:=\sum_k\alpha_kh^{(k)}$ satisfies $\tilde h(b)=h(b)$.
In particular, $\tilde h\in H_{\c{D},b}$.
Set $\hat h:=h-\tilde h$. Then $\hat h\in X_b$, hence
$\Psi(\hat h)\geq\lambda_1(b)\|\hat h\|_{X_b}^2$.
In addition, $\Psi(\tilde h)=\c{B}\tilde h(b)\cdot\tilde h(b)\geq
\ c_1\|\tilde h\|_{1,2}^2$.
Since $\Psi$ is a quadratic functional, we have 
$\Psi''(\tilde h)(\hat h,\hat h)=2\Psi(\hat h)$ and $\Psi'''=0$.
Using \eqref{Psiprime} and integration by parts we also obtain
$$\Psi'(\hat h)\tilde h=\Psi'(\tilde h)\hat h=
2\int_a^b\c{A}\tilde h\cdot\hat h\,dx
+2\c{B}\tilde h\cdot \hat h\Big|_a^b=0,$$
hence there exists $c>0$ such that
$$\Psi(h)=\Psi(\tilde h+\hat h)
=\Psi(\tilde h)+\Psi'(\tilde h)\hat h+\frac12\Psi''(\tilde h)(\hat h,\hat h)  
=\Psi(\tilde h)+\Psi(\hat h)\geq c\|h\|^2_{1,2}.$$

(iii)
Let $h\in H_0$ and $\c{B}_{i}h(b)\ne0$ for some $i\in\IN_b$.
Then $\c{A}h=0$, $h(a)\in\R^N_{\c{D},a}$, $\c{B}h(a)\in\R^N_{\c{N},a}$
and $h(b)=0$, hence
$$
\Psi(h)=\int_a^b\c{A}h\cdot h\,dx+\c{B}h\cdot h\Big|_a^b= 0.$$
Notice also that $h\ne0$ due to $\c{B}_{i}h(b)\ne0$.
Since $D\ne0$ in $(a,b)$, $D(b)=0$,
$\lim_{y\to a+}\lambda_1(y)=1$ and $\lambda_1$ is continuous
and nonincreasing, Proposition~\ref{prop-lambda}(i) implies $\lambda_1(b)=0$,
hence
$h$ is a global minimizer of $\Psi$ in $X_b$.
Choose $\tilde h\in C^1_{\c{D}}$ with $\tilde h(a)=0$, $\tilde
h_j(b)=\delta_{ij}$ for $j=1,2,\dots,N$.
Then
$$ \Psi'(h)\tilde h= 2\int_a^b\c{A}h\cdot\tilde h\,dx+2\c{B}h\cdot\tilde h\Big|_a^b
= 2\c{B}_ih(b)\ne0,$$
hence 
$$ \Psi(h+\eps\tilde h)=\eps\Psi'(h)\tilde h+o(\eps) <0$$
provided $|\eps|$ is small enough and $\eps\c{B}_ih(b)<0$.

If $\ID_b=\emptyset$ and $h\in H_0\setminus\{0\}$,
then $\c{A}h=0$ and $h(b)=0$, hence 
the uniqueness of the initial value problem
for the Jacobi equation implies the existence
of $i\in\IN_b=I$ such that
$\c{B}_ih(b)\ne0$.
\end{proof}

\begin{remark}\label{rem-Jacobi}\rm
(i)  If $\Psi$ is positive semidefinite but not positive definite,
then there exists $h^*\in W^{1,2}_{\c{D}}\setminus\{0\}$
such that $0=\Psi(h^*)=\inf_{W^{1,2}_\c{D}}\Psi$
and $h^*$ can be determined from our analysis.
For example, if $N=1$ and $\ID_a=\ID_b=\emptyset$ (cf.~Theorem~\ref{thm-Jacobi1}),
then $h^*$ is a positive (or negative) solution of the Jacobi equation satisfying $\c{B}h^*(a)=\c{B}h^*(b)=0$. 
If $\Phi$ depends smoothly on a parameter $\theta$, 
$u^0$ is a critical point of $\Phi$ for any $\theta$,
and $u^0$ is (or is not, respectively) a weak minimizer for $\theta>\theta^*$ (or $\theta<\theta^*$, respectively),
then the critical parameter $\theta^*$ corresponds to the case where $h^*$ exists.
(Such situation occurs, for example, in the study of stability of a twisted rod in
Section~\ref{sec-twist}.) 
In this case one can expect bifurcation for the problem $\Phi'(u)=0$ at $\theta=\theta^*$
in the direction of $h^*$, cf. \cite[Theorem 5.6]{CRS}.

(ii) Let $h^{(k)}$, $k=1,2,\dots,N$, be as in Theorem~\ref{thm-Jacobi}, $\xi\in\R^N$
and $h^\xi:=\sum_k\xi_kh^{(k)}$.
Set $\m{A}:=(a_{kl})_{k,l=1}^N$,
where $a_{kl}=\c{B}h^{(k)}(b)\cdot h^{(l)}(b)$,
and 
$$\Xi_{\c{D}}:=\{\xi\in\R^N:h^\xi(b)\in\R^N_{\c{D},b}\}.$$  
Then $\c{B}h^\xi(b)\cdot h^\xi(b)=\m{A}\xi\cdot\xi$,
i.e. the condition  $\c{B}h(b)\cdot h(b)>0$ for any $h\in H_{\c{D},b}\setminus\{0\}$
in Theorem~\ref{thm-Jacobi}(ii), for example, 
is equivalent to $\m{A}\xi\cdot\xi>0$ for any $\xi\in\Xi_\c{D}\setminus\{0\}$. 
In particular, if $\ID_b=\emptyset$ (and $D(b)\ne0$),
then that condition is equivalent to the positive definiteness of the matrix $\m{A}$. 
Notice also that $a_{kl}=a_{lk}$ due to
$2a_{kl}=\Psi'(h^{(k)})h^{(l)}$ and $\Psi'(h^{(k)})h^{(l)}=\Psi'(h^{(l)})h^{(k)}$.

(iii)  
Assertions \eqref{iiPsi0} or \eqref{ih0} show that some of the assumptions in Theorem~\ref{thm-Jacobi}
are sufficient for the positivity or the negativity of $\Psi$, respectively.   
We will show that those assumptions are also necessary, at least in some cases.

Let $\Psi$ be positive definite in $W^{1,2}_\c{D}$. Since $X_b\subset W^{1,2}_\c{D}$,
$\Psi$ is also positive definite in $X_b$ and Proposition~\ref{prop-lambda}(i)
implies $D\ne0$ in $[a,b]$.
If $\IN_b\ne\emptyset$ and $h\in H_{\c{D},b}\setminus\{0\}$,
then $h\in W^{1,2}_\c{D}$, $\c{B}h(a)\in\R^N_{\c{N},a}$, $\c{A}h=0$, hence
$$ 0<\Psi(h)=\int_a^b\c{A}h\cdot h\,dx+\c{B}h\cdot h\Big|_a^b=\c{B}h(b)\cdot h(b),$$
so that the assumptions in  Theorem~\ref{thm-Jacobi}(ii) are satisfied.
This fact and \eqref{iiPsi0} show that the positive definiteness of $\Psi$
in $W^{1,2}_\c{D}$ and the assumptions of Theorem~\ref{thm-Jacobi}(ii) 
are equivalent.

Let $\Psi(\bar h)<0$ for some $\bar h\in W^{1,2}_\c{D}$ and 
\begin{equation} \label{IDbINb}
\ID_b=\emptyset \ \hbox{ or }\ \IN_b=\emptyset.
\end{equation}
Assume that the assumptions of Theorem~\ref{thm-Jacobi}(i) are not satisfied.
Then $D\ne0$ in $(a,b)$ (hence $\lambda_1(b)\geq0$ due to Proposition~\ref{prop-lambda}(i))
and either $\IN_b=\emptyset$ or $\c{B}h(b)\cdot h(b)\geq0$ for any $h\in H_{\c{D},b}$.
If $\IN_b=\emptyset$, then $W^{1,2}_\c{D}=X_b$, hence
$\Psi\geq0$ in $W^{1,2}_\c{D}$, which is a contradiction.
Consequently, $\IN_b\ne\emptyset$, 
$\c{B}h(b)\cdot h(b)\geq0$ for any $h\in H_{\c{D},b}$ and $\ID_b=\emptyset$
(due to \eqref{IDbINb}).
If $D(b)\ne0$, then there exists $\tilde h\in H_{\c{D},b}$ such that 
$\tilde h(b)=\bar h(b)$. 
Set $\hat h:=\bar h-\tilde h\in X_b$. Then
similarly as in the proof of Theorem~\ref{thm-Jacobi}(ii) we obtain
$$ 0>\Psi(\bar h)=\Psi(\tilde h+\hat h)=\Psi(\tilde h)+\Psi(\hat h)
   \geq\c{B}\tilde h(b)\cdot\tilde h(b)+\lambda_1(b)\|\hat h\|_{X_b}^2\geq0,$$
which is a contradiction.
Consequently, $D(b)=0$. Since $\ID_b=\emptyset$ implies \eqref{cond-not},
all assumptions of Theorem~\ref{thm-Jacobi}(iii) are satisfied.
These considerations and \eqref{ih0} show that if \eqref{IDbINb} is true,
then the condition $\Psi(\bar h)<0$ for some $\bar h\in W^{1,2}_\c{D}$
is satisfied if and only if the assumptions of Theorem~\ref{thm-Jacobi}(i)
or the assumptions of Theorem~\ref{thm-Jacobi}(iii) are satisfied.
\qed
\end{remark}

\section{Stability of a twisted rod}
\label{sec-twist}

In this section we use Theorem~\ref{thm-Jacobi} in order to determine
the stability of an unbuckled state 
of an inextensible, unshearable, isotropic Kirchhoff rod.
Under suitable assumptions the strain energy of the rod is given by
$$ \Phi(u)=\int_0^1\Bigl(\frac{A}2\bigl((u_1')^2+(u_2')^2\sin^2u_1\bigr)
    +\frac{C}2(u_3'+u_2'\cos u_1)^2+FL^2\sin u_1\cos u_2\Bigr)\,dx,$$
where $u_1,u_2,u_3$ are so called Euler angles describing the orientation of the director basis,
$A,C>0$ are constants, $L$ is the rod-length and $F\in\R$ is an external
terminal load; the rod is oriented horizontally (along the $x$ axis), see \cite[(9)]{MR}.
The unbuckled state is given by $u^0(x):=(\frac\pi2,0,2\pi Mx)$
where $M$ is a twist parameter. 
Notice that $u^0$ is an extremal satisfying the natural boundary
conditions $f_{p_i}^0(x)=0$ for $i=1,2$ and $x=0,1$.
The stability of $u^0$ was studied in \cite{MR} under
the Dirichlet boundary conditions $u_3(x)=u^0_3(x)$ for $x=0,1$,
and one of the following sets of boundary conditions for $u_1,u_2$:
\begin{equation} \label{MR-D}
 u_1(0)=u_1(1)=\pi/2,\qquad u_2(0)=u_2(1)=0,
\end{equation} 
\begin{equation} \label{MR-ND}
 u_1(0)=u_1(1)=\pi/2,\qquad u'_2(0)=u'_2(1)=0,
\end{equation} 
\begin{equation} \label{MR-N}
 u'_1(0)=u'_1(1)=0,\qquad u'_2(0)=u'_2(1)=0.
\end{equation} 
The results in \cite{MR} are essentially optimal in case \eqref{MR-D},
but the results in cases \eqref{MR-ND} and \eqref{MR-N}
are only partial, leaving several open problems.
Notice that the Neumann boundary conditions are not the same as the natural
boundary conditions in general (see \cite{Olver} for related issues), but
one can easily show (see Proposition~\ref{prop-Neumann} and Remark~\ref{rem-Neumann}
in the Appendix) 
that the problem of stability of $u^0$ considered in
\cite{MR}  in cases \eqref{MR-ND} and \eqref{MR-N} 
is equivalent to the question whether $u^0$ is a weak minimizer
of $\Phi$ in $u^0+C^1_{\c{D}}$ with $\IN_0=\IN_1=\{2\}$ and
$\IN_0=\IN_1=\{1,2\}$, respectively; hence
we can use Theorem~\ref{thm-Jacobi} in order to solve those problems. 
In fact, we will consider all possible subsets $\IN_0,\IN_1$ of $\{1,2\}$,
and in each case we will find the borderline
between the stability and instability (i.e.~between the situations
when $u^0$ is and is not a weak minimizer, respectively).
On the other hand, we will always assume $3\in\ID_0\cap\ID_1$,
i.e. we will always consider the Dirichlet boundary conditions for the third
component $u_3$.
   
In order to have a more graphic notation, given $\IN_0,\IN_1\subset\{1,2\}$,
we denote the corresponding case by 
$\bigl(\strut^{c_0^1c_1^1}_{c_0^2c_1^2}\bigr)$,
where $c_j^i=\c{N}$ if $i\in\IN_j$, $c_j^i=\c{D}$ if $i\in\ID_j$, 
$i=1,2$, $j=0,1$.
For example, $\bigl(\strut^{\c{D}\c{D}}_{\c{N}\c{N}}\bigr)$
corresponds to the case $\IN_0=\IN_1=\{2\}$, i.e.~\eqref{MR-ND},
and $\bigl(\strut^{\c{N}\c{N}}_{\c{N}\c{N}}\bigr)$
corresponds to the case $\IN_0=\IN_1=\{1,2\}$, i.e.~\eqref{MR-N}.
Set also
\begin{equation} \label{alphabeta}
 \alpha:=\frac{2\pi CM}{A},\quad \beta:=-\frac{FL^2}{A},\quad 
\gamma:=\sqrt{\Big|\beta-\frac14\alpha^2\Big|},\quad \delta:=\frac\alpha2,\quad
\theta:=\frac{2\gamma\delta}{\gamma^2+\delta^2}. 
\end{equation}
We will show that we may assume $\alpha>0$, and
for any $\bigl(\strut^{c_0^1c_1^1}_{c_0^2c_1^2}\bigr)$ with $c_j^i\in\{\c{D},\c{N}\}$
we will find a function
$g=g\strut^{c_0^1c_1^1}_{c_0^2c_1^2}:(0,\infty)\to\R:\alpha\mapsto\beta$ 
which describes the borderline between stability and instability.
In the particular cases
\eqref{MR-D}, \eqref{MR-ND} and \eqref{MR-N}
we will also use the notation
$$ g_D:=g^{\c{D}\c{D}}_{\c{D}\c{D}}, \quad 
  g_M:=g^{\c{D}\c{D}}_{\c{N}\c{N}}, \quad\hbox{and}\quad 
  g_N:=g^{\c{N}\c{N}}_{\c{N}\c{N}}, $$
respectively
(the notation $g_M$ reflects the fact that case \eqref{MR-ND}
is called ``Mixed'' in \cite[(13)]{MR}).

\begin{proposition} \label{prop-twist}
Let $u^0$ be as above, $\alpha>0$, and let $\IN_0,\IN_1\subset\{1,2\}$ be fixed.
Then there exists a continuous function $g:(0,\infty)\to\R$ having the properties mentioned above, i.e.
if $\beta>g(\alpha)$ (or $\beta<g(\alpha)$, resp.),
then $u^0$ is a strict weak minimizer (or is not a weak minimizer, resp.).

\strut\hbox{\rm(i)} Let $\ID_0\cap\{1,2\}\ne\emptyset\ne\ID_1\cap\{1,2\}$.
Then
\begin{equation} \label{cases1}
g^{\c{D}\c{D}}_{\c{N}\c{D}}=g^{\c{D}\c{D}}_{\c{D}\c{N}}=g^{\c{D}\c{N}}_{\c{D}\c{D}}=g^{\c{N}\c{D}}_{\c{D}\c{D}},\quad
g^{\c{D}\c{N}}_{\c{N}\c{D}}=g^{\c{N}\c{D}}_{\c{D}\c{N}},\quad
g^{\c{D}\c{D}}_{\c{N}\c{N}}=g^{\c{N}\c{N}}_{\c{D}\c{D}}(=g_M),
\end{equation}
\begin{equation} \label{g4}
 \left.\begin{aligned}
  g_D(\alpha)&=\frac{\alpha^2}{4}-\pi^2,\quad
  g^{\c{D}\c{D}}_{\c{N}\c{D}}(\alpha)=\frac{\alpha^2}4-\frac{\pi^2}4,&&\\
  g^{\c{D}\c{N}}_{\c{N}\c{D}}(\alpha)&=(k+{\scriptstyle\frac12})\pi(\alpha-(k+{\scriptstyle\frac12})\pi) && 
       \kern-15mm\hbox{if }\alpha\in[2k\pi,2(k+1)\pi],\quad k=0,1,2,\dots,  \\ 
 g_M(\alpha)&=k\pi(\alpha-k\pi) && \kern-20mm\hbox{if }\alpha\in[(2k-1)\pi,(2k+1)\pi],\ k=0,1,2,\dots.
\end{aligned}\ \right\}
\end{equation}

\strut\hbox{\rm(ii)} Let either $\ID_0\cap\{1,2\}=\emptyset$ or $\ID_1\cap\{1,2\}=\emptyset$.
Then
\begin{equation} \label{cases2}
\qquad g^{\c{N}\c{D}}_{\c{N}\c{D}}=g^{\c{D}\c{N}}_{\c{D}\c{N}},\quad
g^{\c{N}\c{N}}_{\c{N}\c{D}}=g^{\c{N}\c{N}}_{\c{D}\c{N}},\quad
g^{\c{N}\c{D}}_{\c{N}\c{N}}=g^{\c{D}\c{N}}_{\c{N}\c{N}},
\end{equation}
$$ \begin{aligned}
 g_N(\alpha)&=\inf\{\beta\geq\frac12\alpha^2:
(1-\theta^2)\cosh(2\gamma)+\theta^2\cos(2\delta)=1\}\in[\frac12\alpha^2,\alpha^2],\\
 g^{\c{N}\c{D}}_{\c{N}\c{D}}(\alpha)&=\begin{cases}
  \sup\{\beta\in(\frac14\alpha^2,\frac12\alpha^2):(\alpha^2-2\beta)\cosh(2\gamma)=2\beta\} &\hbox{if }\alpha>2,\\
  \frac14\alpha^2 &\hbox{if }\alpha=2,\\
  \sup\{\beta\in(\frac14(\alpha^2-\pi^2),\frac14\alpha^2):
  (\alpha^2-2\beta)\cos(2\gamma)=2\beta\} & \hbox{if }\alpha\in(0,2),
  \end{cases}\\
 g^{\c{N}\c{N}}_{\c{N}\c{D}}(\alpha)&=
 \inf\{\beta\geq\beta_\alpha:
  (\gamma^2-\delta^2)\sinh(2\gamma)=2\gamma\delta\sin(2\delta)\}, 
\ \ \beta_\alpha:=\begin{cases} \frac12\alpha^2 &\hbox{if }\alpha\leq\pi,\\
                            g^{\c{N}\c{D}}_{\c{N}\c{D}}(\alpha) &\hbox{if }\alpha>\pi, \end{cases} \\
 g^{\c{N}\c{D}}_{\c{N}\c{N}}(\alpha)&=\begin{cases}
 \inf\{\beta\geq g^{\c{N}\c{D}}_{\c{N}\c{D}}(\alpha):
 (\gamma^2-\delta^2)\sinh(2\gamma)=-2\gamma\delta\sin(2\delta)\} &\hbox{if }\alpha\geq\alpha_0,\\
 \inf\{\beta\geq g^{\c{N}\c{D}}_{\c{N}\c{D}}(\alpha):\xi_1^2\sin\xi_2\cos\xi_1=\xi_2^2\sin\xi_1\cos\xi_2\} 
     &\hbox{if }\alpha\in(\frac12\pi,\alpha_0),\\
  0 &\hbox{if }\alpha\in(0,\frac12\pi],
 \end{cases} 
\end{aligned}
$$ 
where $\xi_i:=-\frac12\alpha\pm\gamma$ and $\alpha_0>0$ is defined by $\alpha_0=2\sin\alpha_0$.
\end{proposition}

\begin{remark} \label{rem-twist} \rm
(i) If $u^0$ is a weak minimizer of $\Phi$ with given $\IN_0,\IN_1$ (and the borderline function $g$),
then it remains a weak minimizer if we replace $\IN_x$ with any subset of $\IN_x$ for $x=0,1$,
since the set $C^1_\c{D}$ becomes smaller. Therefore the new borderline function $\tilde g$
has to satisfy $\tilde g\leq g$. In particular, $g_D\leq g\leq g_N$ for any borderline function $g$,
$g^{\c{N}\c{D}}_{\c{N}\c{D}}\leq \min(g^{\c{N}\c{N}}_{\c{N}\c{D}},g^{\c{N}\c{D}}_{\c{N}\c{N}})$,
and $g^{\c{N}\c{D}}_{\c{N}\c{D}}(\alpha)\geq g^{\c{D}\c{D}}_{\c{N}\c{D}}(\alpha)=\frac14(\alpha^2-\pi^2)$.
We also have $g_N(\alpha)\leq\alpha^2$ since the Cauchy inequality implies that 
the corresponding functional $\Psi$ is positive definite for $\beta>\alpha^2$.

\begin{figure}[ht]
\centering
\begin{picture}(300,385)(-15,-60)
\unitlength=1pt
\put(0,0){\line(1,0){280}}
\put(280,0){\vector(1,0){2}}
\put(300,-15){\makebox(0,0)[c]{$\alpha=\frac{2\pi CM}A$}} 
\put(50,-2){\line(0,1){4}}
\put(50,-10){\makebox(0,0)[c]{$\pi$}}
\put(100,-2){\line(0,1){4}}
\put(100,-10){\makebox(0,0)[c]{$2\pi$}}
\put(150,-2){\line(0,1){4}}
\put(150,-10){\makebox(0,0)[c]{$3\pi$}}
\put(200,-2){\line(0,1){4}}
\put(200,-10){\makebox(0,0)[c]{$4\pi$}}
\put(250,-2){\line(0,1){4}}
\put(250,-10){\makebox(0,0)[c]{$5\pi$}}
\put(0,-50){\line(0,1){360}}
\put(0,310){\vector(0,1){2}}
\put(35,310){\makebox(0,0)[c]{$\beta=-\frac{FL^2}{A}$}} 
\put(-10,-5){\makebox(0,0)[c]{$0$}}
\put(-2,50){\line(1,0){4}}
\put(-10,50){\makebox(0,0)[c]{$\pi^2$}}
\put(-13,-50){\makebox(0,0)[c]{$-\pi^2$}}
\put(-2,100){\line(1,0){4}}
\put(-12,100){\makebox(0,0)[c]{$2\pi^2$}}
\put(-2,150){\line(1,0){4}}
\put(-12,150){\makebox(0,0)[c]{$3\pi^2$}}
\put(-2,200){\line(1,0){4}}
\put(-12,200){\makebox(0,0)[c]{$4\pi^2$}}
\put(-2,250){\line(1,0){4}}
\put(-12,250){\makebox(0,0)[c]{$5\pi^2$}}
\put(-2,300){\line(1,0){4}}
\put(-12,300){\makebox(0,0)[c]{$6\pi^2$}}
\bezier100(0,0)(122,0)(250,310)
\put(25,32){\makebox(0,0){$\frac14\alpha^2$}} 
\put(25,23){\line(0,-1){18}}
\put(25,5){\vector(0,-1){2}}
\ecol
\bezier700(0,-50)(120,-50)(250,250)
\put(206,135){\makebox(0,0){$g_{D}$}} 
\refcol
\bezier700(0,-12)(128,-12)(250,300)
\put(138,80){\line(1,0){35}}
\put(139,80){\vector(-1,0){2}}
\put(184,80){\makebox(0,0){$g^{\c{D}\c{D}}_{\c{N}\c{D}}$}}
\blue
\bezier500(0,-12)(50,13)(100,38)
\bezier500(100,38)(150,113)(200,188)
\bezier500(200,188)(225,249)(250,310)
\put(116,60){\line(1,0){40}}
\put(117,60){\vector(-1,0){2}}
\put(169,60){\makebox(0,0){$g^{\c{D}\c{N}}_{\c{N}\c{D}}$}}
\put(50,13){\makebox(0,0){$\bullet$}}
\put(100,38){\makebox(0,0){$\bullet$}}
\put(150,113){\makebox(0,0){$\bullet$}}
\put(200,188){\makebox(0,0){$\bullet$}}
\put(250,310){\makebox(0,0){$\bullet$}}
\red
\put(0,0){\red\line(1,0){50}}
\bezier500(50,0)(100,50)(150,100)
\bezier500(150,100)(200,200)(250,300)
\put(280,295){\makebox(0,0){$g_M=g^{\c{D}\c{D}}_{\c{N}\c{N}}$}}
\put(50,0){\makebox(0,0){$\bullet$}}
\put(150,100){\makebox(0,0){$\bullet$}}
\put(250,300){\makebox(0,0){$\bullet$}}
\put(100,50){\makebox(0,0){$\bullet$}}
\put(200,200){\makebox(0,0){$\bullet$}}
\bezier100(50,0)(100,50)(125,263)
\put(115,250){\makebox(0,0){$\overline{g_M}$}}
\bezier20(0,-50)(10,-50)(25,-25)
\bezier20(25,-25)(40,0)(50,0)
\bezier45(50,0)(60,0)(100,26)
\bezier70(100,26)(120,39)(250,48)
\put(240,55){\makebox(0,0){$\underline{g_M}$}}
\put(33,-30){\makebox(0,0){$\underline{g_M}$}}
\black
\end{picture}
\kern-2mm
   \caption{The case $\ID_0\cap\{1,2\}\ne\emptyset\ne\ID_1\cap\{1,2\}$.}
   \label{fig0}
\end{figure}
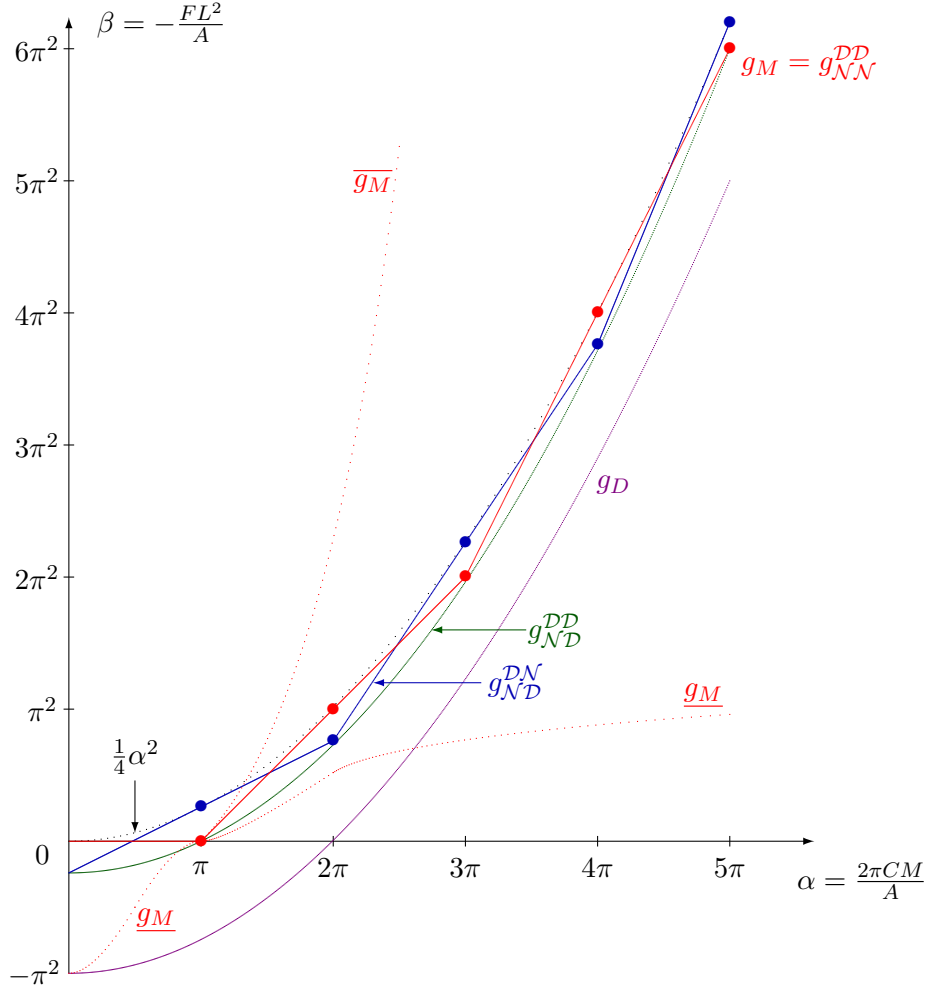

(ii) If $\alpha\in(0,\alpha_0)$ is fixed, then the function
$\Xi(\beta):=\xi_1^2\sin\xi_2\cos\xi_1-\xi_2^2\sin\xi_1\cos\xi_2$ appearing in the
formula for $g^{\c{N}\c{D}}_{\c{N}\c{N}}$ in Proposition~\ref{prop-twist}
has a unique root $\beta^*$ in in the interval
$[g^{\c{N}\c{D}}_{\c{N}\c{D}}(\alpha),\frac14\alpha^2)$:
This follows from our proof, since any root in that interval corresponds
to the case when the corresponding functional $\Psi$
is positive semidefinite but not positive definite,
and the form of $\Psi$ guarantees that, given $\alpha$, this can happen only for one $\beta$.
Consequently,
$$g^{\c{N}\c{D}}_{\c{N}\c{N}}(\alpha)=\sup\{\beta<\frac14\alpha^2:\xi_1^2\sin\xi_2\cos\xi_1=\xi_2^2\sin\xi_1\cos\xi_2\}
\quad\hbox{if }\alpha\in(0,\alpha_0).$$
In addition, our proof implies that if
$\beta^*>g^{\c{N}\c{D}}_{\c{N}\c{D}}(\alpha)$,
then $\Xi$ changes sign at $\beta^*$.
Similarly, if $\alpha>\alpha_0$ (or $\alpha>0$, resp.), 
then the function
$(\gamma^2-\delta^2)\sinh(2\gamma)+2\gamma\delta\sin(2\delta)$
(or 
$(\gamma^2-\delta^2)\sinh(2\gamma)-2\gamma\delta\sin(2\delta)$, resp.)
has a unique root $\beta^*$ in the interval $[g^{\c{N}\c{D}}_{\c{N}\c{D}}(\alpha),\infty)$
(or $[\beta_\alpha,\infty)$, resp.), and it changes sign at $\beta^*$ if $\beta^*>g^{\c{N}\c{D}}_{\c{N}\c{D}}(\alpha)$
(or $\beta^*>\beta_\alpha$, resp.). 
In addition, the estimates in (i) guarantee that that root $\beta^*$
satisfies $\beta^*\leq g_N(\alpha)\leq\alpha^2$.
Analogous statements are true in the case of $g_N$. 

(iii) 
Our definition of $\alpha$ and $\beta$ in \eqref{alphabeta} implies 
that the borderline function $g_M$ was estimated above and below in \cite[Proposition~6]{MR}
by functions
$$ \overline{g_M}(\alpha):=\max(0,\alpha^2-\pi^2)\quad\hbox{and}\quad
   \underline{g_M}(\alpha):=\pi^2(\alpha^2-\pi^2)/(\alpha^2+\pi^2), $$
respectively, see Figure~\ref{fig0}.
Let us also mention that the the upper bound
$\overline{g_N}(\alpha):=\frac14\alpha^2$ for $g_N(\alpha)$
in \cite[Proposition~5]{MR} is incorrect: The error is explained below.

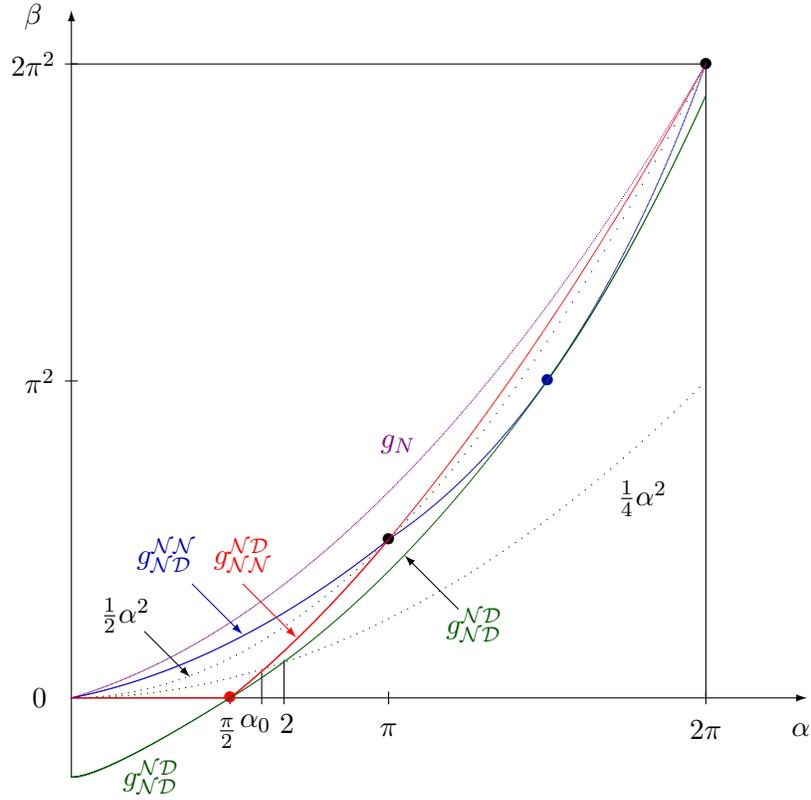
\begin{figure}[ht]
\centering
\begin{picture}(300,300)(-20,-30)
\unitlength=1.2pt 
\put(0,0){\line(1,0){230}}
\put(230,0){\vector(1,0){2}}
\put(230,-10){\makebox(0,0)[c]{$\alpha$}} 
\put(100,-2){\line(0,1){4}}
\put(100,-10){\makebox(0,0)[c]{$\pi$}}
\put(50,-2){\line(0,1){4}}
\put(49,-11){\makebox(0,0)[c]{$\frac\pi2$}}
\put(100,50){\makebox(0,0)[c]{$\bullet$}}
\put(200,200){\makebox(0,0)[c]{$\bullet$}}
\put(60,-2){\line(0,1){10}}
\put(58,-8){\makebox(0,0)[c]{$\alpha_0$}}
\put(67,-2){\line(0,1){14}}
\put(68,-8){\makebox(0,0)[c]{$2$}}
\put(200,-2){\line(0,1){4}}
\put(200,-10){\makebox(0,0)[c]{$2\pi$}}
\put(0,-25){\line(0,1){240}}
\put(0,215){\vector(0,1){2}}
\put(-12,215){\makebox(0,0)[c]{$\beta$}} 
\put(-10,0){\makebox(0,0)[c]{$0$}}
\put(-2,100){\line(1,0){4}}
\put(-10,100){\makebox(0,0)[c]{$\pi^2$}}
\put(-2,200){\line(1,0){4}}
\put(-12,200){\makebox(0,0)[c]{$2\pi^2$}}
\put(200,0){\line(0,1){200}}
\put(0,200){\line(1,0){200}}
\put(180,63){\makebox(0,0)[c]{$\frac14\alpha^2$}}
\put(122,28){\line(-1,1){15}}
\put(107,43){\vector(-1,1){2}}
\put(17,26){\makebox(0,0)[c]{$\frac12\alpha^2$}}
\put(23,21){\line(1,-1){13}}
\put(35.5,8.5){\vector(1,-1){2}}
\bezier100(0,0)(100,0)(200,100)
\bezier100(0,0)(100,0)(200,200)
\blue
\bezier500(0,0)(50,10)(100,50)
\bezier500(100,50)(160,90)(200,200)
\put(30,45){\makebox(0,0)[c]{$g^{\c{N}\c{N}}_{\c{N}\c{D}}$}}
\put(38,36){\line(1,-1){15}}
\put(52,22){\vector(1,-1){2}}
\put(150,100){\makebox(0,0)[c]{$\bullet$}}
\red
\put(0,0){\line(1,0){50}}
\bezier500(50,0)(75,20)(100,50)
\bezier500(100,50)(150,110)(200,200)
\put(54,45){\makebox(0,0)[c]{$g^{\c{N}\c{D}}_{\c{N}\c{N}}$}}
\put(54,36){\line(1,-1){15}}
\put(69,21){\vector(1,-1){2}}
\put(50,0){\makebox(0,0)[c]{$\bullet$}}
\refcol
\bezier500(0,-25)(10,-25)(50,0)
\bezier500(50,0)(100,30)(150,100)
\bezier500(150,100)(175,135)(200,190)
\put(25,-25){\makebox(0,0)[c]{$g^{\c{N}\c{D}}_{\c{N}\c{D}}$}}
\put(127,23){\makebox(0,0)[c]{$g^{\c{N}\c{D}}_{\c{N}\c{D}}$}}
\ecol
\bezier500(0,0)(100,30)(200,200)
\put(103,80){\makebox(0,0)[c]{$g_N$}}
\black
\end{picture}
   \caption{The case $\ID_0\cap\{1,2\}=\emptyset$.}
   \label{fig11}
\end{figure}

(iv)
The function $\hat g(\alpha):=\frac12\alpha^2$
is a good approximation of functions $g$ in Proposition~\ref{prop-twist}(ii) for $\alpha$ large,
see Table~\ref{tab1} and Figure~\ref{fig11}.
The functions $g^{\c{N}\c{D}}_{\c{N}\c{N}},g^{\c{N}\c{N}}_{\c{N}\c{D}}$ oscillate between $g_N$ and $g^{\c{N}\c{D}}_{\c{N}\c{D}}$,
they intersect each other whenever $\alpha=k\pi$, $k=1,2,\dots$,
and then their common values equal $\hat g(\alpha)$ (and also $g_N(\alpha)$ if $k$ is even).
Similarly,
$\min(g^{\c{N}\c{D}}_{\c{N}\c{N}}(\alpha),g^{\c{N}\c{N}}_{\c{N}\c{D}}(\alpha))=g^{\c{N}\c{D}}_{\c{N}\c{D}}(\alpha)$
if $\alpha=(k+\frac12)\pi$, $k=0,1,2,\dots$.
Similar behavior of functions $\tilde g(\alpha)=\frac14\alpha^2$ and 
$g_M$, $g^{\c{D}\c{N}}_{\c{N}\c{D}}$, $g^{\c{D}\c{D}}_{\c{N}\c{D}}$ 
can be observed in Figure~\ref{fig0}.
The formulas for functions $g$ in Proposition~\ref{prop-twist}(ii) 
can be used in the numerical computations of $g$, but they
also can be used in the study of the asymptotic or qualitative behavior
of $g$. For example,
they imply that 
$\lim_{\alpha\to0+}\frac{g_N(\alpha)}{\alpha^2}=1$,
$\lim_{\alpha\to\infty}(\hat g-g^{\c{N}\c{D}}_{\c{N}\c{D}})(\alpha)=0$,
$g^{\c{N}\c{D}}_{\c{N}\c{D}}$ is
$C^1\setminus C^2$ at $\alpha=2$, 
and $g_N$ is $C\setminus C^1$ at $\alpha=2k\pi$, $k=1,2\dots$.

\renewcommand{\arraystretch}{1} 
\begin{table}[ht]
\begin{center}
\begin{tabular}{|c|c|c|c|c|c|c|} 
\hline 
 \strut\lower2mm\vbox to 6.5mm{\vfill} $\alpha/\pi$ & $g_N(\alpha)/\pi^2$ &
 $g^{\c{N}\c{N}}_{\c{N}\c{D}}(\alpha)/\pi^2$ &
 $\hat g(\alpha)/\pi^2$ &
 $g^{\c{N}\c{D}}_{\c{N}\c{N}}(\alpha)/\pi^2$ &
 $g^{\c{N}\c{D}}_{\c{N}\c{D}}(\alpha)/\pi^2$ &
 $\Delta_{\hbox{\small max}}(\alpha)/\pi^2$
\\ 
\hline
\strut\vbox to 4mm{\vfill} 0 &  0 & 0 & 0 & 0 & -0.25 & 0.25\\
0.3& 0.0842& 0.0732& 0.045&  0.0000& -0.1222& 0.2064\\
0.5& 0.2137& 0.1679& 0.125&  0.0000&  0.0000& 0.2137\\
0.7& 0.3792& 0.2820& 0.245&  0.1826&  0.1533& 0.2258\\
1.0& 0.6717& 0.5000& 0.500&  0.5000&  0.4446& 0.2271\\
1.3& 1.0067& 0.8197& 0.845&  0.8663&  0.8129& 0.1938\\
1.5& 1.2549& 1.1032& 1.125&  1.1440&  1.1032& 0.1516\\
1.7& 1.5279& 1.4334& 1.445&  1.4558&  1.4305& 0.0973\\
2.0& 2.0000& 2.0000& 2.000&  2.0000&  1.9923& 0.0076\\
2.5& 3.2058& 3.1274& 3.125&  3.1225&  3.1225& 0.0832\\
3.0& 4.5759& 4.5000& 4.500&  4.5000&  4.4992& 0.0767\\  
3.5& 6.1596& 6.1248& 6.125&  6.1252&  6.1248& 0.0348\\
4.0& 8.0000& 8.0000& 8.000&  8.0000&  7.9999& 0.0001\\
      \hline
    \end{tabular}
  \end{center}
\kern2mm 
\setbox5=\hbox{Numerical values of functions $g$ and 
       $\Delta_{\hbox{\small max}}:=g_N-g^{\c{N}\c{D}}_{\c{N}\c{D}}$ if $\ID_0\cap\{1,2\}=\emptyset$.}
\caption{\box5}
\label{tab1}
\end{table}

(v)
Numerical computations determining the borderlines for stability
could be used also if we did not know the formulas for functions $g$ in Proposition~\ref{prop-twist}.
If $\beta_0<\beta_1$ and the problem
with parameters $(\alpha_0,\beta_0)$ or $(\alpha_0,\beta_1)$ is unstable or stable, respectively, 
then one can set $\beta_2:=(\beta_0+\beta_1)/2$ and
numerically solve the Jacobi equations with suitable initial conditions and
parameters $(\alpha_0,\beta_2)$ (by the Euler method, for example).
If that problem is stable or unstable, then one can
set $\beta_3:=(\beta_0+\beta_2)/2$ or $\beta_3:=(\beta_2+\beta_1)/2$, respectively, 
and solve the problem with parameters $(\alpha_0,\beta_3)$ etc.  
In fact, we used such general approach to compute the numerical values of functions $g_N$ and $g^{\c{N}\c{D}}_{\c{N}\c{D}}$ first,
and we verified a posteriori that the computed critical parameters correspond to the critical
values determined by Proposition~\ref{prop-twist}.

(vi) Let $u^0$ be a weak minimizer.
Then a straightfoward modification
of the proof of Proposition~\ref{prop-ws} shows that $u^0$ 
is also a strong minimizer.
In fact, assume first that there exist $v^k\in W^{1,2}_\c{D}$
such that $r_k:=\|v^k\|_{1,2}\to0$
and $\Phi(u^0+v^k)<\Phi(u^0)$. Since $\Phi\in C^1(W^{1,2})$
is weakly sequentially lower semicontinuous,
we can find a minimizer $u^k$ of $\Phi$ in $\{u\in u^0+W^{1,2}_\c{D}:\|u-u^0\|_{1,2}\leq r_k\}$
and Lagrange multipliers $\lambda_k\leq0$ such that 
$\Phi'(u^k)h=\lambda_k\Theta'(u^k)h$ for any $h\in W^{1,2}_\c{D}$,
where $\Theta(u)=\|u-u^0\|_{1,2}^2$.
The arguments in \cite[Section~2.6]{Ces} guarantee that $u^k\in C^2$ and $u^k$
satisfy the Euler equations
$(F^k_p(x))'=F^k_u(x)$, where $F^k_p(x):=F_p(\lambda_k,x,u^k(x),(u^k)'(x))$ (similarly $F^k_u$)
and $F(\lambda,x,u,p):=f(x,u,p)-\lambda(|p-(u^0)'(x)|^2+|u-u^0(x)|^2)$.
These equations, the particular form of $f,u^0$, the positive definiteness of $F^k_{pp}$
and the convergence $u^k\to u^0$ in $W^{1,2}$ guarantee 
that $\{u^k\}$ is a Cauchy sequence in $W^{2,1}$, hence in $C^1$, 
thus $u^k\to u^0$ in $C^1$. However, this contradicts our assumption
that $u^0$ is a weak minimizer.
Consequently, $u^0$ is a local minimizer in $u^0+W^{1,2}_\c{D}$.
Next assume that there exist $v^k\in C^1_\c{D}$ such that $\|v^k\|_C\to0$
and $\Phi(u^0+v^k)<\Phi(u^0)$. 
Then it is not difficult to show that there exists $c>0$ such that
$0>\Phi(u^0+v^k)-\Phi(u^0)\geq c\|v^k\|^2_{1,2}+o(1)$,
hence $\|v^k\|_{1,2}\to0$, which yields a contradiction
and concludes the proof.
\qed
\end{remark}

\begin{proof}[Proof of Proposition~\ref{prop-twist}]
Notice that $u^0$ is a critical point of $\Phi$ for any choice of $\IN_0,\IN_1\subset\{1,2\}$.
By Proposition~\ref{prop-psi}, 
we have to determine the positivity of functional $\Psi$ in $W^{1,2}_{\c{D}}$.
We have $\Psi(h)=\Psi_1(h_1,h_2)+\Psi_2(h_3)$, where
$$\Psi_1(h_1,h_2)=A\int_0^1\bigl((h'_1)^2+(h'_2)^2-2\alpha h_2'h_1+\beta(h_1^2+h_2^2)\bigr)\,dx,\quad
\Psi_2(h_3)=C\int_0^1(h'_3)^2\,dx.$$
Since the positivity of $\Psi$ does not change if we replace $\alpha$ by
$-\alpha$ (consider $-h_1$ instead of $h_1$), we may assume $\alpha\geq0$.
Since the case $\alpha=0$ is trivial, we assume $\alpha>0$.
Since $\Psi_2$ is positive definite in $W^{1,2}_0([0,1])$,
it is sufficient to study the positivity of the functional
\begin{equation} \label{tildePsi}
\tilde\Psi(h_1,h_2):=\frac1{2A}\Psi_1(h_1,h_2)
=\frac12\int_0^1\bigl((h'_1)^2+(h'_2)^2-2\alpha h_2'h_1+\beta(h_1^2+h_2^2)\bigr)\,dx
\end{equation}
in the space
\begin{equation} \label{tildeW}
\tilde W_\c{D}:=\{h\in W^{1,2}([0,1],\R^2): h_i(j)=0 \hbox{ for } i\in\ID_j,\ i=1,2,\ j=0,1\}.
\end{equation}
In fact, $\Psi$ is positive definite (or semidefinite, resp.) in $W^{1,2}_\c{D}$ 
if and only if $\tilde\Psi$ is positive definite (or semidefinite, resp.) in $\tilde W_\c{D}$.
Therefore, in what follows, we will apply the Jacobi theory from Section~\ref{sec-Jacobi} 
to the functional $\tilde\Psi$ with $\alpha>0$.
Notice that the assumptions in Theorem~\ref{thm-Jacobi} depend 
only on the corresponding functional $\Psi$, and the conclusions can also
be formulated in terms of $\Psi$, see  \eqref{ih0}, \eqref{iiPsi0}.
We will use Theorem~\ref{thm-Jacobi} in this way.
More precisely, we will use assertions \eqref{ih0}, \eqref{iiPsi0}
(with $\Psi$ and $W^{1,2}_\c{D}$ replaced by $\tilde\Psi$ and $\tilde W_\c{D}$, respectively)
to determine the positivity of $\tilde\Psi$ (hence the positivity of $\Psi$)
and then we will use Proposition~\ref{prop-psi} (with $\Psi(h)=\Psi(h_1,h_2,h_3)$) 
to conclude that $u^0$ is (or is not) a minimizer of $\Phi$.  

Notice that the index sets for functional $\tilde\Psi$ 
satisfy $\IDr_j=\ID_j\cap\{1,2\}$
and $\INr_j=\IN_j\cap\{1,2\}=\IN_j$ for $j=1,2$, hence we will use
the notation $\IN_j$ instead of $\INr_j$.
Similarly, the corresponding operators $\tilde{\c{B}}_i$, $i=1,2$ (cf.~\eqref{BA}),
satisfy $\tilde{\c{B}}_i(h_1,h_2)=\c{B}_i(h_1,h_2,0)$ for $i=1,2$, 
and --- without fearing confusion --- we will use the notation
$\c{B}_ih$ instead of $\tilde{\c{B}}_ih$ 
and $\c{B}h:=(\c{B}_1h,\c{B}_2h)$
if $h=(h_1,h_2)$ and $i=1,2$.
The same applies to operators $\c{C}_i$ and $\c{A}_i$.
Since 
\begin{equation} \label{BiCi}
 \c{B}_1h = h_1', \quad \c{B}_2h =-\alpha h_1+h_2', \quad
 \c{C}_1h =\beta h_1-\alpha h_2', \quad \c{C}_2h = \beta h_2,
\end{equation}
the corresponding system of Jacobi equations is
\begin{equation} \label{J2}
\left.\begin{aligned}
 h_1''+\alpha h_2'-\beta h_1 &=0, \\
 h_2''-\alpha h_1'-\beta h_2 &=0,
\end{aligned}\quad\right\}\quad \hbox{in}\ (0,1),
\end{equation}
and the initial conditions for $h^{(1)},h^{(2)}$ in Theorem~\ref{thm-Jacobi} (with $N=2$) are 
$h_i(0)=0$ if  $i\in\IDr_0$ and $i=1,2$, $h_1'(0)=0$ if $1\in\IN_0$, 
and $h_2'(0)=\alpha h_1(0)$ if $2\in\IN_0$.

The existence of continuous borderline functions $g$ follows from the form of $\tilde\Psi$.
Notice that if the index sets $\IDr_0$ and $\IDr_1$ are nonempty, then 
$h_1h_2(0)=h_1h_2(1)=0$ for any $h\in \tilde W_\c{D}$,
hence 
\begin{equation} \label{h1h2}
\int_0^1h_2'h_1\,dx=-\int_0^1h_1'h_2\,dx.
\end{equation}
Identity \eqref{h1h2} shows that the value of $\tilde\Psi$ does not change
if we replace $h_1$ with $h_2$ and $\alpha$ with $-\alpha$.
In general, the value of $\tilde\Psi$ does not change if we replace
$h_i$ with $\tilde h_i(x)=h_i(1-x)$ and $\alpha$ with $-\alpha$.
These two observations guarantee \eqref{cases1} and \eqref{cases2}. 

Let us first consider the cases in Proposition~\ref{prop-twist}(i),
i.e. $\IDr_0\ne\emptyset\ne\IDr_1$. 
Then 
\eqref{h1h2} guarantees
$\int_0^12h_2'h_1\,dx=\int_0^1(h_2'h_1-h_1'h_2)\,dx$
and the Cauchy inequality implies that 
\begin{equation}\label{Psipos1}
\tilde\Psi \hbox{ is positive definite if } \alpha^2<4\beta.
\end{equation}
Hence it is sufficient to study the case $\alpha^2\geq4\beta$. 

Case $\bigl(\strut^{\c{D}\c{D}}_{\c{D}\c{D}}\bigr)$ 
has already been solved in \cite[Proposition 3]{MR}, 
but Theorem~\ref{thm-Jacobi} enables us to 
to show $g_D(\alpha)=\frac{\alpha^2}{4}-\pi^2$ in a simpler way.
Assume $\alpha^2>4\beta$.
We can set
$h^{(1)}(x)=(\sin\xi_1x-\sin\xi_2x,\cos\xi_1x-\cos\xi_2x)$
and
$h^{(2)}(x)=(-\cos\xi_1x+\cos\xi_2x,\sin\xi_1x-\sin\xi_2x)$,
where
$\xi_{1,2}=-\frac12\alpha\pm\gamma$.
The function $D$ in Theorem~\ref{thm-Jacobi} satifies
$D(x)=2-2\cos(\xi_1-\xi_2)x$, hence $D\ne0$ in $(0,1]$
if and only if $|\xi_1-\xi_2|<2\pi$, i.e.~if $\beta>g_D(\alpha)$.
Consequently, if $\beta>g_D(\alpha)$, then $u^0$ is a strict weak minimizer
(this remains true also if $4\beta=\alpha^2$ due to
the monotonicity of $\tilde\Psi$ with respect to $\beta$),
and if $\beta<g_D(\alpha)$, then $u^0$ is not a weak minimizer.

The remaining cases in Proposition~\ref{prop-twist}(i) are
$\bigl(\strut^{\c{D}\c{D}}_{\c{N}\c{D}}\bigr)$,
$\bigl(\strut^{\c{D}\c{N}}_{\c{N}\c{D}}\bigr)$, and $\bigl(\strut^{\c{D}\c{D}}_{\c{N}\c{N}}\bigr)$.
Assume $\alpha^2>4\beta$.
Since $\IN_0=\{2\}$, 
the initial conditions for $h^{(1)},h^{(2)}$ in Theorem~\ref{thm-Jacobi} are
$h_1(0)=0$ and $h_2'(0)=0$.
One can easily check that we can set
$h^{(i)}(x):=(\sin\xi_ix,\cos\xi_ix)$, $i=1,2$, where
$\xi_{1,2}:=-\frac12\alpha\pm\gamma$.  
The function $D$ in Theorem~\ref{thm-Jacobi} satifies
$$D(x)=\sin(\xi_1-\xi_2)x =\sin2\gamma x=\sin\sqrt{\alpha^2-4\beta}\,x,$$
hence 
\begin{equation}\label{D0i}
 \hbox{if \ }\alpha^2-4\beta>\pi^2,\hbox{\ then }
D(x)=0 \hbox{\ for some }x\in(0,1),
\end{equation}
\begin{equation}\label{D0ii}
 \hbox{if \ }0<\alpha^2-4\beta<\pi^2,\hbox{\ then } 
D(x)\ne0 \hbox{\ in }(0,1].\hbox to26mm{\hfill}
\end{equation}
Theorem~\ref{thm-Jacobi}(i) 
(more precisely, assertion \eqref{ih0})
and \eqref{D0i} imply that 
\begin{equation}\label{Psineg1}
\tilde\Psi \hbox{\ is not positive semidefinite if
}\alpha^2-4\beta>\pi^2.
\end{equation}

Let $\IN_1=\emptyset$.
If $0<\alpha^2-4\beta<\pi^2$, then \eqref{D0ii} and
Theorem~\ref{thm-Jacobi}(ii) (more precisely, assertion \eqref{iiPsi0})
guarantee that $\tilde\Psi$ is positive definite.
If $0=\alpha^2-4\beta<\pi^2$ and we replace $\beta$ by $\tilde\beta:=\beta-\eps$
with $\eps>0$ small, then $0<\alpha^2-4\tilde\beta<\pi^2$,
hence the modified functional $\tilde\Psi^{\tilde\beta}$ (with $\beta$ replaced by $\tilde\beta$) 
is positive definite, and the monotonicity of $\tilde\Psi$ with respect to $\beta$
implies that $\tilde\Psi$ is positive definite as well.
These facts together with \eqref{Psipos1} and \eqref{Psineg1} imply
$g^{\c{D}\c{D}}_{\c{N}\c{D}}(\alpha)=\frac{\alpha^2}4-\frac{\pi^2}4$.

If $\IN_1=\{2\}$ and $\alpha^2>4\beta$, 
then $H_{\c{D},b}=\{\tilde h\in\hbox{span}(h^{(1)},h^{(2)}):\tilde h_1(1)=0\}$
is spanned by 
$h:=\sin\xi_2h^{(1)}-\sin\xi_1h^{(2)}$.
We have 
$$ 
B:=\c{B}h(1)\cdot h(1) 
=h'_2(1)h_2(1)=(\xi_2-\xi_1)\sin(\xi_2-\xi_1)\sin\xi_1\sin\xi_2$$
and, assuming $\alpha\in[(2k-1)\pi,(2k+1)\pi]$, $k=0,1,2,\dots$, $\alpha>0$, we have 
$B>0$ or $B<0$
if and only if 
$\beta$ is greater or less than $k\pi(\alpha-k\pi)$,
respectively. Notice that 
\begin{equation} \label{ineq1alphabeta}
\alpha^2/4\geq k\pi(\alpha-k\pi)\geq(\alpha^2-\pi^2)/4.
\end{equation}
These facts, Theorem~\ref{thm-Jacobi}(ii) and \eqref{Psipos1} 
imply that $\tilde\Psi$ is positive definite if $\beta>k\pi(\alpha-k\pi)$,
$\beta\ne\alpha^2/4$. The assumption $\beta\ne\alpha^2/4$ can be removed
by the same argument as above (by considering $\tilde\beta=\beta-\eps$).
If $\beta<k\pi(\alpha-k\pi)$, then    
$\alpha^2>4\beta$ due to \eqref{ineq1alphabeta},
hence $B<0$ and Theorem~\ref{thm-Jacobi}(i) imply that $\tilde\Psi$
is not positive semidefinite. 
Consequently, the formula for $g_M=g^{\c{D}\c{D}}_{\c{N}\c{N}}$ in \eqref{g4} is true.

If $\IN_1=\{1\}$, then we can use the same arguments as in the case $\IN_1=\{2\}$
to show that the formula for $g^{\c{D}\c{N}}_{\c{N}\c{D}}$ in \eqref{g4} is true.
In particular, if $\alpha^2>4\beta$, then
$H_{\c{D},b}=\{\tilde h\in\hbox{span}(h^{(1)},h^{(2)}):\tilde h_2(1)=0\}$
is spanned by 
$h:=\cos\xi_2h^{(1)}-\cos\xi_1h^{(2)}$
and we have 
$$ 
B:=\c{B}h(1)\cdot h(1) 
=h'_1(1)h_1(1)=(\xi_1-\xi_2)\sin(\xi_1-\xi_2)\cos\xi_1\cos\xi_2,$$
hence assuming $\alpha\in[2k\pi,2(k+1)\pi]$, $k=0,1,2,\dots$,
we obtain $B>0$ or $B<0$
if and only if $\beta$ is greater or less than $(k+\frac12)\pi(\alpha-(k+\frac12)\pi)$,
respectively. 

Next consider the cases in Proposition~\ref{prop-twist}(ii), i.e.
$\bigl(\strut^{\c{N}\c{D}}_{\c{N}\c{D}}\bigr)$,
$\bigl(\strut^{\c{N}\c{N}}_{\c{N}\c{D}}\bigr)$, 
$\bigl(\strut^{\c{N}\c{D}}_{\c{N}\c{N}}\bigr)$
and $\bigl(\strut^{\c{N}\c{N}}_{\c{N}\c{N}}\bigr)$.
Since $\IN_0=\{1,2\}$, 
the initial conditions for $h^{(1)},h^{(2)}$ in Theorem~\ref{thm-Jacobi} are
$h_1'(0)=0$ and $h_2'(0)=\alpha h_1(0)$.
We will distinguish the following four subcases:
\begin{itemize}
\item[(ii-1)] $\beta=\frac12\alpha^2$,
\item[(ii-2)] $\beta=\frac14\alpha^2$,
\item[(ii-3)] $\beta>\frac14\alpha^2$ and $\beta\ne\frac12\alpha^2$,
\item[(ii-4)] $\beta<\frac14\alpha^2$.
\end{itemize}

(ii-1) Assume that $\beta=\frac12\alpha^2$.
We will show that 
$\tilde\Psi$ is positive definite
(hence $u^0$ is a strict weak minimizer) 
in case $\bigl(\strut^{\c{N}\c{D}}_{\c{N}\c{D}}\bigr)$ and
$\tilde\Psi$ is not positive semidefinite (hence
$u^0$ is not a weak minimizer) 
in case $\bigl(\strut^{\c{N}\c{N}}_{\c{N}\c{N}}\bigr)$ if $\alpha\ne2k\pi$.
In addition, in case $\bigl(\strut^{\c{N}\c{N}}_{\c{N}\c{D}}\bigr)$,
$u^0$ is or is not a weak minimizer if $\alpha\in((2k-1)\pi,2k\pi)$ or $\alpha\in(2k\pi,(2k+1)\pi)$,
respectively, and the opposite is true in case $\bigl(\strut^{\c{N}\c{D}}_{\c{N}\c{N}}\bigr)$.

Recall that $\delta=\alpha/2$. If we set
$$\begin{aligned}
h^{(1)}(x)&:=(e^{\delta x}(\cos(\delta x)-\sin(\delta x)),e^{\delta x}(\cos(\delta x)+\sin(\delta x))),\\
h^{(2)}(x)&:=(e^{-\delta x}(\cos(\delta x)+\sin(\delta x)),e^{-\delta x}(-\cos(\delta x)+\sin(\delta x))),
\end{aligned}$$
then we obtain $D\equiv-2$, hence 
$\tilde\Psi$ is positive definite 
in case $\bigl(\strut^{\c{N}\c{D}}_{\c{N}\c{D}}\bigr)$ due to Theorem~\ref{thm-Jacobi}(ii).

Considering case $\bigl(\strut^{\c{N}\c{N}}_{\c{N}\c{N}}\bigr)$, 
one can check that the matrix $\m{A}=(a_{kl})$ in Remark~\ref{rem-Jacobi}(ii)
satisfies 
$$a_{11}=4\delta e^{2\delta}\sin^2\delta,\quad a_{22}=-4\delta e^{-2\delta}\sin^2\delta,\quad
a_{12}=a_{21}=-4\delta\sin\delta\cos\delta.$$
If $\delta\ne k\pi$, then 
choosing $\xi:=(0,1)$ and $h:=\sum_{k=1}^2\xi_kh^{(k)}=h^{(2)}\in H_{\c{D},1}=H$
we obtain $\c{B}h(1)\cdot h(1)=\m{A}\xi\cdot\xi=a_{22}<0$,
i.e.\ 
$\tilde\Psi$ is not positive semidefinite due to Theorem~\ref{thm-Jacobi}(i).
Notice also that $\c{B}h(0)=0$, hence
\begin{equation} \label{PsiNNNN}
 \tilde\Psi(h)= \c{B}h\cdot h\Bigl|_0^1\ <0.
\end{equation}
If $\delta=k\pi$, then $\m{A}=0$ (degenerate case).
Already these facts contradict \cite[Proposition 5]{MR} which claims the stability for
$\beta>\frac14\alpha^2$. In fact, the authors of \cite{MR} mention
in their proof that ``We have not used any integration by parts \dots'',
but they seem to use \cite[(35)--(37)]{MR}, and \cite[(35)]{MR}
does use an integration by parts requiring the boundary conditions $h_1h_2(0)=h_1h_2(1)$.

In case $\bigl(\strut^{\c{N}\c{D}}_{\c{N}\c{N}}\bigr)$ we set
$$ h:=e^{-\delta}(\cos\delta+\sin\delta)h^{(1)}-e^\delta(\cos\delta-\sin\delta)h^{(2)}. $$
Since at least one of the numbers $h^{(1)}_1(1)$ and $h^{(2)}_1(1)$ is non-zero,
we have  $\dim H_{\c{D},1}\leq1$.
Since $h_1(1)=0$, we obtain $H_{\c{D},1}=\hbox{\rm span}(h)$, and
$$\c{B}h(1)\cdot h(1)=\c{B_2}h(1)\cdot h_2(1)=(-\alpha h_1+h_2')(1)\cdot h_2(1)=2\alpha\sin\alpha$$
due to $h_2(1)=2$ and $h_2'(1)=\alpha\sin\alpha$.
Consequently, $\c{B}h(1)\cdot h(1)>0$ if $\alpha\in(2k\pi,(2k+1)\pi)$
and $\c{B}h(1)\cdot h(1)<0$ if $\alpha\in((2k-1)\pi,2k\pi)$,
so that our assertion follows from Theorem~\ref{thm-Jacobi}(ii) and
Theorem~\ref{thm-Jacobi}(i), respectively.  

Similarly, in case $\bigl(\strut^{\c{N}\c{N}}_{\c{N}\c{D}}\bigr)$ we set
$$ h:=e^{-\delta}(\cos\delta-\sin\delta)h^{(1)}+e^\delta(\cos\delta+\sin\delta)h^{(2)}. $$
Then $h_2(1)=0$ and $H_{\c{D},1}=\hbox{\rm span}(h)$;
\begin{equation}\label{BhhNNND}
\c{B}h(1)\cdot h(1)=\c{B}_1h(1)\cdot h(1)=h_1'(1)h_1(1)=-2\alpha\sin\alpha
\end{equation} 
due to $h_1(1)=2$ and $h_1'(1)=-\alpha\sin\alpha$.
The rest of the proof is the same as in case $\bigl(\strut^{\c{N}\c{D}}_{\c{N}\c{N}}\bigr)$.
Notice also that (similarly as in the case of \eqref{PsiNNNN}), \eqref{BhhNNND} implies 
\begin{equation} \label{PsiNNND}
 \tilde\Psi(h)= \c{B}h\cdot h\Bigl|_0^1\ <0
\end{equation}
provided $\alpha\in(2k\pi,(2k+1)\pi)$.

(ii-2)
Assume that $\beta=\frac14\alpha^2$.
Set $\xi:=-\frac12\alpha$ and
$$ \begin{aligned}
h^{(1)}(x)&:=(\sin(\xi x)-\xi x\cos(\xi x),\cos(\xi x)+\xi x\sin(\xi x)), \\
h^{(2)}(x)&:=(\cos(\xi x)-\xi x\sin(\xi x),-\sin(\xi x)-\xi x\cos(\xi x)). \\
\end{aligned} $$ 
Notice that the function $D$ in Theorem~\ref{thm-Jacobi} satisfies
$D(x)=\xi^2 x^2-1$, hence $D<0$ in $[0,1]$ if $\alpha<2$,
and $D(x)=0$ for some $x\in(0,1)$ if $\alpha>2$.
This shows that 
$\frac14\alpha^2<g^{\c{N}\c{D}}_{\c{N}\c{D}}(\alpha)
\leq\min\bigl(g^{\c{N}\c{N}}_{\c{N}\c{D}}(\alpha),g^{\c{N}\c{D}}_{\c{N}\c{N}}(\alpha),g_N(\alpha)\bigr)$
if $\alpha>2$,
i.e.~$u^0$ cannot be a weak minimizer in any case.

Let $\alpha<2$. Then $u^0$ is a strict weak minimizer in case $\bigl(\strut^{\c{N}\c{D}}_{\c{N}\c{D}}\bigr)$.
Next consider case $\bigl(\strut^{\c{N}\c{N}}_{\c{N}\c{N}}\bigr)$.
If $\beta=\alpha^2/2$, then \eqref{PsiNNNN} implies that $\tilde\Psi$
is not positive semidefinite. The monotonicity of $\tilde\Psi$ with respect to $\beta$
shows that $\tilde\Psi$ cannot be positive semidefinite if $\beta=\alpha^2/4$ either,
hence $u_0$ is not a weak minimizer.
The same arguments show that $u_0$ is not a weak minimizer 
in case $\bigl(\strut^{\c{N}\c{N}}_{\c{N}\c{D}}\bigr)$, see~\eqref{PsiNNND}.
It remains to consider case 
$\bigl(\strut^{\c{N}\c{D}}_{\c{N}\c{N}}\bigr)$.
Set
$$ h:=(\cos\xi-\xi\sin\xi)h^{(1)}-(\sin\xi-\xi\cos\xi)h^{(2)}, $$
so that $h_1(1)=0$.
Then the restriction $\alpha<2$ implies $h_2(1)=1-\xi^2>0$.
Since $h_2'(1)=-\xi^2+\xi\sin(2\xi)$, we see that $h_2'(1)h_2(1)>0$ only if $\alpha<\alpha_0$,
where $\alpha_0$ is defined by $\alpha_0=2\sin\alpha_0$ ($\alpha_0\approx0.6\pi$). 

(ii-3) 
Assume $\beta>\frac14\alpha^2$, $\beta\ne\frac12\alpha^2$, and set
$$
\vp(x):=e^{\gamma x}(\gamma^2-\delta^2),\ \ \psi_\pm(x):=e^{-\gamma x}(\gamma\pm\delta)^2. 
$$
Then we can take
$$ \begin{aligned}
h^{(1)}(x)&:=[(\vp(x)+\psi_+(x))(\cos(\delta x)+\sin(\delta x)),
             (\vp(x)+\psi_+(x))(-\cos(\delta x)+\sin(\delta x))], \\
h^{(2)}(x)&:=[(\vp(x)+\psi_-(x))(\cos(\delta x)-\sin(\delta x)),
             (\vp(x)+\psi_-(x))(\cos(\delta x)+\sin(\delta x))], 
\end{aligned}
$$
and an easy computation yields
\begin{equation} \label{Dplus}
D(x) = 4(\gamma^2-\delta^2)\Bigl((\gamma^2-\delta^2)\cosh(2\gamma x)+\gamma^2+\delta^2\Bigr).
\end{equation}
The function $D$ does not vanish in $(0,1]$ if and only if $\gamma>\delta$ (i.e.~$\beta>\frac12\alpha^2$), 
or $\gamma<\delta$ and $\cosh(2\gamma)<\frac{\gamma^2+\delta^2}{\delta^2-\gamma^2}$.
The last inequality 
can be written in the form
\begin{equation} \label{Dpos}
(\alpha^2-2\beta)\cosh(2\gamma)<2\beta.
\end{equation}
In case $\bigl(\strut^{\c{N}\c{N}}_{\c{N}\c{N}}\bigr)$, one has to consider
the numbers $a_{kl}$ in Remark~\ref{rem-Jacobi}(ii):
$$ \begin{aligned}
a_{11}&=2\gamma(\vp^2-\psi_+^2)(1)+2\delta(\vp+\psi_+)^2(1)\cos(2\delta), \\
a_{22}&=2\gamma(\vp^2-\psi_-^2)(1)-2\delta(\vp+\psi_-)^2(1)\cos(2\delta), \\
a_{12}=a_{21}&=-2\delta(\vp+\psi_+)(\vp+\psi_-)(1)\sin(2\delta).
\end{aligned} $$
If $\gamma>\delta$ (i.e.~$\beta>\frac12\alpha^2$), then
$$ a_{11}(\gamma+\delta)^{-2}+a_{22}(\gamma-\delta)^{-2}=8(\gamma^2+\delta^2)(\gamma-\theta\delta\cos(2\delta))\sinh(2\gamma)>0,$$
hence the matrix $\m{A}$ is positive definite if and only if $a_{11}a_{22}>a_{12}^2$, which is equivalent to
\begin{equation} \label{Apos}
(1-\theta^2)\cosh(2\gamma)+\theta^2\cos(2\delta)>1. 
\end{equation}
We used the assumption $\beta>\frac12\alpha^2$ in order to derive \eqref{Apos},
but this is not restrictive, since we know that $u^0$ can only be a weak minimizer 
of our problem in case $\bigl(\strut^{\c{N}\c{N}}_{\c{N}\c{N}}\bigr)$ when $\beta>\frac12\alpha^2$.
Hence in this case the condition \eqref{Apos} determines the domain of stability.

In cases
$\bigl(\strut^{\c{N}\c{N}}_{\c{N}\c{D}}\bigr)$ and $\bigl(\strut^{\c{N}\c{D}}_{\c{N}\c{N}}\bigr)$,
we set
$$ h:=(\vp(1)+\psi_-(1))(\cos\delta+\sin\delta)h^{(1)}
    +(\vp(1)+\psi_+(1))(\cos\delta-\sin\delta)h^{(2)} $$
and
$$ h:=(\vp(1)+\psi_-(1))(\cos\delta-\sin\delta)h^{(1)}
    -(\vp(1)+\psi_+(1))(\cos\delta+\sin\delta)h^{(2)}, $$
respectively.
Then $h_2(1)=0$, $h_1(1)=D(1)$,
$$ \c{B}h(1)\cdot h(1) = h_1'h_1(1)
=4\gamma(\gamma^2-\delta^2)D(1)\bigl((\gamma^2-\delta^2)\sinh(2\gamma)-2\gamma\delta\sin(2\delta)\bigr), $$
and $h_1(1)=0$, $h_2(1)=-D(1)$,
$$ \c{B}h(1)\cdot h(1) = h_2'h_2(1)
=4\gamma(\gamma^2-\delta^2)D(1)\bigl((\gamma^2-\delta^2)\sinh(2\gamma)+2\gamma\delta\sin(2\delta)\bigr), $$
respectively, where $D$ is as in \eqref{Dplus}.
Consequently, assuming that $D$ does not vanish in $[0,1]$ (i.e.~\eqref{Dpos} is true),
the stability conditions are
\begin{equation} \label{NNNDplus}
 (\gamma^2-\delta^2)\sinh(2\gamma)-2\gamma\delta\sin(2\delta)>0 
\end{equation}
and
\begin{equation} \label{NDNNplus} 
 (\gamma^2-\delta^2)\sinh(2\gamma)+2\gamma\delta\sin(2\delta)>0, 
\end{equation}
respectively.
Notice that if $\beta=\frac12\alpha^2$ (hence $\gamma=\delta$), 
then \eqref{NNNDplus} and \eqref{NDNNplus}
are equivalent to the corresponding stability conditions in case (ii-1).

(ii-4)
If $\beta<\frac14\alpha^2$, then we can set
$$\begin{aligned}
h^{(1)}(x)&:=(\xi_2\sin(\xi_1x)-\xi_1\sin(\xi_2x),\xi_2\cos(\xi_1x)-\xi_1\cos(\xi_2x)), \\
h^{(2)}(x)&:=(\xi_1\cos(\xi_1x)-\xi_2\cos(\xi_2x),-\xi_1\sin(\xi_1x)+\xi_2\sin(\xi_2x)),
\end{aligned}$$
where $\xi_{1,2}=-\frac12\alpha\pm\gamma$, 
and we obtain
\begin{equation} \label{Dminus}
 D(x)=-2\beta+(\alpha^2-2\beta)\cos(2\gamma x).
\end{equation}
If $\alpha^2-4\beta\geq\pi^2$, then $D$ changes sign in $[0,1]$. 
Hence the condition $D>0$ in $[0,1]$ is equivalent to
\begin{equation} \label{DDpos}
\alpha^2-4\beta<\pi^2 \quad\hbox{and}\quad
(\alpha^2-2\beta)\cos(2\gamma)>2\beta.
\end{equation} 
It is not difficult to check (cf.~case (ii-2))
that if $\alpha<2$ or $\alpha>2$, then 
\eqref{DDpos}
or \eqref{Dpos}, respectively, is the (essentially optimal) sufficient condition for the stability in our problem 
in case  $\bigl(\strut^{\c{N}\c{D}}_{\c{N}\c{D}}\bigr)$.
If $\alpha=2$, then that sufficient condition is $\beta>1$.

Case (ii-2) shows that it remains to
consider only case $\bigl(\strut^{\c{N}\c{D}}_{\c{N}\c{N}}\bigr)$ and $\alpha<\alpha_0$.
Take
$$ h:=(\xi_1\cos\xi_1-\xi_2\cos\xi_2)h^{(1)}-(\xi_2\sin\xi_1-\xi_1\sin\xi_2)h^{(2)}. $$
Then $h_1(1)=0$, $h_2(1)=-D(1)$ (where $D$ is as in \eqref{Dminus}), and
$$ h_2'(1) = (\xi_1^2\sin\xi_2\cos\xi_1-\xi_2^2\sin\xi_1\cos\xi_2)(\xi_2-\xi_1).$$
Assuming $D>0$ in $[0,1]$ (i.e.~\eqref{DDpos}), the condition $h_2'h_2(1)>0$
is equivalent to
\begin{equation} \label{NDNNplus2}
  \xi_1^2\sin\xi_2\cos\xi_1>\xi_2^2\sin\xi_1\cos\xi_2.
\end{equation}
Since $\xi_1=0$ if $\beta=0$, \eqref{NDNNplus2} can only be true if $\beta>0$.
It is not difficult to see that 
$g^{\c{N}\c{D}}_{\c{N}\c{N}}(\alpha)=0$ for $\alpha\leq\frac12\pi$
and $g^{\c{N}\c{D}}_{\c{N}\c{N}}(\alpha_0)=\frac14\alpha^2_0$.
If $\alpha>\alpha_0$, then \eqref{NDNNplus} determines $g^{\c{N}\c{D}}_{\c{N}\c{N}}(\alpha)$.

The formulas for functions $g$ in Proposition~\ref{prop-twist}(ii) follow from the stability conditions
\eqref{Dpos},\eqref{Apos},\allowbreak\eqref{NNNDplus},\eqref{NDNNplus},\eqref{DDpos},\eqref{NDNNplus2}.
\end{proof}

\begin{remark} \label{rem-iv} \rm
Consider case $\bigl(\strut^{\c{D}\c{D}}_{\c{N}\c{N}}\bigr)$.
We have $g^{\c{D}\c{D}}_{\c{N}\c{N}}(\alpha)=g_M(\alpha)
>g^{\c{D}\c{D}}_{\c{N}\c{D}}(\alpha)$ except for $\alpha=\alpha_k:=(2k-1)\pi$, $k=1,2,\dots$.
If $\alpha=\alpha_k$ and $\beta=g_M(\alpha)=g^{\c{D}\c{D}}_{\c{N}\c{D}}(\alpha)$, 
then the function $D$ in Theorem~\ref{thm-Jacobi} satisfies 
$D\ne 0$ in $(0,1)$, $D(1)=0$, 
hence condition \eqref{cond-not} cannot be satisfied
(otherwise \eqref{ih0} would imply $\tilde\Psi(\bar h)<0$ for some $\bar h\in\tilde W_\c{D}$,
so that $\tilde\Psi(\bar h)<0$ also if $\beta$ is slightly greater than
$g_M(\alpha)$, which is a contradiction).
For example, if $k=2$ (i.e.~$\alpha=3\pi$, $\beta=2\pi^2$), then our proof shows that
$H_0$ is spanned by $h(x):=(-\sin(\pi x)-\sin(2\pi x),\cos(\pi x)+\cos(2\pi x))$
and $\c{B}_2h(1)=h_2(1)=h_1(1)=0$ which violates \eqref{cond-not}.
This degeneracy seems to be also responsible for the non-smooth behavior of
$g_M$ at $\alpha=\alpha_k$.
\qed
\end{remark}

\section{Field of extremals}
\label{sec-field}

In this section we modify the Weierstrass theory
to provide necessary and sufficient conditions for weak, strong and global minimizers.
Recall that $B_\eps:=\{\xi\in\R^N:|\xi|<\eps\}$.

\begin{definition} \label{def-field} \rm
Let $f\in C^2$, $\tilde\eps>0$, and let $u^0\in C^2$ be an extremal.
The image $\c{P}$ of
a $C^1$-diffeomorphism $P:[a,b]\times B_{\tilde\eps}\to [a,b]\times\R^N:(x,\alpha)\mapsto(x,\vp(x,\alpha))$
is called a {\it field of extremals for $u^0$} if 
$\vp_x\in C^1$, $\vp(\cdot,\alpha)$
is an extremal for each $\alpha$, and $\vp(\cdot,0)=u^0$.
The {\it slope} of the field of extremals $\c{P}$ is defined as
$\psi:\c{P}\to\R^N:(x,v)\mapsto \vp_x(x,\alpha(x,v))$,
where $\alpha(x,v)$ is defined by 
$\vp(x,\alpha(x,v))=v$. 
\end{definition}

It is known that in the case of the Dirichlet boundary conditions,
the existence of a field of extremals 
$\vp(x,\alpha)$
satisfying the self-adjointness condition \eqref{self-adjoint},
and the nonnegativity of the excess function
$$ E(x,u,p,q):=f(x,u,q)-f(x,u,p)-(q-p)\cdot f_p(x,u,p) $$
for suitable $(x,u,p,q)$ imply that $u^0$ is a strong minimizer.
In addition, the existence of the field is guaranteed
by the sufficient condition for the weak minimizer in Theorem~\ref{thm-Jacobi}(ii).
In the general case we have the following analogue
(see Theorem~\ref{thm-field1} for a simpler version in the scalar case $N=1$):

\begin{theorem} \label{thm-field}
Let $f\in C^2$, $\eps>0$, and let $u^0\in C^2$ be an extremal satisfying \eqref{NBC}.

\strut\hbox{\rm(i)}
Let there exist a field of extremals $\c{P}$ 
for $u^0$ satisfying the conditions
\begin{equation} \label{self-adjoint}
 \frac{\partial f_{p_i}(a,v,\psi(a,v))}{\partial v_j}
=\frac{\partial f_{p_j}(a,v,\psi(a,v))}{\partial v_i} \quad\hbox{whenever }\ i,j\in I, \ v-u^0(a)\in B_\eps, 
\end{equation}
\begin{equation} \label{field-a}
f_p(a,v,\psi(a,v))\cdot(v-u^0(a))\leq0,\quad\quad\hbox{whenever }\ v-u^0(a)\in\R^N_{\c{D},a}\cap B_\eps, 
\end{equation}
\begin{equation} \label{field-b}
f_p(b,v,\psi(b,v))\cdot(v-u^0(b))\geq0,\quad\quad\hbox{whenever }\ v-u^0(b)\in\R^N_{\c{D},b}\cap B_\eps, 
\end{equation}
where $\psi$ denotes the slope of the field.
Assume also
\begin{equation} \label{Epos}
 E(x,v,\psi(x,v),q)\geq0 \quad\hbox{for all }\ ((x,v),q)\in\c{P}\times\R^N.
\end{equation}
Then $u^0$ is a strong minimizer.

If \eqref{Epos} is only true for all $(x,v)\in\c{P}$ and $q=q(x,v)$ satisfying 
$|q-\psi(x,v)|\leq\eta$ for some $\eta>0$, then $u^0$ is a weak minimizer.

If the field is global (i.e.~$\c{P}=[a,b]\times\R^N$) and
\eqref{self-adjoint}, \eqref{field-a}, \eqref{field-b} are true with $B_\eps$ replaced by $\R^N$, 
then $u^0$ is a global minimizer. 

\strut\hbox{\rm(ii)}
Assume $\ID_a=\emptyset$ and let
there exist a field of extremals satisfying  \eqref{self-adjoint}. 
If the reversed inequality ``$\geq$'' is true in \eqref{field-a}, 
and the reversed strict inequality ``$<$'' is true in \eqref{field-b} for 
$v=u^0(b)+tw^0$, where $t\in(0,1)$ and $w^0\in\R^N_{\c{D},b}$ is fixed,
then $u_0$ is not a weak minimizer.

\strut\hbox{\rm(iii)}
Assume \eqref{ass1} 
and let the sufficient conditions for a weak minimizer
in Theorem~\ref{thm-Jacobi}(ii) be satisfied.
If $\ID_a=\emptyset$ or $\IN_a=\emptyset$ or
\begin{equation} \label{special}
\left.\begin{aligned}
&\hbox{$f_{p_i}(a,u,p)$ for $i\in\ID_a$ does not depend on $u_j,p_j$ with $j\notin\ID_a$,}\\ 
&\hbox{$f_{p_iu_j}=f_{p_ju_i}$ for $i,j\in\ID_a$,} 
\end{aligned} \quad\right\}
\end{equation}
then a field of extremals satisfying \eqref{self-adjoint},\eqref{field-a},\eqref{field-b}
exists.
\end{theorem}

\begin{remark} \label{rem-field3} \rm
The well known Weierstrass necessary condition for minimizers asserts that
the inequality $E(x,u^0(x),(u^0)'(x),q)\geq0$
for all $q\in\R^N$ or $q=q(x)$ satisfying $|q-(u^0)'(x)|\leq\eta$
is necessary for $u^0$ to be a strong or weak minimizer, respectively,
hence the nonnegativity conditions on $E$ in Theorem~\ref{thm-field}
are not far from optimal.
Similarly, Theorem~\ref{thm-field}(ii) shows that the sufficient conditions 
\eqref{field-a}--\eqref{field-b}
in Theorem~\ref{thm-field}(i) are also necessary in some sense,
at least if $\ID_a=\emptyset$.
\qed
\end{remark}

The proof of part (iii) of Theorem~\ref{thm-field} is quite technical
and, in addition, we will not need that part
in our examples (since we will prove the existence of the field required by
Theorem~\ref{thm-field}(i)--(ii) by other arguments).
Therefore the proof of part (iii) is postponed to the Appendix.

In what follows we assume that 
\begin{equation} \label{ass3}
\begin{aligned}
&\hbox{$f\in C^2$, $u^0\in C^2$ is an extremal,}\\
&\hbox{$\c{P}$ is a field of extremals for $u^0$ with slope $\psi$,
and \eqref{self-adjoint} is true.}
\end{aligned}
\end{equation}
Given $v\in C^1([a,b],\R^N)$ such that $\hbox{graph}(v):=\{(x,v(x)):x\in[a,b]\}\subset\c{P}$,
we define the Hilbert invariant integral 
$$ I(v):=\int_a^b\bigl[f\bigl(x,v(x),\psi\bigl(x,v(x)\bigr)\bigr)+
   \bigl(v'(x)-\psi\bigl(x,v(x)\bigr)\bigr)\cdot f_p\bigl(x,v(x),\psi\bigl(x,v(x)\bigr)\bigr)\bigr]\,dx.$$
The following proposition is well known, but for the reader's convenience
we provide its proof in the Appendix.

\begin{proposition} \label{prop-HII}
Assume \eqref{ass3}.
Then there exists $S\in C^2(\c{P})$ 
such that 
\begin{equation} \label{HII}
\begin{aligned}
&I(v)=S(b,v(b))-S(a,v(a))\quad\hbox{for any }\ v\in C^1([a,b],\R^N)\ \hbox{ with }\ \hbox{\rm graph}(v)\subset\c{P},\\ 
&S_v(x,v)=f_p(x,v,\psi(x,v))\quad\hbox{for any }\ (x,v)\in\c{P}.
\end{aligned}
\end{equation}
\end{proposition}

\begin{proof}[Proof of Theorem~\ref{thm-field}]
(i) Let $u-u^0\in C^1_{\c{D}}$, $\hbox{graph}(u)\subset\c{P}$,
and let $S$ be the function from Proposition~\ref{prop-HII}.
If $u$ is close to $u^0$ in the sup-norm, then the assumptions
\eqref{field-a}--\eqref{field-b} guarantee
$$S(a,u(a))-S(a,u^0(a))=\int_0^1S_v(a,u^0(a)+t(u(a)-u^0(a)))\cdot(u(a)-u^0(a))\,dt\leq0,$$ 
and similarly $S(b,u(b))-S(b,u^0(b))\geq0$,
hence $I(u^0)\leq I(u)$ due to Proposition~\ref{prop-HII}.
This fact and assumption \eqref{Epos} imply
$$ \Phi(u)-\Phi(u^0)=\Phi(u)-I(u^0)\geq\Phi(u)-I(u)
  =\int_a^b E(x,u(x),\psi(x,u(x)),u'(x))\,dx\geq0,$$ 
hence $u^0$ is a strong minimizer.
The remaining assertions in (i) are obvious.

(ii) Choose $t_k\to0+$ and let $\alpha_k$ be such that $\vp(b,\alpha_k)=u^0(b)+t_kw^0$.
Then $u^k:=\vp(\cdot,\alpha_k)\to u^0$ in $C^1$, 
$u^k-u^0\in C^1_{\c{D}}$ due to $\ID_a=\emptyset$ and $w^0\in\R^N_{\c{D},b}$, 
and, similarly as in (i), we obtain
$$\Phi(u^k)=I(u^k)=S(b,u^k(b))-S(a,u^k(a))<
S(b,u^0(b))-S(a,u^0(a))=I(u^0)=\Phi(u^0),$$
hence $u^0$ is not a minimizer.
\end{proof}

\section{Scalar examples with variable endpoints}
\label{sec-se}

Throughout this section (except for Remark~\ref{rem-Jin}) we assume $N=1$ and $\ID_a=\ID_b=\emptyset$.
Since we will often use Theorem~\ref{thm-field}, let us first reformulate it in this special case.
Notice that the extremals in the field of extremals satisfy $\vp_\alpha(x,\alpha)\ne0$,
hence we can assume $\vp_\alpha>0$ without loss of generality.

\begin{theorem} \label{thm-field1}
Let $N=1$, $\ID_a=\ID_b=\emptyset$,
$f\in C^2$ and let $u^0\in C^2$ be an extremal satisfying \eqref{NBC}.

\strut\hbox{\rm(i)}
Let there exist a field of extremals $\c{P}=\{(x,\vp(x,\alpha)):x\in[a,b],\ \alpha\in(-\eps,\eps)\}$ 
for $u^0$ satisfying the conditions $\vp_\alpha>0$ and
\begin{equation} \label{field1}
 f^\alpha_p(a)\alpha\leq0\leq f^\alpha_p(b)\alpha,\qquad\alpha\in(-\eps,\eps),
\end{equation}
where $f_p^\alpha(x):=f_p(x,\vp(x,\alpha),\vp_x(x,\alpha))$.
Assume also
\begin{equation} \label{Epos1}
 E(x,v,\psi(x,v),q)\geq0 \quad\hbox{for all }\ ((x,v),q)\in\c{P}\times\R.
\end{equation}
Then $u^0$ is a strong minimizer.

If \eqref{Epos1} is only true for all $(x,v)\in\c{P}$ and $q=q(x,v)$ satisfying 
$|q-\psi(x,v)|\leq\eta$ for some $\eta>0$, then $u^0$ is a weak minimizer.

If $\c{P}=[a,b]\times\R$,
then $u^0$ is a global minimizer. 

\strut\hbox{\rm(ii)}
Let there exist a field of extremals satisfying  $\vp_\alpha>0$.
If, for $\alpha>0$ or $\alpha<0$, the reversed inequalities in \eqref{field1}
are true and one of them is strict (for example,
if $f^\alpha_p(a)\geq0>f^\alpha_p(b)$ for $\alpha>0$),  
then $u_0$ is not a weak minimizer.

\strut\hbox{\rm(iii)}
Assume \eqref{ass1} 
and let the sufficient conditions for a weak minimizer
in Theorem~\ref{thm-Jacobi1}(ii) be satisfied.
Then a field of extremals satisfying $\vp_\alpha>0$ and \eqref{field1} exists.
\end{theorem}

\begin{remark} \label{remPQ}\rm 
If $f^0_{up}=0$ and we set $P:=f^0_{pp}$, $Q:=f^0_{uu}$, 
then $\Psi(h)=\int_a^b(P(h')^2+Qh^2)\,dx$
and the Jacobi equation has the form $-\frac{d}{dx}(Ph')+Qh=0$.
Notice also that if $P,Q>0$, then $\Psi$ is positive definite in $W^{1,2}$.
Consequently, Remark~\ref{rem-Jacobi}(iii) implies that
the sufficient conditions for a weak minimizer in Theorem~\ref{thm-Jacobi1}(ii) are satisfied
and Theorem~\ref{thm-field1}(iii) implies the existence of a field
of extremals satisfying $\vp_\alpha>0$ and \eqref{field1}.
\qed
\end{remark}

In the following examples we will consider Lagrangians $f=f(u,p)$ and we will use
the phase plane analysis for the Du Bois-Reymond equation $f^0-(u^0)'f^0_p=C$.

\begin{example} \label{ex-LG} \rm
The study of the deformation of a planar weightless inextensible
and unshearable rod (satisfying suitable boundary conditions)
leads to the minimization of the functional
\begin{equation} \label{Phi-LG}
 \Phi(u) =\int_0^1\bigl(\frac12(u'-K)^2+M\cos u\bigr)\,dx,\qquad u\in C^1([0,1]), 
\end{equation}
where $K\in\R$, $M>0$, and $u$ denotes the angle between the tangent to the rod 
and a suitable vertical, see \cite[(97)]{LG} and cf.~also \cite{B20}.
Functional $\Phi$ possesses multiple critical points, i.e.\
extremals satisfying the natural boundary
conditions $u'(0)=u'(1)=K$; see \cite{LG} for their detailed analysis.
Their stability was also analyzed in \cite{LG}, but that analysis
based on the approach from \cite{Man} is unnecessarily complicated.
Somewhat simpler arguments were used in \cite{B20}, but those arguments
cannot be used for all critical points.
We will show that Theorems~\ref{thm-Jacobi1} and~\ref{thm-field1} yield a very simple
way to determine the stability of any critical point.

Proposition~\ref{prop-ws} implies that
$u^0$ is a weak minimizer of $\Phi$ if and only if it is a strong minimizer.
Therefore we will only speak about minimizers.
Notice also that $f_{pp}=1$ and
the excess function satisfies $E(x,u,p,q)=\frac12(q-p)^2\geq0$.
Proposition~\ref{thm-Euler} guarantees that any critical point of $\Phi$ is $C^\infty$ 
and satisfies the Du Bois-Reymond equation
$(u')^2=2M\cos u+C$, where $C$ is a constant.
Conversely, any non-constant solution of the Du Bois-Reymond equation
is an extremal.

We consider the phase plane $(u,v)$, where $v=u'$, and set
$$\phi_C:=\{(u,v): v^2=2M\cos u+C\},\qquad C\in(-2M,\infty)$$ 
(see Figure~\ref{fig1}). 
The considerations above show that
given any non-constant critical point $u^0$, there exists $C^0>-2M$ such that $(u^0(x),(u^0)'(x))\in\phi_{C^0}$ for $x\in[0,1]$,
$(u^0)'(0)=(u^0)'(1)=K$. On the other hand, if $(A_0,K),(A_1,K)\in\Phi_{C^0}$ for some $C^0\in(2M,\infty)$, $A_0\ne A_1$, 
and $u^0\in C^1$ satisfies
$(u^0(x),(u^0)'(x))\in\phi_{C^0}$ for $x\in[0,1]$,
$(u^0(0),(u^0)'(0))=(A_0,K)$ and $(u^0(b),(u^0)'(b))=(A_1,K)$ for some $b>0$, then
$u^0$ is a critical point if and only if $b=1$ (the value of $b$ is uniquely determined in this case
since $(u^0)'\ne0$). 
Similar assertion is true if $C^0\in(-2M,2M]$
($K\ne0$ if $C^0=2M$),
 but this time 
one can have $(u^0(b),(u^0)'(b))=(A_1,K)$ for multiple values of $b$ (since $u^0$ need not be monotone),
and one has to allow $A_1=A_0$. 

The phase plane analysis can be used to find critical points of $\Phi$ 
(see \cite{Bed} for a particular case), but since those critical points
are known (see \cite{LG}, for example), we will restrict ourselves
to the determination of their stability.
More precisely, considering the case $K\geq0$ (the case $K\leq0$ being
symmetric), we will show the following:
A critical point of $\Phi$ is a minimizer if and only if
either $u^0(x)\equiv(2k+1)\pi$ for some integer $k$
or $u^0$ is a part of curve $\phi_{C^0}$ with $C^0>2M$
and $(u^0)''(0)<0<(u^0)''(1)$.

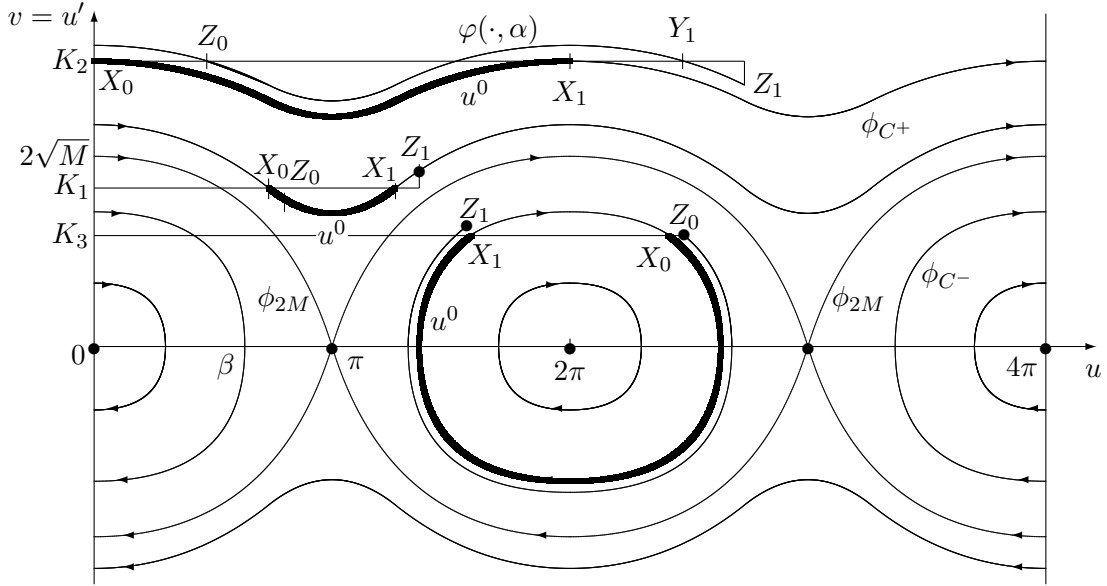
\begin{figure}[ht]
\centering
\begin{picture}(390,240)(-20,-100) 
\unitlength=.6pt
\put(0,0){\line(1,0){630}}
\put(630,0){\vector(1,0){2}}
\put(630,-15){\makebox(0,0)[c]{$u$}}
\put(300,-15){\makebox(0,0)[c]{$2\pi$}}
\put(300,-5){\line(0,1){10}}
\put(165,-10){\makebox(0,0)[c]{$\pi$}}
\put(83,-13){\makebox(0,0)[c]{$\beta$}}
\put(585,-13){\makebox(0,0)[c]{$4\pi$}}
\put(0,-150){\line(0,1){360}}   
\put(600,-150){\line(0,1){360}}
\put(0,210){\vector(0,1){2}}
\put(-30,210){\makebox(0,0)[c]{$v=u'$}}
\put(-10,-5){\makebox(0,0)[c]{$0$}}
\put(0,100){\line(1,0){205}} 
\put(205,100){\line(0,1){15}}
\put(190,93){\line(0,1){10}}
\put(203,125){\makebox(0,0)[c]{$Z_1$}}
\put(180,112){\makebox(0,0)[c]{$X_1$}}
\put(205,110){\makebox(0,0)[c]{$\bullet$}}
\put(133,110){\makebox(0,0)[c]{$Z_0$}}
\put(112,114){\makebox(0,0)[c]{$X_0$}}
\put(110,93){\line(0,1){13}}
\put(120,85){\line(0,1){12}}
\put(0,180){\line(1,0){410}}  
\put(255,200){\makebox(0,0)[c]{$\vp(\cdot,\alpha)$}}
\put(371,175){\line(0,1){10}}
\put(371,200){\makebox(0,0)[c]{$Y_1$}} 
\put(300,175){\line(0,1){10}}
\put(300,160){\makebox(0,0)[c]{$X_1$}} 
\put(75,195){\makebox(0,0)[c]{$Z_0$}}
\put(425,165){\makebox(0,0)[c]{$Z_1$}}
\put(410,165){\line(0,1){15}}
\put(71,175){\line(0,1){10}}
\put(-14,101){\makebox(0,0)[c]{$K_1$}}
\put(-15,180){\makebox(0,0)[c]{$K_2$}}
\put(14,167){\makebox(0,0)[c]{$X_0$}}
\put(-15,70){\makebox(0,0)[c]{$K_3$}}
\put(0,70){\line(1,0){140}} 
\put(160,70){\line(1,0){212}}
\put(-25,122){\makebox(0,0)[c]{$2\sqrt{M}$}}
\put(482,30){\makebox(0,0)[c]{$\phi_{2M}$}}
\put(120,30){\makebox(0,0)[c]{$\phi_{2M}$}}
\put(538,45){\makebox(0,0)[c]{$\phi_{C^-}$}}
\put(500,140){\makebox(0,0)[c]{$\phi_{C^+}$}}
\put(150,72){\makebox(0,0)[c]{$u^0$}}
\put(240,160){\makebox(0,0)[c]{$u^0$}}
\put(0,-2){\makebox(0,0)[c]{$\bullet$}}
\put(150,-2){\makebox(0,0)[c]{$\bullet$}}
\put(300,-2){\makebox(0,0)[c]{$\bullet$}}
\put(450,-2){\makebox(0,0)[c]{$\bullet$}}
\put(600,-2){\makebox(0,0)[c]{$\bullet$}}
\bezier500(0,190)(60,190)(110,165)
\bezier500(71,180)(86,175)(110,165)
\bezier500(110,165)(150,145)(190,165)
\bezier500(300,190)(240,190)(190,165)
\bezier500(300,190)(360,190)(410,165)
\bezier500(0,120)(113,120)(150,0)
\bezier500(300,120)(187,120)(150,0)
\put(20,119){\vector(1,0){2}}
\put(280,119){\vector(1,0){2}}
\bezier500(0,85)(95,85)(95,0)
\put(16,84){\vector(1,0){2}}
\bezier500(0,40)(45,40)(45,0)
\put(10,39){\vector(1,0){2}}
\bezier500(300,85)(205,85)(205,0)
\put(284,84){\vector(1,0){2}}
\bezier500(300,40)(255,40)(255,0)
\put(290,39){\vector(1,0){2}}
\bezier500(0,140)(60,140)(110,100)
\put(20,139){\vector(1,0){2}}
{\linethickness{2pt}
\bezier500(110,100)(150,68)(190,100)
}
\bezier500(300,140)(240,140)(190,100)
{\linethickness{2pt}
\bezier500(0,180)(60,180)(110,155)
\bezier500(110,155)(150,135)(190,155)
\bezier500(300,180)(240,180)(190,155)
}
\bezier500(300,120)(413,120)(450,0)
\bezier500(600,120)(487,120)(450,0)
\put(580,119){\vector(1,0){2}}
\bezier500(300,85)(395,85)(395,0)
{\linethickness{2pt}
\bezier500(362,70)(395,45)(395,0)
\bezier500(395,0)(395,-85)(300,-85)
\bezier500(300,-85)(205,-85)(205,0)
\bezier500(205,0)(205,45)(238,70)
}
\put(221,18){\makebox(0,0)[c]{$u^0$}}
\put(353,55){\makebox(0,0)[c]{$X_0$}} 
\put(372,82){\makebox(0,0)[c]{$Z_0$}}
\put(372,70){\makebox(0,0)[c]{$\bullet$}}
\bezier500(372,70)(402,45)(402,0)
\bezier500(402,0)(402,-92)(300,-92)
\bezier500(300,-92)(198,-92)(198,0)
\bezier500(198,0)(198,47)(235,76)
\put(235,76){\makebox(0,0)[c]{$\bullet$}}
\put(240,86){\makebox(0,0)[c]{$Z_1$}}
\put(247,57){\makebox(0,0)[c]{$X_1$}}
\put(584,84){\vector(1,0){2}}
\bezier500(300,40)(345,40)(345,0)
\put(590,39){\vector(1,0){2}}
\bezier500(600,85)(505,85)(505,0)
\bezier500(600,40)(555,40)(555,0)
\bezier500(300,140)(360,140)(410,100)
\put(580,139){\vector(1,0){2}}
\bezier500(410,100)(450,68)(490,100)
\bezier500(600,140)(540,140)(490,100)
\bezier500(300,180)(360,180)(410,155)
\bezier500(410,155)(450,135)(490,155)
\bezier500(600,180)(540,180)(490,155)
\put(580,179){\vector(1,0){2}}
\bezier500(0,-120)(113,-120)(150,0)
\bezier500(300,-120)(187,-120)(150,0)
\put(20,-119){\vector(-1,0){2}}
\put(280,-119){\vector(-1,0){2}}
\bezier500(0,-85)(95,-85)(95,0)
\put(16,-84){\vector(-1,0){2}}
\bezier500(0,-40)(45,-40)(45,0)
\put(10,-39){\vector(-1,0){2}}
\bezier500(300,-85)(205,-85)(205,0)
\put(284,-84){\vector(-1,0){2}}
\bezier500(300,-40)(255,-40)(255,0)
\put(290,-39){\vector(-1,0){2}}
\bezier500(0,-140)(60,-140)(110,-100)
\put(20,-139){\vector(-1,0){2}}
\bezier500(110,-100)(150,-68)(190,-100)
\bezier500(300,-140)(240,-140)(190,-100)
\bezier500(300,-120)(413,-120)(450,0)
\bezier500(600,-120)(487,-120)(450,0)
\put(580,-119){\vector(-1,0){2}}
\bezier500(300,-85)(395,-85)(395,0)
\put(584,-84){\vector(-1,0){2}}
\bezier500(300,-40)(345,-40)(345,0)
\put(590,-39){\vector(-1,0){2}}
\bezier500(600,-85)(505,-85)(505,0)
\bezier500(600,-40)(555,-40)(555,0)
\bezier500(300,-140)(360,-140)(410,-100)
\put(580,-139){\vector(-1,0){2}}
\bezier500(410,-100)(450,-68)(490,-100)
\bezier500(600,-140)(540,-140)(490,-100)
\end{picture}
\kern-2mm
   \caption{Phase plane and extremals for Example~\ref{ex-LG} and $0\leq u\leq4\pi$;
  $C^-<2M<C^+$,
   $Z_i=(\vp(i,\alpha),\vp_x(i,\alpha))$, $i=0,1$, $Y_1=(A_1+\alpha,K)$,
    $X_i=(A_i,K)=(u^0(i),(u^0)'(i))$, $i=0,1$.}
   \label{fig1}
\end{figure}


Let us first consider a critical point $u^0$ being a part of curve $\phi_{C^0}$ with $C^0>2M$,
and let $(A_i,K)$ be as above.
For symmetry reasons we may assume $K>0$. 
Notice that $u''=-M\sin u$, $|(u^0)''(0)|=|(u^0)''(1)|$,
and that $u^0(x)$ can also be defined
(as an extremal, hence a part of $\phi_{C^0}$) for $x\notin[0,1]$. 

If $(u^0)''(0)<0<(u^0)''(1)$ 
(i.e.~$u^0(0)\in(2k\pi,(2k+1)\pi)$ and $u^0(1)\in((2m+1)\pi,(2m+2)\pi)$ for some $m\geq k$;
see the extremal $u^0$ with $(u^0)'(0)=K_1$ in Figure~\ref{fig1}), 
then $\vp(x,\alpha):=u^0(x+\alpha)$, $x\in[0,1]$, $\alpha\in(-\eps,\eps)$, is a field of extremals for $u^0$
satisfying \eqref{field1}, hence Theorem~\ref{thm-field1}(i) guarantees that $u^0$ is a minimizer.
If $(u^0)''(0)>0>(u^0)''(1)$, then the same argument and Theorem~\ref{thm-field1}(ii)
show that $u^0$ is not a minimizer.

Next assume that $(u^0)''(0)\cdot(u^0)''(1)\geq0$. 
We will show that $u^0$ is not a minimizer.

Assume $(u^0)''(0)<0$, or $(u^0)''(0)=0$ and $(u^0)'''(0)<0$
(the cases $(u^0)''(0)>0$, or $(u^0)''(0)=0$ and $(u^0)'''(0)>0$ are analogous).
We necessarily have $A_1=A_0+2k_0\pi$ for some $k_0\in\{1,2,\dots\}$.
Let $\vp(\cdot,\alpha)$ (with $|\alpha|$ being small) 
be the extremal with initial values $Z_0:=(\vp(0,\alpha),\vp_x(0,\alpha))=(A_0+\alpha,K)$
(see the extremal $u^0$ with $(u^0)'(0)=K_2$ in Figure~\ref{fig1}). 
Then $\vp$ is a field of extremals for $u^0$, and
$\vp(\cdot,\alpha)$ is a part of the curve $\phi_{C^\alpha}$, where $C^\alpha$ is close to $C^0$,
$C^\alpha>C^0$ if $\alpha>0$.

Let $\alpha>0$ be small.
If $u^1$ and $u^2$ are extremals in $\phi_{C^0}$ and $\phi_{C^\alpha}$, respectively,
and $u^1(0)=u^2(0)=0$, then $u^1(b_1)=u^2(b_2)=2\pi$ for some $0<b_1<b_2$
(due to $(u^2)'>(u^1)'$ whenever $u^2=u^1$).
This fact and the $2\pi$-periodicity of the problem guarantee that
$\vp(b,\alpha)=A_1+\alpha$ for some $b<1$, hence $\vp_x(1,\alpha)<(u^0)'(1)$,
and Theorem~\ref{thm-field1}(ii) implies that $u^0$ is not a minimizer.

Next consider the case $C^0\in(-2M,2M]$ and $K\geq0$; $K\ne0$ if $C^0=2M$.
If $K>0$ and
 $(u^0)''(0)>0>(u^0)''(1)$, then the same arguments as above guarantee that $u^0$ is not a minimizer.
If $K=0$ or $(u^0)''(0)<0<(u^0)''(1)$ (hence $A_1<A_0$) or
$(u^0)''(0)\cdot(u^0)''(1)\geq0$ (hence $A_0=A_1=2k\pi$),
then choosing $\vp(\cdot,\alpha)$ to be an extremal satisfying initial conditions 
$(\vp(0,\alpha),\vp_x(0,\alpha))=(A_0+\alpha,K)$
we see from the phase plane that $\vp(\cdot,\alpha)$ and $u^0$
intersect in $(0,1)$ for any $\alpha\ne0$ small
(if, for example, 
$(u^0)''(0)<0<(u^0)''(1)$ and $\alpha>0$ is small, then
there exists $y\in(0,1)$ such that $\vp(y,\alpha)=\min\vp(\cdot,\alpha)<\min u^0$,
and the inequalities 
$\vp(0,\alpha)>u^0(0)$, $\vp(y,\alpha)<u^0(y)$ imply that
$\vp(\cdot,\alpha)$ and $u^0$ intersect in $(0,y)$; 
see the extremal $u^0$ with $(u^0)'(0)=K_3$ in Figure~\ref{fig1}). 
Consequently, $h:=\vp_x(\cdot,0)$
is a solution of the Jacobi equation satisfying $h(0)=1$, $h'(0)=0$,
$h(y)=0$ for some $y\in(0,1]$, and Theorem~\ref{thm-Jacobi1}
guarantees that $u^0$ is not a minimizer.

Similar considerations as above can be used in the case of constant extremals
$k\pi$, but we will use a different argument:
If $u^0\equiv(2k+1)\pi$, then $P=1$, $Q=-M\cos u^0=M$, 
and the solution $h(x)=e^{\sqrt{M}x}+e^{-\sqrt{M}x}$ of the Jacobi equation
satisfies $h>0$, $h'(0)=0$, $h'(1)>0$, hence $u^0$ is a minimizer.
If $u^0\equiv2k\pi$, then $P=1$, $Q=-M$
and the solution $h(x)=\cos(\sqrt{M}x)$ of the Jacobi equation satisfies
$h(0)>0$, $h'(0)=0$ and
 either
$h(x)=0$ for some $x\in(0,1]$ or $h'(1)<0$, hence $u^0$ is not a minimizer.  
\qed
\end{example}

\begin{remark} \label{rem-Jin} \rm
The author of \cite{Jin} considers the functional $\Phi$ in \eqref{Phi-LG}
with $K=0$, $[a,b]=[-1/2,1/2]$ (instead of $[a,b]=[0,1]$), and 
the Dirichlet boundary conditions $u(-1/2)=u(1/2)=0$, see \cite[(6)]{Jin}.
He considers the extremal $u^0$ satisfying $u^0(0)=\beta$ and $(u^0)'(0)=0$,
i.e.~the extremal passing through the point $(\beta,0)$ in Figure~\ref{fig1},
and he provides explicit formulas for this extremal, its field of extremals $\vp$
and the derivative $\vp_\alpha$ (see \cite[(8),(9),(13),(14) and (16)]{Jin};
functions $u^0,\vp$ and $\vp_\alpha$ are denoted by $\theta,y$ and $\partial y/\partial\gamma$,
respectively).  
The nonnegativity of the excess function then implies
that $u^0$ is a strong minimizer.
In \cite[Introduction]{Jin}, the author claims that
``Based on the Jacobian test, potential energy of Euler elasticas \dots\ 
was proved to hold a weak minimum value\dots'', but
``\dots it is an open problem to find sufficient conditions for the potential energy
for these Euler elasticas to hold a strong minimum.''
However, Proposition~\ref{prop-ws} shows that weak and strong minimizers of 
functional $\Phi$ in \eqref{Phi-LG} are equivalent.
In addition, Theorem~\ref{thm-field}(iii) implies that the positive
definiteness of the second variation $\psi$ in $W^{1,2}_0(-1/2,1/2)$ 
(i.e.~the sufficient condition for a weak minimizer) 
guarantees the existence of the required field $\vp$,
hence the technical construction of the field in \cite{Jin} is not necessary
even if we do not consider Proposition~\ref{prop-ws}.
\qed
\end{remark}

\begin{example} \label{ex-well} \rm 
Consider the functional $\Phi(u)=\int_a^b f(u,u')\,dx$ in $C^1([a,b])$,
where $f(u,p)=g(p)+u^2$ and $g$ is a double-well function.
More precisely, we will consider the following two cases (see Figure~\ref{fig6}): 

\begin{figure}[ht]
\centering
\begin{picture}(300,100)(0,-15)
\unitlength=1pt
\put(0,0){\line(1,0){100}}
\put(50,0){\line(0,1){80}}
\put(100,0){\vector(1,0){2}}
\put(50,-10){\makebox(0,0)[c]{$0$}}
\put(21,-10){\makebox(0,0)[c]{$-1$}}
\put(75,-10){\makebox(0,0)[c]{$1$}}
\put(100,-10){\makebox(0,0)[c]{$p$}}
\put(12,42){\makebox(0,0)[c]{$g$}}
\put(30,70){\makebox(0,0)[c]{(a)}} 
\bezier500(0,45)(10,0)(25,0)
\bezier500(25,0)(30,0)(37,12)
\bezier500(37,12)(45,25)(50,25)
\bezier500(100,45)(90,0)(75,0)
\bezier500(75,0)(70,0)(63,12)
\bezier500(63,12)(55,25)(50,25)
\put(150,0){\line(1,0){125}}
\put(200,0){\line(0,1){80}}
\put(175,0){\line(0,1){50}}
\put(275,0){\vector(1,0){2}}
\put(200,-10){\makebox(0,0)[c]{$0$}}
\put(171,-10){\makebox(0,0)[c]{$-1$}}
\put(250,-10){\makebox(0,0)[c]{$2$}}
\put(275,-10){\makebox(0,0)[c]{$p$}}
\put(168,62){\makebox(0,0)[c]{$g$}}
\put(230,70){\makebox(0,0)[c]{(b)}} 
\bezier500(150,70)(160,50)(175,50)
\bezier500(175,50)(180,50)(185,52)
\bezier500(185,52)(190,55)(200,55)
\bezier500(200,55)(210,55)(220,30)
\bezier500(220,30)(232,0)(250,0)
\bezier500(250,0)(260,0)(275,50)
\end{picture}
\kern-2mm
   \caption{Graphs of $g$ in the symmetric and non-symetric cases.}
   \label{fig6}
\end{figure}
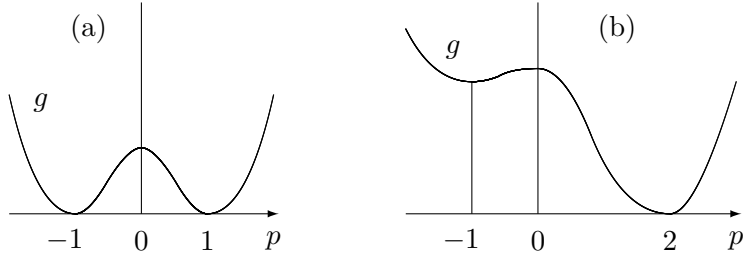
 
(a) $g(p)=(p^2-1)^2$ (hence $g'(p)=4p(p^2-1)$, $g''(p)=4(3p^2-1)$),
 
(b) $g(p)=\frac14 p^4-\frac13 p^3-p^2+\frac83$
(hence $g'(p)=(p+1)p(p-2)$, $g''(p)=3p^2-2p-2$).

Let us consider the symmetric case (a) first.
The Du Bois-Reymond equation has the form
$$ u^2=C+h(u'), \qquad\hbox{where}\quad h(p):=3p^4-2p^2, $$
see Figures~\ref{fig7} and~\ref{fig8} 
for the graph of $h$ and the phase plane $(u,u')$, respectively.
All minimizers have to satisfy $u'(a),u'(b)\in\{0,\pm1\}$;
the only constant extremal
is $u\equiv0$.
Functional $\Phi$ does not possess local maximizers since $\Phi''(u^0)(1,1)>0$ for any $u^0$.

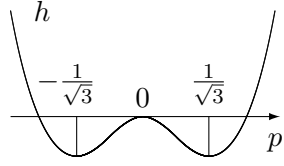
\begin{figure}[ht]
\centering
\begin{picture}(100,60)(0,-15)
\unitlength=1pt
\put(0,0){\line(1,0){100}}
\put(100,0){\vector(1,0){2}}
\put(50,7){\makebox(0,0)[c]{$0$}}
\put(21,12){\makebox(0,0)[c]{$-\frac1{\sqrt3}$}}
\put(75,12){\makebox(0,0)[c]{$\frac1{\sqrt3}$}}
\put(100,-10){\makebox(0,0)[c]{$p$}}
\put(12,40){\makebox(0,0)[c]{$h$}}
\put(25,0){\line(0,-1){15}}
\put(75,0){\line(0,-1){15}}
\bezier500(0,40)(10,-15)(25,-15)
\bezier500(25,-15)(30,-15)(37,-8)
\bezier500(37,-8)(45,0)(50,0)
\bezier500(100,40)(90,-15)(75,-15)
\bezier500(75,-15)(70,-15)(63,-8)
\bezier500(63,-8)(55,0)(50,0)
\end{picture}
\kern-2mm
   \caption{Graph of $h$ in the symmetric case.}
   \label{fig7}
\end{figure}

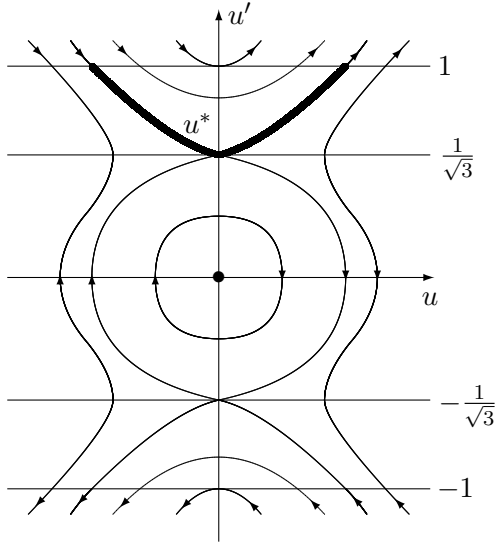
\begin{figure}[ht]
\centering
\begin{picture}(160,200)(-80,-100) 
\unitlength=0.8pt
\put(-100,0){\line(1,0){200}}
\put(100,0){\vector(1,0){2}}
\put(100,-10){\makebox(0,0)[c]{$u$}}
\put(-100,100){\line(1,0){200}}
\put(107,100){\makebox(0,0)[c]{$1$}}
\put(112,-100){\makebox(0,0)[c]{$-1$}}
\put(-100,58){\line(1,0){200}}
\put(113,56){\makebox(0,0)[c]{$\frac1{\sqrt3}$}}
\put(118,-61){\makebox(0,0)[c]{$-\frac1{\sqrt3}$}}
\put(-100,-58){\line(1,0){200}}
\put(-100,-100){\line(1,0){200}}
\put(0,-125){\line(0,1){250}}
\put(0,125){\vector(0,1){2}}
\put(10,125){\makebox(0,0)[c]{$u'$}}
\put(-10,73){\makebox(0,0)[c]{$u^*$}}
\put(0,0){\makebox(0,0)[c]{$\bullet$}}
\bezier300(-70,112)(-65,105)(-60,100)
{\linethickness{2pt}
\bezier500(-60,100)(-23,63)(0,58)
\bezier500(0,58)(23,63)(60,100)
}
\bezier300(70,112)(65,105)(60,100)
\bezier300(-70,-112)(-65,-105)(-60,-100)
\bezier300(70,-112)(65,-105)(60,-100)
\put(65,-105.5){\vector(-1,1){2}}
\put(-65,-105.5){\vector(-1,-1){2}}
\put(-65,105.5){\vector(1,-1){2}}
\put(65,105.5){\vector(1,1){2}}
\bezier500(-60,-100)(-23,-63)(0,-58)
\bezier500(0,-58)(23,-63)(60,-100)
\bezier300(-50,112)(0,58)(50,112)
\put(-45,107){\vector(1,-1){2}}
\put(45,107){\vector(1,1){2}}
\bezier300(-20,112)(0,88)(20,112)
\put(-16,107.7){\vector(1,-1){2}}
\put(16,107.7){\vector(1,1){2}}
\bezier300(-50,-112)(0,-58)(50,-112)
\put(-45,-107){\vector(-1,-1){2}}
\put(45,-107){\vector(-1,1){2}}
\bezier300(-20,-112)(0,-88)(20,-112)
\put(-16,-107.7){\vector(-1,-1){2}}
\put(16,-107.7){\vector(-1,1){2}}
\bezier500(-60,0)(-60,45)(0,58)
\put(-60,0){\vector(0,1){2}}
\put(60,0){\vector(0,-1){2}}
\bezier500(-60,0)(-60,-45)(0,-58)
\bezier500(60,0)(60,45)(0,58)
\bezier500(60,0)(60,-45)(0,-58)
\bezier500(-30,0)(-30,29)(0,29)
\bezier500(30,0)(30,29)(0,29)
\bezier500(-30,0)(-30,-29)(0,-29)
\bezier500(30,0)(30,-29)(0,-29)
\put(-30,0){\vector(0,1){2}}
\put(30,0){\vector(0,-1){2}}
\bezier500(-90,112)(-50,70)(-50,58)
\put(-85,106.5){\vector(1,-1){2}}
\bezier500(-50,58)(-50,46)(-62,31)
\bezier500(-62,31)(-75,16)(-75,0)
\put(-75,0){\vector(0,1){2}}
\put(75,0){\vector(0,-1){2}}
\bezier500(90,112)(50,70)(50,58)
\put(85,106.5){\vector(1,1){2}}
\bezier500(50,58)(50,46)(62,31)
\bezier500(62,31)(75,16)(75,0)
\bezier500(-90,-112)(-50,-70)(-50,-58)
\put(-85,-106.5){\vector(-1,-1){2}}
\bezier500(-50,-58)(-50,-46)(-62,-31)
\bezier500(-62,-31)(-75,-16)(-75,0)
\bezier500(90,-112)(50,-70)(50,-58)
\put(85,-106.5){\vector(-1,1){2}}
\bezier500(50,-58)(50,-46)(62,-31)
\bezier500(62,-31)(75,-16)(75,0)
\end{picture}
\kern-2mm
   \caption{Phase plane in the symmetric case.}
   \label{fig8}
\end{figure}

Since $f_{pp}(u,p)=4(3p^2-1)$, the extremals in the region
$|u'|\leq1/\sqrt3$ (satisfying $(u^0)'(a)=(u^0)'(b)=0$)
cannot be local minimizers.
The extremal $u^*$ with $(u^*)'(a)=1$ and $\min(u^*)'=1/\sqrt3$ (see Figure~\ref{fig8}) 
satifies $u^*(b^*)=1$ for some $b^*>a$.
If $b\in(a,b^*)$, then there exists a unique extremal $u^0$
satisfying $(u^0)'(a)=(u^0)'(b)=1$
(and a unique extremal $u^1$ satisfying $(u^1)'(a)=(u^1)'(b)=-1$);
in addition $(u^0)'>1/\sqrt3$ (and $(u^1)'<-1/\sqrt3$).
Since $P,Q>0$ and the excess function $E=(q-p)^2((q+p)^2+2(p^2-1))$ 
considered as a function of $q$
changes sign if $|p|<1$,
Remarks~\ref{remPQ} and \ref{rem-field3} show that 
the extremals $u^0,u^1$ are weak but not strong minimizers.
(Remark~\ref{remPQ} also guarantees the existence
of a field of extremals, but this fact is not needed here:
The Weierstrass necessary condition for strong minimizers
in Remark~\ref{rem-field3}
does not require the existence of a field of extremals.)
Notice also that $\inf\Phi=0$ is not attained (neither in $C^1$, nor in $W^{1,4}$):
A minimizing sequence in $C^1$ can be obtained by suitable smooth approximation of 
piecewise $C^1$-functions $u_\eps$ satisfying $|u_\eps'|=1$ a.e. and
$|u_\eps|\leq\eps$.

Next consider the nonsymmetric case (b).
The Du Bois-Reymond equation has the form
$$ u^2=C+h(u'), \qquad\hbox{where}\quad h(p):=\frac34 p^4-\frac23 p^3-p^2,$$
see Figures~\ref{fig9} and~\ref{fig10} 
for the graph of $h$ and the phase plane $(u,u')$, respectively. 
All minimizers have to satisfy $u'(a),u'(b)\in\{0,-1,2\}$;
the only constant extremal 
is $u\equiv0$.

\begin{figure}[ht]
\centering
\begin{picture}(150,90)(0,-40)
\unitlength=1pt
\put(0,0){\line(1,0){150}}
\put(150,0){\vector(1,0){2}}
\put(50,7){\makebox(0,0)[c]{$0$}}
\put(23,12){\makebox(0,0)[c]{$\frac{1-\sqrt7}3$}}
\put(104,12){\makebox(0,0)[c]{$\frac{1+\sqrt7}3$}}
\put(150,-10){\makebox(0,0)[c]{$p$}}
\put(9,40){\makebox(0,0)[c]{$h$}}
\put(25,0){\line(0,-1){5}}
\put(104,0){\line(0,-1){40}}
\bezier500(0,40)(10,-5)(25,-5)
\bezier500(25,-5)(30,-5)(40,-2)
\bezier500(40,-2)(45,0)(50,0)
\bezier500(150,40)(130,-40)(104,-40)
\bezier500(104,-40)(80,-40)(73,-20)
\bezier500(73,-20)(66,0)(50,0)
\end{picture}
\kern-2mm
   \caption{Graph of $h$ in the non-symmetric case.}
   \label{fig9}
\end{figure}
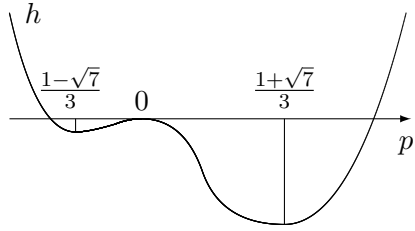
\begin{figure}[ht]
\centering
\begin{picture}(200,290)(-100,-100) 
\unitlength=0.8pt
\put(-100,0){\line(1,0){200}}
\put(100,0){\vector(1,0){2}}
\put(100,-10){\makebox(0,0)[c]{$u$}}
\put(-160,200){\line(1,0){320}}
\put(167,200){\makebox(0,0)[c]{$2$}}
\put(112,-100){\makebox(0,0)[c]{$-1$}}
\put(-100,121){\line(1,0){200}}
\put(120,122){\makebox(0,0)[c]{$\frac{1+\sqrt7}3$}}
\put(120,-58){\makebox(0,0)[c]{$\frac{1-\sqrt7}3$}} 
\put(-100,-58){\line(1,0){200}}
\put(-100,-100){\line(1,0){200}}
\put(0,-125){\line(0,1){350}}
\put(0,225){\vector(0,1){2}}
\put(10,225){\makebox(0,0)[c]{$u'$}}
\put(-25,140){\makebox(0,0)[c]{$u^*$}}
\put(0,0){\makebox(0,0)[c]{$\bullet$}}

\bezier300(-70,-112)(-65,-105)(-60,-100)
\bezier300(70,-112)(65,-105)(60,-100)
\put(65,-105.5){\vector(-1,1){2}}
\put(-65,-105.5){\vector(-1,-1){2}}
\bezier300(-140,224)(-130,210)(-120,200)
{\linethickness{2pt}
\bezier500(-120,200)(-46,126)(0,121)
\bezier500(0,121)(46,126)(120,200)
}
\bezier300(140,224)(130,210)(120,200)
\put(-130,211){\vector(1,-1){2}}
\put(130,211){\vector(1,1){2}}
\bezier300(-50,-112)(0,-58)(50,-112)
\put(-45,-107){\vector(-1,-1){2}}
\put(45,-107){\vector(-1,1){2}}
\bezier300(-20,-112)(0,-88)(20,-112)
\put(-16,-107.7){\vector(-1,-1){2}}
\put(16,-107.7){\vector(-1,1){2}}
\bezier500(-30,0)(-30,-29)(0,-29)
\bezier500(30,0)(30,-29)(0,-29)
\put(-30,0){\vector(0,1){2}}
\put(30,0){\vector(0,-1){2}}
\bezier500(-90,-112)(-50,-70)(-50,-58)
\put(-85,-106.5){\vector(-1,-1){2}}
\bezier500(-50,-58)(-50,-48)(-62,-40)
\bezier500(-62,-40)(-90,-18)(-90,0)
\bezier500(90,-112)(50,-70)(50,-58)
\put(85,-106.5){\vector(-1,1){2}}
\bezier500(50,-58)(50,-48)(62,-40)
\bezier500(62,-40)(90,-16)(90,0)
\bezier500(-60,-100)(-23,-63)(0,-58)
\bezier500(0,-58)(23,-63)(60,-100)
\put(-60,0){\vector(0,1){2}}
\put(60,0){\vector(0,-1){2}}
\bezier500(-60,0)(-60,-45)(0,-58)
\bezier500(60,0)(60,-45)(0,-58)
\put(-90,0){\vector(0,1){2}}
\put(90,0){\vector(0,-1){2}}
\bezier300(-100,224)(0,116)(100,224)
\put(-90,214){\vector(1,-1){2}}
\put(90,214){\vector(1,1){2}}
\bezier300(-40,224)(0,176)(40,224)
\put(-32,215.4){\vector(1,-1){2}}
\put(32,215.4){\vector(1,1){2}}
\bezier500(60,0)(60,105)(0,121)
\bezier500(-60,0)(-60,105)(0,121)
\bezier500(-30,0)(-30,60)(0,60)
\bezier500(30,0)(30,60)(0,60)
\bezier500(-170,224)(-70,140)(-70,121)
\put(-157,213){\vector(1,-1){2}}
\bezier500(-70,121)(-70,110)(-80,70)
\bezier500(-80,70)(-90,30)(-90,0)
%
\bezier500(170,224)(70,140)(70,121)
\put(157,213){\vector(1,1){2}}
\bezier500(70,121)(70,110)(80,70)
\bezier500(80,70)(90,30)(90,0)
\end{picture}
\kern-2mm
   \caption{Phase plane in the non-symmetric case.}
   \label{fig10}
\end{figure}
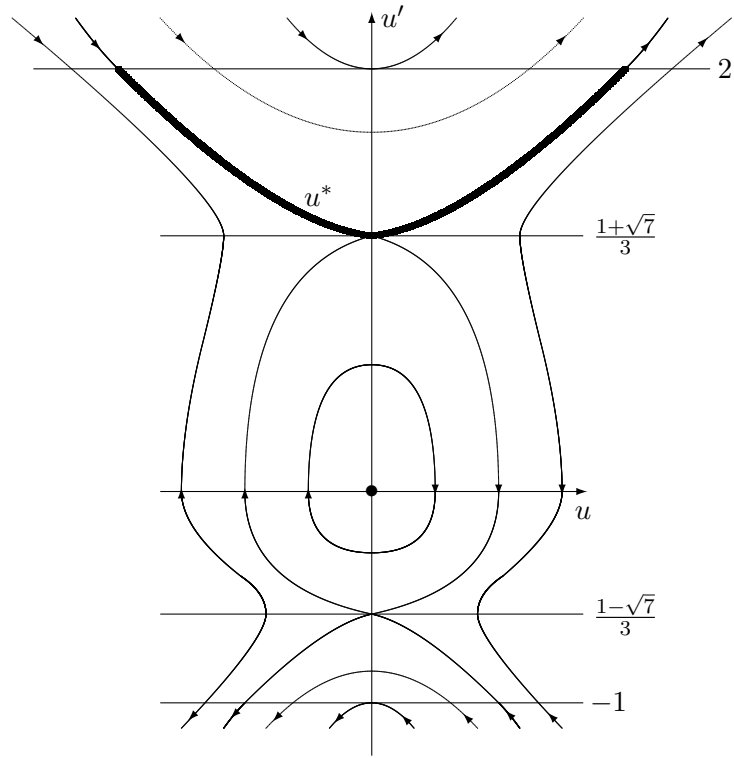

Since $f_{pp}(u,p)=3p^2-2p-2$, similarly as in case (a) we see that
the extremals in the region $u'\in[\frac{1-\sqrt7}3,\frac{1+\sqrt7}3]$
are neither local minimizers nor local maximizers.
The extremal $u^*$ with $(u^*)'(a)=2$ and $\min(u^*)'=\frac{1+\sqrt7}3$ (see Figure~\ref{fig10})  
satifies $u^*(b^*)=2$ for some $b^*>a$.
If $b\in(a,b^*)$, then there exists a a unique extremal $u^0$
satisfying $(u^0)'(a)=(u^0)'(b)=2$ and, as above, this extremal 
is a weak local minimizer.
However, now 
$E=\frac1{12}(q-p)^2((\sqrt3(q+p)-\frac2{\sqrt3})^2+6p^2-4p-13\frac13)\geq0$ for
all $q$
if $p\leq p_1$ or $p\geq p_2$, where
if $p_1=\frac13(1-\sqrt{21})<-1$, $p_2=\frac13(1+\sqrt{21})\in(\frac13(1+\sqrt7),2)$,
and Remark~\ref{remPQ} guarantees the existence of a field of
extremals satisfying $\vp_\alpha>0$ and \eqref{field1},   
hence $u^0$ is a strong local minimizer
provided $\min (u^0)'>p_2$ 
(and it is not if $\min (u^0)'<p_2$).
In fact, if $\min (u^0)'>p_2$, then Proposition~\ref{prop-global} below shows
the existence of a global field of extremals for $u^0$
satisfying the assumptions of Theorem~\ref{thm-field1}(i), with slope 
$\psi>p_2$, hence $u^0$ is a global minimizer.

An analogous analysis as in the case $u'>\frac{1+\sqrt7}3$ shows that the extremals
in the region $u'<\frac{1-\sqrt7}3$ are weak but not strong local minimizers.
\qed
\end{example}

\begin{proposition} \label{prop-global}
Let $\Phi$ and $p_2$ be as in Example~\ref{ex-well}(b),
and let $u^0$ be a critical point of $\Phi$ satisfying $\min (u^0)'>p_2$.
Then there exists a global field of extremals for $u^0$
satisfying the assumptions of Theorem~\ref{thm-field1}(i), with slope 
$\psi>p_2$.
\end{proposition}

\begin{proof}
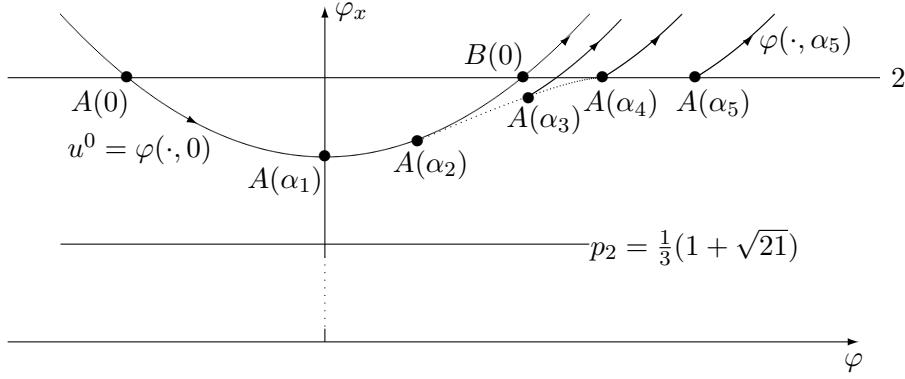
\begin{figure}[ht]
\centering
\begin{picture}(240,163)(-70,90)  
\unitlength=1pt
\put(0,132){\line(0,1){93}}
\put(-120,200){\line(1,0){330}}
\put(217,200){\makebox(0,0)[c]{$2$}}
\put(-100,137){\line(1,0){200}}
\put(140,136){\makebox(0,0)[c]{$p_2=\frac13(1+\sqrt{21})$}}  
\put(0,225){\vector(0,1){2}}
\put(10,225){\makebox(0,0)[c]{$\vp_x$}}
\put(-120,100){\line(1,0){320}}
\put(200,100){\vector(1,0){2}}
\put(200,92){\makebox(0,0)[c]{$\vp$}}
\put(0,100){\line(0,1){5}}
\bezier9(0,106)(0,119)(0,132) 
\bezier500(-100,224)(0,116)(100,224)
\put(-50,183.5){\vector(2,-1){2}}
\put(-75,200){\makebox(0,0)[c]{$\bullet$}}
\put(-85,190){\makebox(0,0)[c]{$A(0)$}}
\put(-70,173){\makebox(0,0)[c]{$u^0=\vp(\cdot,0)$}}
\put(75,200){\makebox(0,0)[c]{$\bullet$}}
\put(64,208){\makebox(0,0)[c]{$B(0)$}}
\put(0,170){\makebox(0,0)[c]{$\bullet$}}  
\put(-15,161){\makebox(0,0)[c]{$A(\alpha_1)$}}
\put(35,176){\makebox(0,0)[c]{$\bullet$}}
\put(40,167){\makebox(0,0)[c]{$A(\alpha_2)$}}
\bezier50(35,176)(90,200)(105,200)
\put(77,192){\makebox(0,0)[c]{$\bullet$}} 
\bezier150(77,193)(94,203)(112,222)
\put(102,212){\vector(1,1){2}}
\put(83,184){\makebox(0,0)[c]{$A(\alpha_3)$}}
\put(105,200){\makebox(0,0)[c]{$\bullet$}}  
\bezier150(105,200)(120,209)(135,224)
\put(125,214){\vector(1,1){2}}
\put(112,190){\makebox(0,0)[c]{$A(\alpha_4)$}}
\put(140,200){\makebox(0,0)[c]{$\bullet$}}
\bezier150(140,200)(155,209)(170,224)
\put(160,214){\vector(1,1){2}}
\put(182,214){\makebox(0,0)[c]{$\vp(\cdot,\alpha_5)$}}
\put(147,190){\makebox(0,0)[c]{$A(\alpha_5)$}}
\put(91,214){\vector(1,1){2}}
\end{picture}
   \caption{Global field of extremals: $A(\alpha)=(\vp(a,\alpha),\vp_x(a,\alpha))$,
   $B(\alpha)=(\vp(b,\alpha),\vp_x(b,\alpha))$,
  $(b-a)/2=\alpha_1<b-a-\eps=\alpha_2<\alpha_3<\alpha_4=b-a+\eps<\alpha_5$.}
   \label{fig10a}
\end{figure}

Assume first $\alpha\geq0$.
Then we choose the extremals $u^\alpha:=\vp(\cdot,\alpha)$ in the global field 
such that $\vp(\cdot,\alpha)$ is the solution of the Du Bois-Reymond equation
satisfying $$(\vp(a,\alpha),\vp_x(a,\alpha))=A(\alpha),$$ where
$A(\alpha)=(A_1(\alpha),A_2(\alpha)):(0,\infty)\to\R^2$ is smooth,
\begin{equation} \label{Z1}
A(\alpha)=\begin{cases}
 (u^0(a+\alpha),(u^0)'(a+\alpha)) &\hbox{ if }\alpha\leq b-a-\eps,\\
 (u^0(b)+\alpha-(b-a),2) &\hbox{ if }\alpha\geq b-a+\eps,
\end{cases}\end{equation}
\begin{equation} \label{Z2}
A_1'(\alpha)\geq1,\quad A_2'(\alpha)>0
\quad\hbox{for }\ \alpha\in(b-a-\eps,b-a+\eps),\ \hbox{ where }\ \eps\in(0,(b-a)/2),
\end{equation} 
see Figure~\ref{fig10a}.
Notice that $A_1'(b-a-\eps)=(u^0)'(b-\eps)>p_2>1$,
$A_2'(b-a-\eps)=(u^0)''(b-\eps)>0$,
$A_1'(b-a+\eps)=1$,
$A_2'(b-a+\eps)=0$,
$A_1(b-a+\eps)-A_1(b-a-\eps)>2\eps$
(since $A_1(b-a+\eps)=u^0(b)+\eps$, $A_1(b-a-\eps)=u^0(b-\eps)$,
$u^0(b)-u^0(b-\eps)=(u^0)'(b-\theta\eps)\eps>p_2\eps$),
$A_2(b-a+\eps)>A_2(b-a-\eps)$, so that 
\eqref{Z2} can be satisfied.
 
Let us show that $\vp_\alpha>0$. 
Since $\vp(x,\alpha)=u^0(x+\alpha)$ for $\alpha\leq b-a-\eps$
and $(u^0)'>0$, we may assume $\alpha>b-a-\eps$, hence $\vp>0$. 
Set $w(x,\alpha)=\vp_\alpha(x,\alpha)$. 
Then \eqref{Z1}--\eqref{Z2} imply $w(a,\alpha)\geq1$.
Let $h^{-1}$ denote the inverse of the increasing function
$h|_{(p_2,\infty)}$. Since $\vp(\cdot,\alpha)$ solves the Du Bois-Reymond
equation, there exists $C(\alpha)$
such that $\vp(x,\alpha)^2=C(\alpha)+h(\vp_x(x,\alpha))$.
Consequently,
\begin{equation} \label{eq-w}
 w_x=\frac{\partial}{\partial x}(\vp_\alpha)
    =\frac{\partial}{\partial\alpha}(\vp_x)
  =\frac{\partial}{\partial\alpha}(h^{-1}(\vp^2-C(\alpha))
  =\underbrace{(h^{-1})'(\vp^2-C(\alpha))}_{>0}[2\vp w-C'(\alpha)]. 
\end{equation}
If $w_x(a,\alpha)>0$ (which is true for $\alpha<b-a+\eps$ due to
\eqref{Z1}--\eqref{Z2}), then 
$\varphi_x>0$ and \eqref{eq-w} guarantee $w_x(x,\alpha)>0$ for $x>a$,
hence $w(x,\alpha)\geq w(a,\alpha)\geq1$.
If $w_x(a,\alpha)=0$ (which is true for $\alpha\geq b-a+\eps$ due to
\eqref{Z1}), then $(2\vp w)(a,\alpha)=C'(\alpha)$ and
$$\frac{d}{dx}(2\vp w-C'(\alpha))(a,\alpha)=2\vp_x w+2\vp w_x=2\vp_x w>2p_2>0 $$
hence $w_x(x,\alpha)>0$ for $x>a$ close to $a$,
and \eqref{eq-w} implies $w_x(x,\alpha)>0$ for all $x>a$.
As before, this implies $w(x,\alpha)\geq1$.

If $\alpha<0$, then the choice of $\vp(\cdot,\alpha)$ is symmetric:
The extremal $\vp(\cdot,\alpha)$ solves the Du Bois-Reymond equation in $[a,b]$
and $(\vp(b,\alpha),\vp_x(b,\alpha))=B(\alpha):=(-A_1(-\alpha),A_2(-\alpha))$.

As an alternative to the technical construction of the global field above,
we could also set $(\vp(a,\alpha),\vp_x(a,\alpha))=A(\alpha)$, where
$$A(\alpha)=\begin{cases}
 (u^0(a+\alpha),(u^0)'(a+\alpha)) &\hbox{ if }0\leq\alpha\leq b-a,\\
 (u^0(b)+\alpha-(b-a),2) &\hbox{ if }\alpha>b-a,
\end{cases}$$
and analogously for $\alpha<0$.
Then the field $\vp(\cdot,\alpha)$ is not sufficiently smooth if
$|\alpha|=b-a$, but a simple generalization of Theorem~\ref{thm-field1}
shows that this does not matter. 
In fact, denote $v^\pm:=\vp(\cdot,\pm(b-a))$.
Let $u\in C^1([a,b])$; we want to show $\Phi(u)\geq\Phi(u^0)$.
Approximating $u$ suitably, we may assume that the set
$\{x\in[a,b]:u(x)=v^+(x)\hbox{ or }u(x)=v^-(x)\}$ 
is finite. Set $\tilde u:=\max(v^-,\min(v^+,u))$
and approximate $\tilde u$ by a sequence of $C^1$-functions $u^k$
such that $\hbox{graph}(u^k)\subset\c{P}_1:=\{(x,\vp(x,\alpha)):x\in[a,b],\ |\alpha|\leq b-a\}$
and $u^k\to\tilde u$ in $W^{1,4}$. Then Theorem~\ref{thm-field1}
shows that $\Phi(u^k)\geq\Phi(u^0)$, hence $\Phi(\tilde u)\geq\Phi(u^0)$
due to the continuity of $\Phi$ in $W^{1,4}$.
Let $[x_1,x_2]$ be any maximal interval where $\tilde u=v^+$ (i.e.~$u\geq v^+$)
or $\tilde u=v^-$. Notice that either $x_1=a$ or $u(x_1)=v^\pm(x_1)$,
and either $x_2=b$ or $u(x_2)=v^\pm(x_2)$.
Denote $\Phi_{x_1}^{x_2}(u)=\int_{x_1}^{x_2}f(x,u(x),u'(x))\,dx$.
Then the proof of Theorem~\ref{thm-field1} shows 
$\Phi_{x_1}^{x_2}(u)\geq\Phi_{x_1}^{x_2}(v^+)$ (if $u\geq v^+$ in $[x_1,x_2]$) or 
$\Phi_{x_1}^{x_2}(u)\geq\Phi_{x_1}^{x_2}(v^-)$,
hence $\Phi(u)\geq\Phi(\tilde u)\geq\Phi(u^0)$.
\end{proof}

\section{Appendix} \label{sec-app}

\begin{proof}[Proof of Proposition~\ref{prop-NBC}]
We will consider only the special case $N=1$, $\ID_a=\emptyset$, $\IN_b=\emptyset$,
but the arguments in our proof can also be used in the general case.
 
If $h\in C^1_\c{D}=\{\varphi\in C^1([a,b]):\varphi(b)=0\}$, then
integration by parts yields
\begin{equation} \label{eq-Euler} 
 \begin{aligned} 
0 &=\Phi'(u^0)h=\int_a^b (f^0_p(x)h'(x)+f^0_u(x)h(x))\,dx \\
  &= gh\Big|_a^b + \int_a^b (f^0_p(x)-g(x))h'(x)\,dx,
\end{aligned}
\end{equation}
where $g(x):=\int_a^x f^0_u(\xi)\,d\xi$ is $C^1$.
Considering test functions $h$ with compact support in $(a,b)$,
the Du Bois-Reymond Lemma and \eqref{eq-Euler} yield the existence of 
a constant $C$ such that
$f^0_p(x)=g(x)+C$, hence $f^0_p\in C^1$ and the Euler equation
$\frac{d}{dx}(f^0_p)=f^0_u$ is satisfied.
This equation and the choice of $h$ with $h(a)=1$ in \eqref{eq-Euler} imply
$$  \begin{aligned} 
0 &=\Phi'(u^0)h=\int_a^b (f^0_p(x)h'(x)+f^0_u(x)h(x))\,dx \\
 &= f^0_ph\Big|_a^b +\int_a^b\Bigl(-\frac{d}{dx}(f^0_p(x))+f^0_u(x)\Bigr)h(x)\,dx =
-f^0_p(a), 
\end{aligned} $$ 
which concludes the proof of the first part.
If $f_p\in C^1$ and $f^0_{pp}\geq c^0>0$, then the function
$F(x,p):=f_p(x,u^0(x),p)-g(x)-C$ is $C^1$, 
$F(x,(u^0)'(x))=0$, $F_p(x,(u^0)'(x))>0$,
hence the Implicit Function Theorem implies $u^0\in C^2$.
\end{proof}

\begin{proof}[Proof of Proposition~\ref{prop-ws}]
The proof is based on an idea due to \cite{BN}.

Let $u^0\in C^1$ be a weak minimizer of $\Phi$ in $u^0+C^1_\c{D}$.
Assume first that there exist
$v^k\in W^{1,2}_\c{D}$, $k=1,2,\dots$, such that $r_k:=\|v^k\|_{1,2}\to0$
and $\Phi(u^0+v^k)<\Phi(u^0)$. Since $\Phi$ is weakly lower semicontinuous in
$W^{1,2}$,
there exists a minimizer $u^k$ of $\Phi$
in the set $\{u\in u^0+W^{1,2}_\c{D}:\|u-u^0\|_{1,2}\leq r_k\}$, hence
$\Phi(u^k)\leq\Phi(u^0+v^k)<\Phi(u^0)$.
Set $\Theta(u):=\|u-u^0\|_{1,2}^2$. Then there exists a Lagrange multiplier
$\lambda_k$ such that $\Phi'(u^k)h=\lambda_k\Theta'(u^k)h$
for any $h\in W^{1,2}_\c{D}$
(where the derivatives are considered in $W^{1,2}$).
Since $\Phi'(u^k)(u^k-u^0)\leq0$, we have $\lambda_k\leq0$.
Standard theory 
implies that $u^0,u^k\in C^2$ solve the Euler equation
$$ 2(1-\lambda_k)(u^k)''=g'(u^k)-2\lambda_k((u^0)''+u^k-u^0), $$
which shows that the sequence $u^k$ is bounded in $C^2$.
Since $u^k\to u^0$ in $W^{1,2}$, the boundedness in $C^2$ implies $u^k\to u^0$ in $C^1$
which contradicts the fact, that $u^0$ is a weak minimizer.
Consequently, $u^0$ is a local minimizer in $u^0+W^{1,2}_\c{D}$.

Next assume that there exist $v^k\in C^1_\c{D}$ such that $\|v^k\|_C\to0$
and $\Phi(u^0+v_k)<\Phi(u^0)$. 
Since $\Phi'(u^0)h=\int_a^b(2((u^0)'-K)h'+g'(u^0)h)\,dx=0$ for $h\in C^1_\c{D}$,
we have
$$ \begin{aligned}
0 <\Phi(u^0)-\Phi(u^0+v^k) &= \int_a^b[((u^0)'-K)^2-((u^0)'+(v^k)'-K)^2]\,dx + o(1) \\
&=-\int_a^b(v^k)'[(v^k)'+2((u^0)'-K)]\,dx +o(1) \\
&=-\|v^k\|_{1,2}^2+\int_a^b g'(u^0)v^k\,dx + o(1) = -\|v^k\|_{1,2}^2+o(1),
\end{aligned}
$$
hence $v^k\to 0$ in $W^{1,2}$, which yields a contradiction.
Consequently, $u^0$ is a strong minimizer.
\end{proof}

\begin{proof}[Proof of Proposition~\ref{prop-psi}]
Assume that $\Psi(h)\geq c\|h\|_{1,2}^2$ for some $c>0$ and all $h\in W^{1,2}_{\c{D}}$
and recall that $\Psi(h)=\Phi''(u^0)(h,h)$ if $h\in C^1$. 
If $u^1$ is close $u^0$ in $C^1$
and $\Psi^1$ denotes the functional $\Psi$ with $u^0$ replaced by $u^1$,
then one can easily check that $\Psi^1(h)=\Phi''(u^1)(h,h)\geq\frac c2\|h\|_{1,2}^2$ 
for $h\in C^1_{\c{D}}$, and the Mean Value Theorem implies
the existence of $\theta\in(0,1)$ such that
$$\Phi(u^0+h)-\Phi(u^0)=\frac12\Phi''(u^0+\theta h)(h,h)\geq \frac c4\|h\|_{1,2}^2$$
whenever $h\in C^1_{\c{D}}$ is small enough. Consequently, $u^0$ is a strict weak minimizer
in $u^0+C^1_\c{D}$.
 
If $\Psi(h)<0$ for some $h\in W^{1,2}_{\c{D}}$, then
the density of $C^1_{\c{D}}$ in $W^{1,2}_{\c{D}}$
and the continuity of $\Psi$ in $W^{1,2}_{\c{D}}$
guarantee the existence of $\tilde h\in C^1_{\c{D}}$ such that
$0>\Psi(\tilde h)=\Phi''(u^0)(\tilde h,\tilde h)$, which shows that $u^0$ is not a
weak minimizer $u^0+C^1_\c{D}$.
\end{proof}

\begin{proof}[Proof of Theorem~\ref{thm-field}(iii)]
First assume that $\IN_a=\emptyset$. 
If $\IN_b=\emptyset$, then the assertion is well known (see \cite{GF} or \cite{GH}, for
example), hence we may assume $\IN_b\ne\emptyset$. Our assumptions imply
$D\ne 0$ in $(a,b]$
and $\c{B}h(b)\cdot h(b)>0$ for any $h\in H_{\c{D},b}\setminus\{0\}$. 
We may also assume that $f$ is defined and of class $C^3$ 
in an open neighbourhood of  $\{(x,u^0(x),(u^0)'(x)):x\in[a,b]\}$ in $\R\times\R^N\times\R^N$
(see \cite{Bed} for a detailed proof if $N=1$). Consequently, there exists $\eps>0$ small such that
$u^0$ can be extended (as an extremal) for $x\in[a-\eps,a]$, $f^0$ satisfies \eqref{f-conv} in $[a-\eps,b]$,
and the solutions $h^{(k)}$, $k=1,2,\dots,N$ of the Jacobi equation in $[a-\eps,b]$
with initial conditions $h^{(k)}(a-\eps)=0$, $(h_i^{(k)})'(a-\eps)=\delta_{ik}$, satisfy $D>0$ in $(a-\eps,b]$
and $\c{B}h(b)\cdot h(b)>0$ for any $h\in H_{\c{D},b}\setminus\{0\}$ 
due to the continuous dependence of solutions of ODEs on initial values. 
Let $\vp(\cdot,\alpha)$ be the extremal satifying the initial conditions
$\vp(a-\eps,\alpha)=u^0(a-\eps)$, $\vp_x(a,\alpha)=(u^0)'(a-\eps)+\alpha$.
The arguments in \cite{GF, GH} guarantee that such extremals define a field of extremals for $u^0$
(in $[a,b]$) satisfying \eqref{self-adjoint}.
Condition \eqref{field-a} is empty, hence we only have to show that \eqref{field-b} is true.
Thus assume that $v-u^0(b)\in\R^N_{\c{D},b}\cap B_\eps\setminus\{0\}$.
We have $v=\vp(b,\alpha)$ for some $\alpha$ small.
Set $h^\alpha:=\sum_k\alpha_kh^{(k)}$.
If $i\in\ID_b$, then 
$0=\vp_i(b,\alpha)-u^0_i(b)=h^\alpha_i(b)+o(\alpha)$, hence $h^\alpha=h^{\tilde\alpha}+o(\alpha)$
for some $h^{\tilde\alpha}\in H_{\c{D},b}\setminus\{0\}$ and $\tilde\alpha=\alpha+o(\alpha)$.
Since our assumptions imply 
$\c{B}h^{\tilde\alpha}(b)\cdot h^{\tilde\alpha}(b)=\sum_{i\in\IN_b}\c{B}_ih^{\tilde\alpha}(b)h_i^{\tilde\alpha}(b)>0$,
we also have
$$\begin{aligned}
f_p(b,v,\psi(b,v))\cdot(v-u^0(b)) 
&=\sum_{i\in\IN_b}f_{p_i}(b,\vp(b,\alpha),\vp_x(b,\alpha))(\vp_i(b,\alpha)-u^0_i(b)) \\
&=\sum_{i\in\IN_b}\bigl(\c{B}_ih^\alpha(b)+o(\alpha)\bigr)\bigl(h^\alpha_i(b)+o(\alpha)\bigr) \\
&=\sum_{i\in\IN_b}\bigl(\c{B}_ih^{\tilde\alpha}(b)+o(\tilde\alpha)\bigr)\bigl(h^{\tilde\alpha}_i(b)+o(\tilde\alpha)\bigr)>0.  
\end{aligned}$$

Next assume $\ID_a=\emptyset$.
Since our proof in this case uses similar arguments as in the case $\IN_a=\emptyset$
(and a very detailed proof in the case $N=1$ can be found in \cite{Bed}),
we will be brief. 
Given $\alpha\in\R^N$ small and $v=v(\alpha):=u^0(a)+\alpha$, the Implicit Function Theorem implies
the existence of a unique $w=w(\alpha)\in\R^N$ close to $(u^0)'(a)$ such $f_p(a,v(\alpha),w(\alpha))=0$.
Let $\vp(\cdot,\alpha)$ be the extremal satifying the initial conditions
$\vp(a,\alpha)=v(\alpha)$, $\vp_x(a,\alpha)=w(\alpha)$.
We claim that such extremals $\vp(\cdot,\alpha)$ define the required field.
In fact, the function $P$ in Definition~\ref{def-field} is a $C^1$-diffeomorphism and $\vp_x\in C^1$
due to the differentiablity of solutions of ODEs on initial values    
and the fact that $h^{(k)}:=\frac{\partial\vp}{\partial\alpha_k}(\cdot,0)$, $k=1,\dots,N$, are linearly
independent solutions of the Jacobi equation $\c{A}h=0$ satisfying the initial conditions $\c{B}h(a)=0$,
hence $\det(h^{(1)},\dots,h^{(N)})\ne0$ in $[a,b]$ due to our assumptions.
Properties \eqref{self-adjoint} and \eqref{field-a} follow from $f_{p}(a,v,\psi(a,v))=0$
and the proof of \eqref{field-b} is the same as in the case $\IN_a=\emptyset$.

Finally assume \eqref{special}.
Let $h^{(1)},\dots,h^{(N)}$ be solutions
of the Jacobi equation $\c{A}h=0$ in $[a,b]$
satisfying the initial conditions
$$\begin{aligned}
h^{(k)}_i(a)&=\eta\delta_{ik}&\quad&\hbox{for $k\in \ID_a$, $i\in I$,}\\
h^{(k)}_i(a)&=\delta_{ik}&\quad&\hbox{for $k\in \IN_a$, $i\in I$,}
\end{aligned}\qquad
\begin{aligned}
(h^{(k)}_i)'(a)&=\delta_{ik}&\quad&\hbox{for $k\in I$, $i\in \ID_a$,}\\
\c{B}_ih^{(k)}(a)&=0&\quad&\hbox{for $k\in I$, $i\in \IN_a$,}
\end{aligned} $$
where $\eta\in[0,1]$.
If $\zeta\geq0$ is small, then 
$$ \begin{aligned}
  h^{(k)}_i(a+\zeta)&=(\eta+\zeta)\delta_{ik}+o(\zeta) &\quad&\hbox{if }k,i\in\ID_a,\\
  h^{(k)}_i(a+\zeta)&=\delta_{ik}+O(\zeta) &\quad&\hbox{otherwise},
\end{aligned}$$
hence $D(x):=\det(h^{(1)}(x),\dots,h^{(N)}(x))>0$ for $x\in[a,a+\zeta]$
and $\eta\in(0,1]$. 
If $\eta=0$, then
our assumptions imply $D(x)>0$ for $x\in[a+\zeta,b]$ and
$\c{B}h(b)\cdot h(b)>0$ for any $h:=\sum_k\beta_k h^{(k)}$ satisfying
$h_i(b)=0$ for $i\in\ID_b$ and $h\not\equiv0$.
Those properties remain true for $\eta>0$ small and we fix such $\eta>0$.
Set
$v_i(\alpha)=u^0_i(a)+\eta\alpha_i$ if $i\in\ID_a$, 
$v_i(\alpha)=u^0_i(a)+\alpha_i$ if $i\in\IN_a$
and $w_i(\alpha)=(u^0_i)'(a)+\alpha_i$ if $i\in\ID_a$.
The Implicit Function Theorem gurantees that there exist unique $w_i(\alpha)$ for $i\in\IN_a$
(close to $(u^0_i)'(a)$)
such that $f_{p_i}(a,v(\alpha),w(\alpha))=0$ for $i\in\IN_a$ and $\alpha$ small.
Let $\vp(\cdot,\alpha)$ be extremals satisfying the initial conditions
$\vp(a,\alpha)=v(\alpha)$, $\vp_x(a,\alpha)=w(\alpha)$.
Then $\vp_{\alpha_k}(a,0)=h^{(k)}(a)$
and $\vp_{x\alpha_k}(a,0)=(h^{(k)})'(a)$, 
which shows that these extremals define a field of extremals for $\alpha$ small.
The same arguments as above guarantee that
properties \eqref{field-a},\eqref{field-b} are satisfied.
Let us show that \eqref{self-adjoint} is true.
If $i,j\in\IN_a$, then this follows from $f_{p_i}(a,v,\psi(a,v))=f_{p_j}(a,v,\psi(a,v))=0$.
Let $i\in\ID_a$. If $j\in\IN_a$, then the left-hand side in \eqref{self-adjoint} is zero
due to $f_{p_iu_j}=f_{p_ip_j}=0$.
If $j\in\ID_a$, then that left-hand side equals
$f_{p_iu_j}(a,v,\psi(a,v))+\sum_{k\in I}f_{p_ip_k}(a,v,\psi(a,v))\psi_{k,v_j}(a,v)$.
Since $f_{p_iu_j}=f_{p_ju_i}$, $f_{p_ip_k}(a,v,\psi(a,v))=0$ for $k\in\IN_a$
and $\psi_{k,v_j}(a,v)=\frac1\eta\delta_{kj}$ if $k\in\ID_a$,
we see that that left-hand side equals to the right-hand side.
\end{proof}

\begin{proof}[Proof of Proposition~\ref{prop-HII}]
If $w=(w_1,\dots,w_N)$ depends on $\theta$, then we denote $w_{i,\theta}:=\frac{\partial w_i}{\partial\theta}$.
By differentiating the identity $\vp_x(x,\alpha)=\psi\bigl(x,\vp(x,\alpha)\bigr)$ we obtain
$$ \vp_{j,xx}=\psi_{j,x}+\sum_k\psi_{j,v_k}\vp_{k,x}=\psi_{j,x}+\sum_k\psi_{j,v_k}\psi_k.$$
If we substitute this relation into the Euler equations 
$$ \sum_j (f_{p_ip_j}\vp_{j,xx}+f_{p_iu_j}\vp_{j,x})+f_{p_ix}-f_{u_i}=0,$$
(where the arguments of the derivatives of $f$ and $\vp$ are $\bigl(x,\vp(x,\alpha),\vp_x(x,\alpha)\bigr)$ and $(x,\alpha)$, respectively),
then we obtain
\begin{equation} \label{Hilb1}
 \sum_j (f_{p_ip_j}(\psi_{j,x}+\sum_k\psi_{j,v_k}\psi_k)+f_{p_iu_j}\psi_j)+f_{p_ix}-f_{u_i}=0,
\end{equation}
where the arguments of the derivatives of $f$ and $\psi$ are $\bigl(x,v,\psi(x,v)\bigr)$ and $(x,v)$, respectively.
For $(x,v)\in\c{P}$ we set
\begin{equation} \label{M1M2}
\begin{aligned}
 V(x,v) &:= f\bigl(x,v,\psi(x,v)\bigr)-f_p\bigl(x,v,\psi(x,v)\bigr)\cdot\psi(x,v), \\
 W(x,v) &:= f_p\bigl(x,v,\psi(x,v)\bigr).
\end{aligned}
\end{equation}
We claim that
\begin{equation} \label{self-adjoint-x}
(W_{i,v_j}-W_{j,v_i})(x,v)
=\frac{\partial f_{p_i}(x,v,\psi(x,v))}{\partial v_j}
-\frac{\partial f_{p_j}(x,v,\psi(x,v))}{\partial v_i}=0, \quad i,j\in I. 
\end{equation}
In fact, if $f$ and $\vp$ are of class $C^3$, then setting $v=\vp(x,\alpha)$
and $\psi(x,v)=\vp_x(x,\alpha)$ in \eqref{self-adjoint-x},
the Euler equations imply that
the $d/dx$-derivative of the resulting expression vanishes,
hence the conclusion follows from \eqref{self-adjoint}.
Such argument can also be used without the additional smoothness assumptions on $f,\vp$,
see the proof of \cite[Proposition 6.1.1.4]{GH}.

Now \eqref{self-adjoint-x} and \eqref{Hilb1} imply 
$V_v=W_x$.
This fact and \eqref{self-adjoint-x} guarantee the existence of
$S\in C^2(\c{P})$ such that $S_x=V$ and $S_v=W$.
Finally,
$$ \begin{aligned}
I(v)&=\int_a^b(V+W\cdot v')\,dx=\int_a^b(S_x+S_v\cdot v')\,dx= 
  \int_a^b\frac{d}{dx}S\bigl(x,v(x)\bigr)\,dx \\
&= S\bigl(b,v(b)\bigr)-S\bigl(a,v(a)\bigr).\end{aligned}$$
\end{proof}

\begin{remark}\label{rem-ZZ}\rm
Necessary and sufficient conditions for weak minimizers in
\cite{Zei, Zez93} are formulated in terms of (semi-)coupled points
and seem to be more complicated than our conditions.
In order to compare them, let us consider the scalar case with
variable endpoints (i.e.~$\ID_a=\ID_b=\emptyset$),
and let $h$ be the solution of the Jacobi equation
satisfying the initial conditions $h(a)=1$, $\c{B}h(a)=0$.
Let us also denote $Q:=f^0_{uu}$.
Then our sufficient condition for a weak minimizer in Theorem~\ref{thm-Jacobi1}
is equivalent to
\begin{equation} \label{our}
h(y)\ne0\ \hbox{ for }\ y\in(a,b]\quad \hbox{ and }\quad \c{B}h(b)>0,
\end{equation}
while the sufficient condition for a weak minimizer in \cite{Zei, Zez93}
is equivalent to
\begin{equation} \label{their}
-\c{B}h(y)\ne\Bigl(\int_y^b Q\Bigr)h(y)\ \hbox{ for }\ y\in(a,b]\quad \hbox{ and }\quad \int_a^b Q>0.
\end{equation}
The proofs of the sufficiency guarantee that \eqref{our} is equivalent to \eqref{their}.
Let us show this equivalence directly: For simplicity, consider just
Lagrangians of the form $2f(x,u,p)=p^2+Q(x)u^2$. Then $\c{B}h=h'$ and the Jacobi equation has the form
$h''=Qh$. Let $h$ be the solution of this equation with initial conditions $h(a)=1$, $h'(a)=0$.

First assume that \eqref{our} is true.
Then integration by parts yields
\begin{equation} \label{Q1}
\int_a^b Q=\int_a^b\frac{h''}h=\frac{h'}h\Big|_a^b+\int_a^b\frac{(h')^2}{h^2}>0.
\end{equation}
Assume to the contrary that $-h'(y)=(\int_y^b Q)h(y)$ for some $y\in(a,b]$.
Then
\begin{equation} \label{Q2}
 -\int_y^b Q =\frac{h'(y)}{h(y)}= \frac{h'}h\Big|_a^y=\int_a^y\Bigl(\frac{h''}h-\frac{(h')^2}{h^2}\Bigr)
 =\int_a^y\Bigl(Q-\frac{(h')^2}{h^2}\Bigr).
\end{equation}
Now \eqref{Q2} and \eqref{Q1} imply
$$\int_a^b Q=\int_a^y\frac{(h')^2}{h^2}<\frac{h'}h\Big|_a^b+\int_a^b\frac{(h')^2}{h^2}=\int_a^b Q,$$
which yields a contradiction.

Next assume that \eqref{our} fails, i.e. either $h(y)=0$ for some $y\in(a,b]$
or $h'(b)\leq0$, and assume also to the contrary \eqref{their} is true.
If $h(y)=0$ for some $y\in(a,b]$ and $h>0$ on $[a,y]$, then $h'(y)<0$, hence 
$$ \begin{aligned}
 -h'(a) &=0<\Bigl(\int_a^b Q\Bigr)h(a), \\
 -h'(y) &>0=\Bigl(\int_y^b Q\Bigr)h(y),
\end{aligned}$$
so that there exists $z\in(a,y)$ such that $-h'(z)=\bigl(\int_z^b Q\bigr)h(z)$,
which yields a contradiction.
If $h>0$ and $h'(b)\leq0$, then
$$ \begin{aligned}
 -h'(a) &=0<\Bigl(\int_a^b Q\Bigr)h(a), \\
 -h'(b) &\geq 0=\Bigl(\int_b^b Q\Bigr)h(b),
\end{aligned}$$
so that there exists $z\in(a,b]$ such that $-h'(z)=\bigl(\int_z^b Q\bigr)h(z)$
and we arrive at contradiction again.

The proof above shows that if $y_1$ is the first (= smallest) zero of $h$, 
then the smallest solution $z_1$ of the equation
$-h'(z)=\bigl(\int_z^b Q\bigr)h(z)$ satisfies $z_1<y_1$.
The inequality $z_1\leq y_1$ also follows from the proof of Theorem~\ref{thm-Jacobi}
and the corresponding proof in \cite{Zez93}.
In fact, those proofs show that $y_1$ and $z_1$ correspond to the zeroes
of the continuous nonincreasing functions 
$\lambda_1(y)=\inf_{S_y}\Psi$
and
$\tilde\lambda_1(z)=\inf_{\tilde S_z}\Psi$, respectively,
where $S_y$ is the unit sphere in $X_y$ (see \eqref{lambda1}) and
$\tilde S_z$ is the unit sphere in
$\tilde X_z=\{h\in W^{1,2}([a,b]): h(x)=h(z)\hbox{ for }x\geq z\}$.
Since $X_y\subset \tilde X_y$ and the norm in $X_y$ is equivalent to the norm in $W^{1,2}$, 
we have 
$\tilde\lambda_1\leq \max\{C\lambda_1,0\}$.
\qed
\end{remark}

The following proposition is motivated by \cite{MR} and Section~\ref{sec-twist}.
Given $u^0\in C^1([a,b],\R^N)$, we will use the following notation (cf.~\eqref{M}):
$$ \begin{aligned}
 \c{M} &:=u^0+C^1_\c{D}=\{u\in C^1([a,b]): (u_i-u^0_i)(x)=0 \hbox{ for }i\in\ID_x\hbox{ and }x\in\{a,b\}\}, \\ 
 \c{M}_\c{N} &:= \{u\in\c{M}: u_i'(x)=0\hbox{ for }i\in\IN_x\hbox{ and }x\in\{a,b\}\}.
\end{aligned}$$

\begin{proposition} \label{prop-Neumann}
Let $f\in C^1$ and let $u^0$ be a weak minimizer of $\Phi$ in $\c{M}_\c{N}$.
Then $u^0$ is a weak minimizer in $\c{M}$.
Conversely, if $u$ is a weak minimizer in $\c{M}$
and $u^0\in\c{M}_\c{N}$, then
$u^0$ is a weak minimizer in $\c{M}_\c{N}$.
\end{proposition}

\begin{proof}
For simplicity,
we will prove the assertion only in the special case $N=1$, $\ID_a=\emptyset$, $\IN_b=\emptyset$,
but it will be clear from the proof that our arguments can also be used in the general case.

Hence assume first that $u^0$ is a weak minimizer of $\Phi$ in 
$$\c{M}_\c{N}=\{u\in C^1([a,b]): (u-u^0)(b)=0,\ u'(a)=0\}.$$
Then there exists $\eps>0$ such that $u^0$ is a (global) minimizer of $\Phi$ in the set 
$$\c{M}_\c{N}^\eps:=\{u\in\c{M}_\c{N}:\|u-u^0\|_{C^1}<\eps\}.$$
We will show that $u^0$ is a (global) minimizer in the set $\c{M}^{\eps/4}$, where
$$\c{M}^\eps:=\{u\in\c{M}:\|u-u^0\|_{C^1}<\eps\},$$
hence $u^0$ is a weak minimizer of $\Phi$ in 
$\c{M}=\{u\in C^1([a,b]): (u-u^0)(b)=0\}$.
 
Fix $u\in\c{M}^{\eps/4}$.
Since $(u^0)'(a)=0$,
given $k\in\N$, there exists
$\delta_k\in(0,1/k)$ such that 
$$|(u^0)'(x)|<1/k\ \hbox{ for }\ x\in J_k:=[a,a+\delta_k].$$
Since $\|u-u^0\|_{C^1}<\eps/4$, we also have 
$|u'(x)|<\eps/4+1/k$ for $x\in J_k$.
Consequently, we can modify the function $u$ in $J_k$
such that the modified function $u^k\in C^1([a,b])$ satisfies $u^k=u$ on $[a+\delta_k,b]$,
$(u^k)'(a)=0$ and $|(u^k)'(x)|<\eps/4+1/k$ for $x\in J_k$
(for example, we can choose $(u^k)'(x)=u'(\delta_k)(x-a)/(\delta_k-a)$ for $x\in J_k$).
Then 
$$|(u^k)'-(u^0)'|\leq|(u^k)'|+|(u^0)'|<\eps/4+2/k \ \hbox{ on }\ J_k$$ 
and the Mean Value Theorem implies
$$|u^k-u^0|\leq|u^k-u|+|u-u^0|<\max_{J_k}|(u^k-u)'|\delta_k+\eps/4<(\eps/2+2/k)/k+\eps/4\ \hbox{ on }\ J_k,$$
hence $u^k\in \c{M}_\c{N}^\eps$ for $k$ large, which implies $\Phi(u^k)\geq\Phi(u^0)$.
Since $\Phi(u^k)\to\Phi(u)$,  we have $\Phi(u)\geq\Phi(u^0)$. 

The converse assertion is trivial.
\end{proof}

\begin{remark} \label{rem-Neumann} \rm
In \cite[Propositions~5 and~6]{MR}
the authors consider the function $u^0$ and the functional $\Phi$
from our Section~\ref{sec-twist}, and they
provide conditions guaranteeing that $u^0$
is a weak minimizer subject to the Neumann boundary conditions
for some of its components (see \eqref{MR-N} and \eqref{MR-ND} above).  
Proposition~\ref{prop-Neumann} shows that the Neumann boundary conditions
do not play any role in such assertions, i.e. 
$u^0$ remains a weak minimizer if we replace ``the Neumann boundary conditions''
with ``no boundary conditions''.
Consequently (see Proposition~\ref{prop-NBC}), $u^0$
then has to satisfy the corresponding natural boundary conditions
(instead of the Neumann boundary conditions).
The Neumann boundary conditions are different from
the natural boundary coditions in general,
but the first two components of the function $u^0$ in Section~\ref{sec-twist}
satisfy both the Neumann and the natural boundary conditions.
\qed
\end{remark}

\section*{Acknowledgements}
This work was supported in part by VEGA grant 1/0245/24.
The author thanks the anonymous referee for many helpful comments.



\begin{thebibliography}{16}
%
\bibitem{B20}
Batista M.,
{\em On stability of non-inflectional elastica},
C. R. M\'ecanique {\bf 348} (2020), 137--148.
%
\bibitem{Bed}
Bedn\'arik D.,
{\em Weak and strong minimizers of variational integrals},
Master thesis, Comenius University, Bratislava 2022
({\tt http://www.iam.fmph.uniba.sk/institute/quittner/bednarik.pdf}). 
%
\bibitem{BB}
Borum A. and Bretl T.,
{\em When is a helix stable?},
Phys. Rev. Letters {\bf 125} (2020), \#088001.
%
\bibitem{BN}
Brezis H. and Nirenberg L.,
{\em $H^1$ versus $C^1$ local minimizers},
C. R. Acad. Sci. Paris {\bf 317} (1993), 465--472.
%
\bibitem{Ces}
Cesari L.,
{\em Optimization - theory and applications},
Springer, Berlin 1983.
%
%
\bibitem{CRS}
Cicalese M., Ruf M. and Solombrino F.,
{\em On global and local minimizers of prestrained thin elastic rods},
Calc. Var. {\bf 56} (2017), \#115.
%
\bibitem{GF}
Gelfand I. M. and Fomin S. V.,
{\em Calculus of Variations}, Prentice Hall, Englewood Cliffs, N.J., 1963.
%
\bibitem{GH}
Giaquinta M. and Hildebrandt S.,
{\em Calculus of Variations I}, Springer, Berlin - Heidelberg, 2004.
%
\bibitem{Jin}
Jin M.,
{\em Stability to discontinuous perturbations for one inflexion Euler elasticas
with one end fixed and the other clamped in rotation},
Eur. J. Mech. A Solids {\bf 81} (2020), \#103954.
%
\bibitem{LG}
Lessinnes T. and Goriely A.,
{\em Geometric conditions for the positive definiteness of the second variation
in one-dimensional problems},
Nonlinearity {\bf 30} (2017), 2023--2062.
%
\bibitem{MR}
Majumdar A. and Raisch A.,
{\em Stability of twisted rods, helices and buckling solutions in three dimensions},
Nonlinearity {\bf 27} (2014), 2841--2867.
%
%
\bibitem{Man}
Manning R. S.,
{\em Conjugate points revisited and Neumann-Neumann problems},
SIAM Review {\bf 51} (2009), 193--212.
%
\bibitem{Olver}
Olver P. J.,
{\em Boundary conditions and null Lagrangians in the calculus of variations and elasticity},
J. Elast., to appear 
({\tt https://doi.org/10.1007/s10659-022-09912-5}).
%
\bibitem{OP}
O'Reilly O. M. and Peters D. M.,
{\em On stability analyses of three classical buckling problems for the elastic strut}, 
J. Elast. {\bf 105} (2011), 117--136. 
%
\bibitem{Zei}
Zeidan V.,
{\em Sufficient conditions for variational problems with variable endpoints: Coupled points},
App. Math. Optim. {\bf 27} (1993), 191--209.
%
\bibitem{Zez93}
Zezza P.,
{\em Jacobi condition for elliptic forms in Hilbert spaces},
J. Optim. Theory Appl. {\bf 76} (1993), 357--380.
%
\bibitem{Zez97}
Zezza P.,
{\em Errata corrige: ``Jacobi condition for elliptic forms in Hilbert spaces''},
J. Optim. Theory Appl. {\bf 95} (1997), 741--742.  
%
\end{thebibliography}
\end{document}